\documentclass[12pt]{article}

\title{Torus knot filtered embedded contact homology of the tight contact 3-sphere}

\author{Jo Nelson and Morgan Weiler}
\date{}
\usepackage{subcaption}
\usepackage[normalem]{ulem}
\usepackage{amssymb}
\usepackage{latexsym}
\usepackage{amsmath}
\usepackage{amsthm}
\usepackage{eucal}
\usepackage[scr=rsfso]{mathalfa}
\DeclareFontFamily{OT1}{pzc}{}
\DeclareFontShape{OT1}{pzc}{m}{it}{<-> s * [1.10] pzcmi7t}{}
\DeclareMathAlphabet{\mathpzc}{OT1}{pzc}{m}{it}

\usepackage{graphicx, graphbox}
\usepackage[abs]{overpic}

\usepackage[all,cmtip]{xy}

\usepackage[margin=1in]{geometry}

\usepackage[matha,  mathx]{mathabx}

\usepackage{wasysym}

\usepackage{caption}
\usepackage{subcaption}
\usepackage{multirow}

\usepackage{enumerate}
\usepackage[usenames,dvipsnames,svgnames,table]{xcolor}

\usepackage{hyperref}
\hypersetup{colorlinks}

\definecolor{indigo}{RGB}{75,0,150}
\definecolor{brightpurple}{RGB}{102,0,153}
\definecolor{fuchsia}{RGB}{180,51,180}
\definecolor{jolightpurple}{RGB}{188,171,240}
\definecolor{blb}{RGB}{153, 102, 51}
\definecolor{bdb}{RGB}{77, 38, 0}
\definecolor{bg}{RGB}{85, 85, 94}

\hypersetup{colorlinks,
linkcolor=indigo,
filecolor=brightpurple,
urlcolor=brightpurple,
citecolor=fuchsia}

\normalbaselines
\newcommand{\tb}{\textcolor{Blue}}

\usepackage{tikz}

\usepackage[abs]{overpic}

\newcommand{\mc}[1]{{\mathcal #1}}

\numberwithin{equation}{section}
\numberwithin{figure}{section}

\newtheorem{theorem}{Theorem}[section]
\newtheorem{proposition}[theorem]{Proposition}
\newtheorem{corollary}[theorem]{Corollary}
\newtheorem{lemma}[theorem]{Lemma}
\newtheorem{lemma-definition}[theorem]{Lemma-Definition}

\theoremstyle{definition}
\newtheorem{definition}[theorem]{Definition}
\newtheorem{remark}[theorem]{Remark}

\newtheorem{example}[theorem]{Example}

\newtheorem{notation}[theorem]{Notation}

\renewcommand{\frak}{\mathfrak}

\newcommand{\C}{{\mathbb C}}
\newcommand{\CP}{{\mathbb C}{\mathbb P}}
\newcommand{\Q}{{\mathbb Q}}
\newcommand{\R}{{\mathbb R}}
\newcommand{\N}{{\mathbb N}}
\newcommand{\Z}{{\mathbb Z}}

\newcommand{\cpq}{\mathbb{CP}^1_{2,q}}

\newcommand{\ve}{\varepsilon}

\newcommand{\fp}{\mathfrak{p}}

\newcommand{\D}{\mathbf{D}}

\newcommand{\cO}{{\Sigma}}

\newcommand{\op}{\operatorname}

\newcommand{\M}{\mc{M}}

\newcommand{\CZ}{\op{CZ}}

\newcommand{\bpm}{\begin{pmatrix}}
\newcommand{\epm}{\end{pmatrix}}

\newcommand{\fb}{\mathcal{F}_{b} }
\newcommand{\af}{\substack{\mathcal{A} < L \\ \mathcal{F}_{b}  \leq K}}
\newcommand{\afo}{\substack{\mathcal{A} < L_0 \\ \mathcal{F}_{b}  \leq K}}
\newcommand{\afe}{\substack{\mathcal{A} < L(\varepsilon) \\ \mathcal{F}_{b}  \leq K}}
\newcommand{\afepp}{\substack{\mathcal{A} < L(\varepsilon'') \\ \mathcal{F}_{b}  \leq K}}
\newcommand{\afep}{\substack{\mathcal{A} < L(\varepsilon') \\ \mathcal{F}_{b}  \leq K}}
\newcommand{\afp}{\substack{\mathcal{A} < L' \\ \mathcal{F}_{b}  \leq K}}

\newcommand{\afplus}{\substack{\mathcal{A} < L^+ \\ \mathcal{F}_{b}  \leq K}}
\newcommand{\afminus}{\substack{\mathcal{A} < L^- \\ \mathcal{F}_{b}  \leq K}}
\newcommand{\afpm}{\substack{\mathcal{A} < L^\pm \\ \mathcal{F}_{b}  \leq K}}
\newcommand{\id}{\op{id}}

\newcommand{\A}{\mathcal{A}}
\newcommand{\cur}{\mathcal{C}}

\newcommand{\rot}{\op{rot}}

\begin{document}

\maketitle

\begin{abstract}
Knot filtered embedded contact homology was first introduced by Hutchings in 2015; it has been computed for the standard transverse unknot in irrational ellipsoids by Hutchings and for the Hopf link in lens spaces $L(n,n-1)$ via a quotient by Weiler.   While toric constructions can be used to understand the ECH chain complexes of many contact forms adapted to open books with binding the unknot and Hopf link, they do not readily adapt to general torus knots and links.  In this paper, we generalize the definition and invariance of knot filtered embedded contact homology to allow for degenerate knots with rational rotation numbers.  We then develop new methods for understanding the embedded contact homology chain complex of positive torus knotted fibrations of the standard tight contact 3-sphere in terms of their presentation as open books and as Seifert  {fiber} spaces.  We provide Morse-Bott methods, using a doubly filtered complex and the energy filtered perturbed Seiberg-Witten Floer theory developed by Hutchings and Taubes, and use them to compute the $T(2,q)$ knot filtered embedded contact homology, for $q$ odd and positive.  
\end{abstract}

\tableofcontents

\section{Introduction}

 Knot filtered embedded contact homology is a topological spectral invariant which was first introduced by Hutchings, who computed it for the standard transverse unknot in the irrational  ellipsoid to study the mean action of area preserving disk maps which are rotation near the boundary \cite{HuMAC}.   Knot filtered embedded contact homology was subsequently computed for the Hopf link in the lens spaces $(L(n,n-1), \xi_{std})$ (obtained as a quotient of the irrational ellipsoid) by Weiler and used to study area preserving diffeomorphisms of the closed annulus subject to a boundary condition \cite{weiler, weiler2}.  To understand knot filtered embedded contact homology with respect to the right handed $T(p,q)$ torus knots in $(S^3,\xi_{std})$, we introduce new non-toric techniques, which also elucidate the embedded contact homology (ECH) chain complexes of more general open books and arbitrary Seifert {fiber} spaces.

  We also generalize the definition and invariance of knot filtered embedded contact homology, to encompass knots with rational rotation numbers, and provide Morse-Bott methods to compute it. This involves direct limits of doubly filtered direct systems, which is similar in spirit but more involved than our prior work for prequantization bundles \cite{preech} and utilizes the work of Hutchings and Taubes \cite{cc2}.
  
   Previously, the only well-studied ECH chain complexes were those of toric contact forms, as initiated in \cite{T3}, and prequantization bundles over closed symplectic surfaces \cite{preech}.  Our methods allow us to understand the embedded contact homology of $T(p,q)$ fibrations of the standard tight 3-sphere in terms of their associated open book decompositions and presentation as Seifert fiber spaces.  Knot filtered embedded contact homology is defined with respect to a trivialization induced by the Seifert surface of the knot, which is best understood in terms of the open book decomposition.  The main complication in the setting at hand is the computation of the ECH index, which is more subtle than in the toric or prequantization setting.  This is due in part to the trivializations available for fibers projecting to points with isotropy, as one cannot use the `constant' trivialization, which is available for all fibers of prequantization bundles.  
   
 We begin with the $p=2$ case in this paper and complete the study of all $p$ in the sequel in \cite{calabi}. For general $p$, an alternate family of nondegenerate perturbations is needed, which gives rise to non-vanishing differentials, and requires a more involved adaptation of our Morse-Bott methods previously established for prequantization bundles in \cite{preech}. The results of this paper will be used in the sequel \cite{calabi} to deduce quantitative existence results for Reeb orbits associated to any contact form on $(S^3,\xi_{std})$ admitting a maximal self linking torus knot $T(p,q)$ as an elliptic Reeb orbit with (approximate) rotation number $pq$ and whose volume is at most $\frac{1}{pq}$.

\subsection{Overview of embedded contact homology}\label{ss:overviewECH}
Let $Y$ be a closed three-manifold with a contact form $\lambda$. Let $\xi=\ker(\lambda)$ denote the associated contact structure, and let $R$ denote the associated Reeb vector field, which is uniquely determined by 
\[
\lambda(R)=1, \ \ \ d\lambda(R, \cdot)=0.
\]
A {\em Reeb orbit\/} is a map $\gamma:\R/T\Z\to Y$ for some $T>0$ such that $\gamma'(t)=R(\gamma(t))$, modulo reparametrization.  
 
A Reeb orbit is said to be \textit{embedded} whenever this map is  injective. For a Reeb orbit as above, the linearized Reeb flow for time $T$ defines a symplectic linear map
\begin{equation}
\label{eqn:lrt}
P_\gamma{(T)} :(\xi_{\gamma(0)},d\lambda) \longrightarrow (\xi_{\gamma(0)},d\lambda).
\end{equation}
The Reeb orbit $\gamma$ is {\em nondegenerate\/} if $P_\gamma$ does not have $1$ as an eigenvalue.  The contact form $\lambda$ is nondegenerate if all its Reeb orbits are nondegenerate; nondegenerate contact forms form a comeager subset of all contact forms. 

  A nondegenerate Reeb orbit $\gamma$ is {\em elliptic\/} if $P_\gamma$ has eigenvalues on the unit circle and {\em  hyperbolic\/} if $P_\gamma$ has  real eigenvalues. If $\tau$ is a homotopy class of trivializations of $\xi|_\gamma$, then the {\em Conley-Zehnder index\/} $\CZ_\tau(\gamma)\in\Z$ is defined in terms of the induced path of symplectic linear matrices.  The parity of the Conley-Zehnder index does not depend on the choice of trivialization $\tau$.  If $\gamma$ is an embedded Reeb orbit, the Conley-Zehnder index is even when $\gamma$ is positive hyperbolic and odd otherwise.

We say that an almost complex structure $J$ on $\R_s \times Y$ is {\em $\lambda$-compatible\/} if
\begin{itemize}
\itemsep-.25em
\item $J(\xi)=\xi$ and $J(\partial_s)=R$; 
\item $J$ rotates the contact planes positively, meaning $d\lambda(v,Jv)>0$ for nonzero $v\in\xi$;  
\item $J$ is invariant under translation of the $\R$ factor.
\end{itemize} 
We consider $J$-holomorphic curves $u: (\dot{\Sigma},j) \to (\R \times Y, J)$, where $(\dot{\Sigma},j)$ is a punctured possibly disconnected Riemann surface, modding out by the usual equivalence relation, namely composition with biholomorphic maps between domains.  If $\gamma$ is a (possibly multiply covered) Reeb orbit, a \emph{positive end} of $u$ at $\gamma$ is a puncture near which $u$ is asymptotic to $\R \times \gamma$ as $s \to \infty$, meaning it admits coordinates of a positive half-cylinder $(\sigma,\tau) \in [0,\infty) \times (\R / T\Z)$ such that 
\[
j(\partial_\sigma) = \tau, \ \ \ \lim_{\sigma \to + \infty} \pi_\R(u(\sigma,\tau)) = + \infty, \ \ \ \lim_{\sigma \to + \infty} \pi_Y(u(\sigma, \cdot)) = \gamma.
\]
A \emph{negative end} is defined analogously with $\sigma \in (-\infty, 0]$ and $s \to -\infty$.  We assume that all punctures are positive or negative ends.  

Embedded contact homology (ECH) is a `symplectic shadow'  of Seiberg-Witten Floer homology due to Hutchings, roughly defined as follows; see also the survey \cite{lecture}. Given a closed 3-manifold $Y$ equipped with a nondegenerate contact form $\lambda$ and a generic $\lambda$-compatible $J$, the \emph{embedded contact homology chain complex} with respect to a fixed class $\Gamma \in H_1(Y,\Z)$ is the $\Z/2$-module\footnote{It is possible to define ECH with integer coefficients as explained in \cite[\S 9]{obg2}, but that is not necessary for the purposes of this paper or its sequel.} $ECC_*(Y,\lambda,\Gamma, J)$.  The chain complex is freely generated by \emph{admissible Reeb currents}, which are finite sets of Reeb currents\footnote{In some literature, Reeb currents are called orbit sets.} $\alpha = \{ (\alpha_i, m_i) \}$, such that 
\begin{itemize}
\itemsep-.25em
\item the $\alpha_i$ are distinct embedded Reeb orbits;
\item the $m_i$ are positive integers and $m_i=1$ whenever $\alpha_i$ is hyperbolic;
\item the total homology class of $\alpha$ is $\sum_i m_i[\alpha_i] = \Gamma$. 
\end{itemize} 
Sometimes we use the multiplicative notation, $\alpha = \prod_i \alpha_i^{m_i}$.

The chain complex and its homology are relatively $\Z/d$ graded by the ECH index $I$ (defined momentarily), where $d$ denotes the divisibility of $c_1(\xi) + 2 \op{PD}(\Gamma) \in H^2(Y;\Z)$ mod torsion.  This means that if $\alpha$ and $\beta$ are two admissible Reeb currents, we can define their index difference $I(\alpha,\beta)$ by choosing an arbitrary $Z\in H_2(Y,\alpha,\beta)$ (the notation indicates 2-chains with boundary on $\alpha-\beta$, modulo boundaries of 3-chains) and setting 
\[
I(\alpha,\beta) = [I(\alpha,\beta,Z)] \in \Z/d,
\]
which is well-defined by the index ambiguity formula \cite[\S 3.3]{Hindex}; see also \S \ref{ss:ECHI}. When the chain complex is nonzero, we can further define an absolute $\Z/d$ grading by picking some generator $\beta$ and declaring its grading to be zero, so that the grading of any other generator $\alpha$ is
\[
|\alpha| = I(\alpha,\beta).
\]
(By the additivity property of the ECH index, the differential decreases this absolute grading by 1.) In particular, if $\Gamma =0$, then the empty set of Reeb orbits is a generator of the chain complex, which depends only on $Y$ and $\xi$.  As a result, $ECH_*(Y,\xi,0)$ has a canonical absolute $\Z/d$ grading, in which the empty set is assigned to have grading zero.  Finally, as a result of the ECH index parity property, for every $\Gamma \in H_1(Y,\Z)$, there is a canonical absolute $\Z/2$ grading given by the parity of the number of positive hyperbolic Reeb orbits \cite[\S 3.3]{Hindex}; also reviewed in \S \ref{ss:ECHI}.

\begin{definition}
If $Z \in H_2(Y,\alpha,\beta)$ and $\tau$ is a trivialization of $\xi$ over the Reeb orbits $\{\alpha_i\}$ and $\{\beta_j\}$, which is symplectic with respect to $d\lambda$, we define the \emph{ECH index} to be
\begin{equation}\label{ECHindex}
I(\alpha,\beta,Z) = c_\tau (Z) + Q_\tau(Z) + CZ_\tau^I(\alpha,\beta).
\end{equation}
The terms $c_\tau$ (relative first Chern number), $Q_\tau$ (relative intersection pairing), and $CZ^I_\tau$ (total Conley-Zehnder index) will be defined in \S\ref{s:generalities} and \S\ref{s:components}.
\end{definition}

Let $\M_k(\alpha,\beta,J)$ denote the set of $J$-holomorphic currents from $\alpha$ to $\beta$ with ECH index $k$.  The \emph{ECH differential} is given by 
\[
\partial \alpha = \sum_\beta \#_2 \left( \M_1(\alpha,\beta,J)/\R \right) \beta{,}
\]
{the} mod 2 count of ECH index 1 currents in $\R \times Y$, modulo $\R$ translation and equivalence of $J$-holomorphic currents. A \emph{$J$-holomorphic current} is a finite set of pairs $\cur= \{ (C_k, d_k)\}$, where the $C_k$ are distinct, connected, somewhere injective $J$-holomorphic curves in $(\R \times Y, d(e^s \lambda))$ and $d_k \in \Z_{\geq 0}$ subject to the asymptotic condition that $\cur$ ``converge as a current" to $\alpha$ as $s \to + \infty$ and to  $\beta$  as $s \to -\infty$. This asymptotic convergence as a current condition means that the positive ends of $C_k$ are at covers of the Reeb orbits $\alpha_i$ wherein the sum over $k$ of $d_k$ multiplied by the total covering amount of all ends of $C_k$ at iterates of $\alpha_i$ is $m_i$, and analogously for the negative ends.   The notion of a \emph{current}  is that of a linear functional on the space of differential forms and is due to Federer, cf. \cite[\S 4]{fed} and \cite[\S 1.4]{mor}.  Currents provide a natural topology on the space of real surfaces, admitting extremely useful compactness properties, cf. \cite[\S 4.2.17]{fed}, \cite[Thm.~5.5]{mor}, which were further developed and exploited by Taubes in the $J$-holomorphic setting \cite{swgr}.   

The definition of the ECH index and the associated index inequality is the key nontrivial ingredient used to define embedded contact homology \cite{Hindex}.  In particular, the assumption that $J$ is generic guarantees that the currents of ECH index 1 consist of {a single} embedded Fredholm and ECH index 1 $J$-holomorphic curve, and possibly an ECH index 0 current, which must be a union of trivial cylinders $\R \times \gamma$, with multiplicities, where $\gamma$ is a Reeb orbit. 

\begin{remark}[The role of degree]
A notion which will be crucial to this paper, as it was in \cite{preech}, is the \emph{degree} of an ECH generator, pair of generators, or curve counted by the differential (see Definition \ref{def:degreep}). Degree is a concept only appearing when the Reeb orbits of $\lambda$ agree with orbits of an $S^1$ action on $Y$, and is not intrinsic to ECH generally. (The $S^1$ action we use in this paper is described in \S\ref{s:topology}.) Essentially, degree counts the relative algebraic multiplicity of the multisets of fibers underlying the Reeb currents. In \cite{preech} and the sequel we relate the degree of a pair of generators to the degree of the holomorphic covering map to the base (here $\CP^1_{2,q}$) arising from any curves between them counted by the differential. As there is no differential in this paper, we do not need that relationship, but degree does govern both our action and knot filtrations, as is fully explained in \S\ref{s:spectral}.
\end{remark}

 We denote the homology of the ECH chain complex by $ECH_*(Y,\lambda,\Gamma,J)$.  That the differential squares to zero is rather involved and was established by Hutchings and Taubes \cite{obg1,obg2}; it requires obstruction bundle gluing as a result of the presence of an intermediate level consisting of multiply covered trivial cylinders between two ECH index one curves.  The embedded contact homology does not depend on the choice of $J$ or on the contact form $\lambda$ for $\xi$, and so defines a well-defined $\Z/2$-module $ECH_*(Y,\xi,\Gamma)$.   
  
 The proof of invariance of embedded contact homology goes through Taubes' isomorphism with Seiberg-Witten Floer cohomology \cite{taubesechswf}-\cite{taubesechswf5}.  
 
 \begin{theorem}[Taubes]\label{thm:taubes}
 If $Y$ is connected, then there is a canonical isomorphism of relatively graded $\Z[U]$-modules 
\[
ECH_*(Y,\lambda,\Gamma,J) \simeq \widehat{HM}^{-*}(Y,\mathfrak{s}_\xi + \op{PD}(\Gamma)),
\]
which sends the ECH contact invariant $c(\xi):=[\emptyset] \in ECH(Y,\xi,0)$ to the contact invariant in Seiberg-Witten Floer cohomology.
 \end{theorem}

Here $\widehat{HM}^*$ denotes the `from' version of Seiberg-Witten Floer cohomology, which is fully explained in the book by Kronheimer and Mrowka \cite{KMbook}, and $\mathfrak{s}_\xi$ denotes the canonical spin-c structure determined by the oriented 2-plane field $\xi$.  The contact invariant in Seiberg-Witten Floer cohomology was defined in \cite{km-contact}; see also \cite[\S 6.3]{kmos} and \cite{mariano}. Taubes' isomorphism demonstrates that ECH is a topological invariant of $Y$, except that one needs to shift $\Gamma$ when the contact structure is changed.\footnote{Given a fixed spin-c structure on a closed oriented 3-manifold $Y$, there is a construction establishing a one-to-one correspondence between the isomorphism classes of spin-c structures and the isomorphism classes of complex line bundles $L \to Y$.  We can replace isomorphism classes of complex line bundles with the elements of $H_1(Y;\Z)$ because line bundles $L$ are classified by $c_1(L) \in H^2(Y;\Z)$.  See also \cite[\S 1]{KMbook} and \cite[\S 6]{osbook}.}   The original motivation for the construction of embedded contact homology was to find a symplectic model for Seiberg-Witten Floer cohomology, so as to generalize Taubes' theorem establishing the equivalence between the Seiberg Witten invariant and Gromov invariant for closed symplectic 4-manifolds \cite{swgr}.
 
There  are two filtrations on embedded contact homology of interest, summarized below, with further  details given in \S \ref{s:spectral}.  The first is by symplectic action, which provides obstructions to symplectic embeddings of 4-manifolds \cite{qech} and enables computations of ECH by successive approximations for prequantization bundles \cite{preech}. The \emph{symplectic action} or {length} of an orbit set $\alpha=\{(\alpha_i,m_i)\}$ is
\begin{equation}\label{eq:action}
\mathcal{A}(\alpha):=\sum_im_i\int_{\alpha_i}\lambda.
\end{equation}
If $J$ is $\lambda$-compatible and there is a $J$-holomorphic current from $\alpha$ to $\beta$, then $\mathcal{A}(\alpha)\geq\mathcal{A}(\beta)$ by Stokes' theorem, since $d\lambda$ is an area form on such $J$-holomorphic curves. Since $\partial$ counts $J$-holomorphic currents, it decreases symplectic action,\footnote{In fact, $\A(\alpha) = \A(\beta)$ only if $\alpha = \beta$.} i.e.,
\[
\langle\partial\alpha,\beta\rangle\neq0\Rightarrow\mathcal{A}(\alpha)\geq\mathcal{A}(\beta).
\]

Let $ECC_*^L(Y,\lambda,\Gamma;J)$ denote the subgroup of $ECC_*(Y,\lambda,\Gamma;J)$ generated by admissible Reeb currents of symplectic action less than $L$. Since $\partial$ decreases action, it is a subcomplex. By \cite[Theorem 1.3]{cc2}  the homology of $ECC_*^L(Y,\lambda,\Gamma;J)$ is independent of $J$, therefore we denote its homology by $ECH_*^L(Y,\lambda,\Gamma)$, which we call \emph{action filtered embedded contact homology}.  Action filtered ECH of $(Y,\lambda)$ gives rise to the ECH spectrum $\{c_k(Y,\lambda)\}$, provided that the homology class of the cycle given by the empty set is nonzero, as explained in \S \ref{ss:ECHspectrum}.

The \emph{knot filtration} on embedded contact homology was first defined by Hutchings in \cite{HuMAC} for nondegenerate contact manifolds with $H_1(Y,\Z) =0$ and computed for the standard transverse unknot in the boundary of the four dimensional irrational ellipsoids.  

Knot filtered ECH is defined in terms of a linking number with a fixed embedded elliptic Reeb orbit realizing a transverse knot, denoted by $b$, with rotation number $\rot(b)$, defined in terms of a trivialization\footnote{This is a somewhat `atypical' choice of trivialization; usually one uses a trivialization which extends over a disk (or surface) spanned by the orbit, cf. Remarks \ref{rem:linkingtriv} and \ref{rem:c1sl}.} such that the  push off of $b$ with respect to the trivialization has linking number zero with $b$.  When $H_1(Y,\Z)=0$, this rotation number is well-defined, meaning that the push off linking zero trivialization is in a distinguished homotopy class.  In \cite[Thms.~5.2 \& 5.3]{weiler} it is explained how to obtain a well-defined rotation number for $Y$ when $H_1(Y,\Z)$ is torsion. Moreover, for an elliptic Reeb orbit $b$, the push off linking zero trivialization can always be chosen so that $P_b(t)$ is rotation by angle $2\pi \theta_t$ for each $t \in [0,T]$, and we set $\rot(b):=\theta_T$.   

Let $b^m\alpha$ be a Reeb current where $m \in \Z_{\geq 0}$  and $\alpha$ is a Reeb current not containing $b$. Assume $H_1(Y;\Z)=0$ so that the linking number of knots is well-defined. The \emph{knot filtration} on embedded contact homology with respect to $b$ of the Reeb current  $b^m\alpha$  is given by
\[
\mathcal{F}_b(b^m\alpha) := m \rot(b) + \ell(\alpha,b),
\]
where $\ell(\alpha, b)$ is given by
\[
{\ell(\alpha,b)=\sum_im_i\ell(\alpha_i,b).}
\]
 If $b$ is nondegenerate then  $\rot(b) \in \R \setminus \Q$.  When $b$ is nondegenerate, $\mathcal{F}_b$ is not integer valued, but it is true that if $\rot(b)>0$ (resp. $\rot(b)<0$) and if every Reeb orbit other than $b$ has nonnegative (resp. nonpositive) linking number with $b$, then $\mathcal{F}_b$ takes values in a discrete set of nonnegative (resp. nonpositive) real numbers.   The ECH differential $\partial$ does not increase the knot filtration $\mathcal{F}_b$, as proven in \cite[Lem.~5.1]{HuMAC}.

If $K$ is a real number, let
\[
ECH_*^{\mathcal{F}_b \leq K}(Y,\lambda,J, b, \rot(b)) 
\]
denote the homology of the subcomplex generated by admissible Reeb currents $b^m\alpha$ where $\mathcal{F}_b(b^m\alpha) \leq K$.   Unlike action filtered ECH, which depends heavily on the choice of contact form, knot filtered ECH is a topological spectral invariant denoted by $ECH_*^{\fb \leq K}(Y,\xi,b,\rot(b))$ which depends only on on $(Y,\xi)$, the Reeb orbit $b$ with fixed rotation number $\rot(b)$, and filtration level $K$ by \cite[Thm.~5.3]{HuMAC}.  In \S \ref{s:cobordismfun} we generalize the definition and invariance of knot filtered ECH to allow for rational rotation angles.

\begin{remark}[Comparison with knot embedded contact homology]\label{rem:ECK}
\ \\
 Knot filtered embedded contact homology is distinct from knot embedded contact homology $ECK$.   
 Using sutures, Colin, Ghiggini, Honda, and Hutchings \cite{cghh} defined a hat version $\widehat{ECK}(b,Y,\lambda)$ of knot embedded contact homology for a (neighborhood of a) null-homologous transverse knot $b$ in a closed contact manifold $(Y,\lambda)$.   Sutured ECH has been shown to be a topological invariant (up to isotopy of $\lambda$ and choice of embedding data of $J$) by Colin, Ghiggini, and Honda \cite{cgh}, as well as by Kutluhan, Sivek, and Taubes \cite{kst} (who additionally establish naturality of sutured ECH).    Colin, Ghiggini, and Honda conjecture that 
\[
\widehat{ECK}(b,Y,\lambda) \cong \widehat{HFK}(-b,-Y),
\]
in connection with the isomorphism between embedded contact homology and Heegaard Floer homology by Colin, Ghiggini, and Honda \cite{cgh}-\cite{cghiii} and the extension of Heegaard Floer homology to balanced sutured manifolds by  Juh\'asz \cite{juhasz}, which incorporates knot Floer homology of Ozsv\'ath-Szab\'o \cite{os-hfk} and Rasmussen \cite{ras-thesis} as a special case.   There is also an analogue of this story in monopole and instanton Floer homologies as developed by Kronheimer and Mrowka \cite{km-suture}.  Work of Kutluhan, Lee, and Taubes establishes the isomorphism between Seiberg-Witten Floer homology and Heegaard Floer homology \cite{klti}-\cite{kltv}.

The `hat' knot contact homology is defined as the first page of a spectral sequence arising from a filtration induced by a null-homologous transverse knot; this mirrors the filtration on Heegaard Floer homology induced by a null-homologous topological knot.  The hat version is equipped with an equivalent of the Alexander grading in the Heegaard Floer setting, and furthermore categorifies the Alexander polynomial.  After suitably adapting the contact form to an open book decomposition of the manifold, and using the sutured formulation of embedded contact homology \cite{cghh}, it is shown in \cite{cgh} that the knot embedded contact homology  can be computed by considering only orbits and differentials in the complement of the binding of the open book.

Spano explains how to define a ``full version" of knot embedded contact  homology $ECK$, which also extends to links, and shows that ECK is a categorification of the multivariable Alexander polynomial   \cite{spano}.  Brown generalizes these constructions to hold in rational open book decompositions, which permits a definition of $ECK$ for rationally null-homologous knots, and additionally establishes a large negative $n$-surgery formula for $ECK$ \cite{brown}.  
 The computation of $\widehat{ECK}$ for positive $T(2,q)$ torus knots in $(S^3,\xi_{std})$ is given in \cite[\S 11]{brown}.

J. Rasmussen has recently communicated to us that for positive $T(p,q)$  torus knots in $(S^3,\xi_{std})$, our computation of knot filtered embedded contact homology with rotation number $pq$ coincides with an associated filtration on the corresponding positive knot Heegaard Floer homology $HFK^+$.  However, it is unclear if this correspondence holds for more general knots or what the analogue of different rotation numbers correspond to in the Heegaard Floer setting, and merits further study.
\end{remark}

\subsection{Main results, organization, and future directions}\label{ss:main}
In \S \ref{s:cobordismfun}, we prove the following theorem, which allows us to generalize the definition and invariance properties of knot filtered ECH to allow for degenerate contact forms so that the rotation number can be rational and provides Morse-Bott computational methods for appropriate Seifert fiber spaces in the spirit of \cite{preech}.

\begin{definition}\label{def:knotadmissible}
A pair of families $\{(\lambda_\ve, J_\ve)\}$ is said to be a \emph{knot admissible pair} for $(Y,\lambda, b, \rot(b))$, where $\lambda$ is a degenerate contact form admitting the transverse knot $b$ as an embedded Reeb orbit whenever
\begin{itemize}
\itemsep-.25em

\item  $f_{s}: [0,c_0]_s \times Y \to \R_{>0}$  are smooth functions such that $\frac{\partial f}{\partial s}>0$ and  $\lim_{s \to 0}f_{s}=1$ in the $C^0$-topology;
\item  There is a full measure set $\mathcal{S} \subset(0,c_0]$ such that for each $\ve \in \mathcal{S}$, $f_{\ve}\lambda$  is $L(\ve)$-nondegenerate and $L(\ve)$ monotonically increases towards $+\infty$ as $\ve$ decreases towards 0;
\item $\lambda_{\varepsilon}:=f_{\varepsilon}\lambda$ each admit the transverse knot $b$ as an embedded elliptic Reeb orbit (when $\ve \neq 0$) and $\{ \rot_{\varepsilon}(b)\}$ is monotonically decreasing to $\rot(b)$ as $\varepsilon \in [0, c_0]$ decreases to 0;
\item $J_{\varepsilon}$ is an $ECH^{L(\varepsilon)}$ generic $\lambda_{\varepsilon}$-compatible almost complex structure (when $\varepsilon \neq 0$).
\end{itemize}
Sometimes we also suppress the almost complex structure and refer to the sequence of contact forms $\{\lambda_\ve\}$ as a \emph{knot admissible family}, provided it satisfies the above conditions.  Precise definitions of each condition can be found in the statement of Lemma \ref{lem:nicecobmap1}.  By the discussion in Lemma \ref{lem:nicecobmap1} and Remark \ref{rem:nicecobmap1}, it follows that for any $\varepsilon' \in (0,\varepsilon)$, the admissible Reeb currents of action less than $L(\varepsilon)$ associated to $\lambda_{\varepsilon}$ and $\lambda_{\varepsilon'}$ are in bijective correspondence.  
\end{definition}

\begin{theorem}\label{thm:introok}
Let $(Y, \xi)$ be a closed contact 3-manifold with $H_1(Y)=0$, $b\subset Y$ be a transverse knot and $K\in \R$.  If $\lambda$ is degenerate, 
we define
\[
ECH_*^{\mathcal{F}_b \leq K}(Y,\lambda, b, \rot(b)) :=
\lim_{\varepsilon \to 0}ECH_*^{\afe}(Y, \lambda_{\varepsilon}, b, \rot_{\varepsilon}(b),J_{\varepsilon}),  \]
where $\{(\lambda_\ve, J_\ve)\}$ is a  {knot admissible pair} for $(Y,\lambda, b, \rot(b))$, and the right hand side is the action filtered subcomplex, which has been further restricted to the knot filtered subcomplex.    Then $ECH_*^{\mathcal{F}_b \leq K}(Y,\lambda, b, \rot(b))$ is well-defined and depends only on $Y, \xi, b, \rot(b),$ and $K$.
\end{theorem}

The proof of this result is carried out in \S \ref{s:cobordismfun}.  It comes by way of a doubly filtered direct limit, where the chain maps induced by cobordisms on action filtered embedded contact homology chain complexes are obtained from energy filtered perturbed Seiberg-Witten Floer theory via the results of Hutchings and Taubes \cite{cc2}.  This is a more involved generalization of the direct systems and direct limits we carried out in \cite[\S 7]{preech}.  
  
 In particular, this procedure also allows us to compute knot filtered embedded contact homology via successive approximations involving knot admissible families of contact forms admitting a fixed transverse knot as a nondegenerate embedded elliptic Reeb orbit with monotonically decreasing irrational rotation numbers converging to a rational rotation number.   Our methods allow one to directly construct a knot admissible family of contact forms for any fiber of a Seifert fiber space with negative Euler class, equipped with a tight $S^1$-invariant contact structure.

We now explain our result and methods for computing positive $T(2,q)$ filtered embedded contact homology of $(S^3,\xi_{std})$.  Consider the unit 3-sphere $S^{3}$ in $\C^{2}$ and let $J_{\C^2}$ be the standard complex structure on $\C^2$.  Then the standard tight contact structure is given by
\[
\left(\xi_{std}\right)\vert_p = T_pS^3 \cap J_{\C^2}(T_pS^3)
\]
and may expressed as the kernel of the 1-form
\[
\lambda_0 = \frac{i}{2}  \left( z_1 d\bar{z}_1 - \bar{z}_1 dz_1 + z_2 d\bar{z}_2 - \bar{z}_2 dz_2 \right).
\]
We can realize the right handed torus knot $T(p,q)$ in $S^3$ as
\[
T(p,q)=\left\{(z_1,z_2)\in S^3 \subset \C^2 \ |\ z_1^p+z_2^q=0\right\};
\]
the projection map $\pi:S^3 \ \setminus \ T(p,q) \to S^1$ is the Milnor fibration, cf.~Remark \ref{rmrk:Milnor}.   Etnyre shows in \cite{torus}, that positive (e.g. right-handed) transversal torus knots are transversely isotopic if and only if they have the same topological knot type and the same self-linking number. Thus it makes sense to refer to the standard transverse positive (right-handed) $T(p,q)$ torus knot, which we denote by ${b}$ and is of maximal self-linking number $pq-p-q$.

Given $p,q\in\R$, let $N(p,q)$ denote the sequence $(pm+qn)_{m,n \in \Z_{\geq0}}$ of nonnegative integer linear combinations of $p$ and $q$, written in increasing order with multiplicity. We use $N_k(p,q)$ to denote the $k^\text{th}$ element of this sequence, including multiples and starting with $N_0(p,q)=0.$  We can now state our main results.

\begin{theorem}\label{thm:kech-intro}
Let $\xi_{std}$ be the standard tight contact structure on $S^3$.  
Let ${b_0}$ be the standard {right-handed} transverse $T(2,q)$ torus knot for $q$ odd and positive.  Then for $k \in \N$,
\[
ECH_{2k}^{\fb \leq K}(S^3,\xi_{std},{b_0},2q)=\begin{cases}\Z/2&K\geq{ N_k(2,q)  },
\\0&\text{otherwise,}\end{cases}
\]
and in all other gradings $*$,
\[
ECH_*^{\fb \leq K}(S^3,\xi_{std},{b_0},2q)=0.
\]
If $\delta$ is a sufficiently small positive irrational number, then up to grading $k \in \N$ and knot filtration threshold $K$ inversely proportional to $\delta$,
\[
ECH_{2k}^{\fb \leq K}(S^3,\xi_{std},{b_0},2q+\delta)=\begin{cases}\Z/2&K\geq{ N_k(2,q) + \delta {(\$N_k(2,q) -1)}},
\\0&\text{otherwise,}\end{cases}
\]
where { $\$N_k(2,q)$ is the number of repeats in $\{ N_j(2,q)\}_{j\leq k}$ with value $N_k(2,q)$,} and in all other gradings $*$, up to the threshold inversely proportional to $\delta$,
\[
ECH_*^{\fb \leq K}(S^3,\xi_{std},{b_0},2q+\delta)=0.
\]
\end{theorem}

\begin{remark}[Threshold between grading and filtration with $\delta$]
The relationship between the threshold of the grading $2k$ and filtration level $K$ with respect to the size of $\delta$ is as follows.  We require $\delta$ to be small enough so that $N_k(2,q)+\delta(\$N_k(2,q)-1)  \leq N_{k+1}(2,q)$ for all $k$.  
\end{remark}

\begin{remark}[Generalization to $T(p,q)$ for $p\neq 2$]\label{rem:intropq} \ \\
We have constructed alternate knot admissible families of contact forms associated to a different family of orbifold Morse functions and established the associated ECH chain complex.  We establish that the same result holds in Theorem \ref{thm:kech-intro} with 2 replaced by $p$ and $\op{gcd}(p,q)=1$ using a different perturbation in \cite{calabi}.    
\end{remark}

\begin{remark}[Comparison to a toric perturbation]\label{rmk:toric}
Using a clever perturbation, cf.~Figure \ref{fig:morse23}, the case $p=2$ can be handled entirely combinatorially; this is fairly involved, but rewarding as we can prove using (comparatively) elementary methods that the differential vanishes.  
Still, one might ask why we use such methods, seeing as the contact form in question is contactomorphic to a very simple one, namely a rescaling of an ellipsoid (as can be proven via Kegel-Lange \cite{kl} and Cristofaro-Gardiner -- Mazzucchelli \cite{cgm}).  There are indeed convex and concave toric perturbations of these degenerate contact forms, and their generators, ECH indices, and ECH spectra are well-understood via lattice paths in the plane. These toric perturbations have a different chain complex than the one studied in this paper when $p=2$, and are more similar to the chain complex studied in the sequel, as detailed in \cite{calabi}. However, the combinatorial toric ECH differential is not fully understood, as we explain in a detailed comparison to a hypothetical convex toric differential, cf.~\cite[\S 5.6]{calabi}. The complication arises when including the ``virtual edges" corresponding to the covers of certain elliptic orbits. Thus we required new non-toric methods.

While it is a priori feasible to understand the toric differential involving ``virtual edges" in terms of Taubes' (punctured) pseudoholomorphic beasts in $\R \times (S^1 \times S^2)$ \cite{letter, beasts, beasts2, beasts3}, our approach has a number of advantages.  It allows one to understand the embedded contact homology chain complexes of more general Seifert fiber spaces and open books. This is not possible from a toric perspective, as toric contact forms only exist on closed 3-manifolds which are diffeomorphic to 3-spheres, $S^2 \times S^1$, or lens spaces. It is also more geometrically intuitive to compute the knot filtration from the open book.
\end{remark}

Since $H_1(S^3,\Z)=0$, we also need a detailed understanding of the ECH index of arbitrary generator sets because all orbit sets are homologous, as opposed to the case in \cite{preech} where the ordinary homology helped us organize the admissible Reeb currents.   Also, establishing formulae for the sundry Conley-Zehnder sums of Reeb orbit fibers which project to orbifold points is a bit intense.   As a result, our derivation of a formula for the ECH index and establishing its bijective correspondence to the nonnegative even integers are nontrivial combinatorial endeavors.  They are necessary to establish the behavior of the differential and our computation of knot filtered embedded contact homology. In particular, in order to compute the knot filtration one must know the ECH index and action of every homologically essential generator in addition to its knot filtration value.

We now review the outline of our arguments, indicating where they appear in the paper.  (The interested reader may wish to additionally consult the more detailed summary of the chain complexes and the statements of our index theorems in \S \ref{ss:2}.)

\subsubsection*{Reeb dynamics}
In \S \ref{s:topology} we detail the Reeb dynamics in terms of the associated open book decomposition, Seifert fibration, and prequantization orbibundle of $S^3$ realizing the transverse right handed torus knots $T(p,q)$ as Reeb orbits. The naturally associated contact forms $\lambda_{p,q}$ are established to all be strictly contactomorphic with each other by way of Kegel-Lange \cite{kl} and Cristofaro-Gardiner -- Mazzucchelli \cite{cgm}. In particular, the Reeb vector field associated to $\lambda_{p,q}$ generates the Seifert fibration; the Reeb orbits correspond to the fibers of the prequantization orbibundle.  The pages of the open book (which supports $\lambda_{p,q}$) are used to induce a trivialization along the $T(p,q)$ binding, with respect to which we compute knot filtered embedded contact homology.  The pages also play a supporting role in finding additional surfaces and trivializations used in our computation of the ECH index.

Now specializing to the case $p=2$, we perturb the degenerate contact form $\lambda_{p,q}$, using the lift of a perfect orbifold Morse-Smale function $H_{2,q}$ on the base orbifold $\CP^1_{2,q}$ by exploiting geometric symmetries present in the $T(2,q)$ fibration of $S^3$;  cf.~Figure \ref{fig:morse23}.  This is in the spirit of \cite{leo, jo2, preech}, and we define
\[
\lambda_{2,q,\varepsilon}:=  (1+\varepsilon \frak{p}^*H_{2,q})\lambda_{2,q}.
\]
Up to large action $L$, the only Reeb orbits of perturbed Reeb vector field $R_{2,q,\varepsilon(L)}$ are the fiber iterates of:

\begin{itemize}
\itemsep-.25em
\item The binding ${b}$, an embedded elliptic orbit, which is a regular $T(2,q)$ knotted fiber orbit that projects to the nonsingular maximum of $H_{2,q}$;
\item The embedded negative hyperbolic orbit $h$, a singular fiber of the Seifert fibration, which projects to the singular index 1 critical orbifold point of $H_{2,q}$ with isotropy $\Z/2$;
\item The embedded elliptic orbit $e$, a singular fiber of the Seifert fibration, which projects to the singular minimum  of $H_{2,q}$ with isotropy $\Z/q$.  
 \end{itemize}

In particular, the generators of $ECC_*^{L}(S^3,\lambda_{2,q,\varepsilon(L)}, J)$ are of the form $b^Bh^He^E$, where $B,E \in \Z_{\geq 0}$ and $H=0,1$. As a result, we obtain the direct system $\{ ECH^L_*(Y,\lambda_{2,q,\varepsilon(L)})\}$ such that the direct limit  is the homology of the chain complex generated by the associated admissible Reeb currents.

\subsubsection*{ECH index}

 Our computation of the ECH index, completed in \S \ref{s:generalities}-\ref{s:ECHI}, makes use of three different trivializations. Relating everything together so that we can understand the ECH index of arbitrary orbit sets takes some care. We now sketch some of what goes into this.  To understand some patterns in the generator sets, see Table \ref{table:gen2}.

In \S \ref{s:generalities}, we first assemble the necessary generalities about the ECH index scattered throughout \cite{Hindex,Hrevisit,lecture} and set up our trivializations: the constant trivialization for regular fibers, the orbibundle trivialization for all fibers, and the page trivialization induced by the open book decomposition {for the binding}.  We also establish a  change in trivialization formula {for covers of simple orbits}, Proposition \ref{prop:nonsimpleCOT}, which may be of independent interest, and is used to relate our trivializations in the setting at hand.

The prequantization orbibundle description allows us to understand the monodromy angles determining the Conley-Zehnder index and first Chern number $c_\tau$, which is carried out in \S \ref{s:components}.  The calculation of the total Conley-Zehnder indices $CZ^I_\tau$ is combinatorially involved and completed in \S \ref{s:ECHI}.  Understanding the relative intersection pairing $Q_\tau$ requires one to find suitable surfaces representing classes in $H_2(Y,\alpha,\beta)$.  The orbibundle is not well suited for this task, so we use the open book and Seifert fiber space presentations, and compute the relative intersection pairing in \S \ref{s:components}.  

\subsubsection*{Spectral invariants of ECH}

 The combinatorics detailed in \S \ref{s:ECHI} allow us to establish the relationship between the ECH index of a generator and its \textit{degree} (defined below).  This in turn governs the associated spectral invariants of embedded contact homology of interest, which we establish in \S \ref{s:spectral}.

  \begin{definition}\label{def:degreep}
Given a pair of (homologous) Reeb currents $\alpha$ and $\beta$ expressed in terms of embedded orbits realizing fibers of a prequantization orbibundle or Seifert fiber space, we define their \textit{relative degree} to be the relative algebraic multiplicity of the associated fiber sets. That is, given (homologous) Reeb currents $\alpha=b^Bh^He^E$ and $\beta=b^{B'}h^{H'}e^{E'}$, the relative degree of the pair  is:
\[
d(\alpha, \beta) =  \frac{B+\frac{1}{2}H+\frac{1}{q}E-B'-\frac{1}{2}H' -\frac{1}{q}E'}{| e| } =  2qB + qH +2E-2qB' - qH' -2E'.
\]
\end{definition}

\begin{remark}\label{def:degree}
\begin{enumerate}[{ (i)}]
\itemsep-.25em
\item We usually consider the degree of a single Reeb current, which is defined to be the relative degree of the pair when $\beta=\emptyset$; we will denote this degree by $d(b^Bh^He^E)$ or simply $d$ when the generator is unspecified or clear from context.

\item Intuitively, the degree of $b$ is $2q$ because a regular fiber bounds a page of the open book decomposition, which is a $2q$-fold cover of the base $\CP^1_{2,q}$. The $q$-fold cover $e^q$ of $e$ also bounds a page (in homology; to see $e^q$ as the boundary of a surface homologous to a page, the surface $S_e$ made up of a union of fibers must be glued to the page along $b$), so the homology intersection number of a surface with boundary $e$ with a regular fiber (i.e., the degree of $e$) must be two ($2q$ divided by $q$). Similarly, the degree of $h$ must be $q$ ($2q$ divided by 2).
\end{enumerate}
\end{remark}

\begin{remark}[Generalization to Seifert fiber spaces]
When the ECH differential does not vanish for index reasons, the degree $d$ of a pair of admissible Reeb currents $(\alpha,\beta)$ corresponds to the \textit{degree} of {the cover of the orbifold base induced by} any curves counted in $\langle\partial\alpha,\beta\rangle$ for Seifert fiber spaces of negative Euler class, similar to \cite[\S 4]{preech}. In analogy with \cite{preech, moy}, the embedded contact homology of a Seifert fiber space equipped with an $S^1$-invariant contact structure is expected to recover the exterior algebra of the orbifold Morse homology of the base.  However, depending on the choice of orbifold Morse function, there will not always be a bijective correspondence between generators at the chain level, as evidenced in the $T(2,q)$-fibration of $S^3$.  
\end{remark}

The degree allows us to compute the ECH spectrum: When the grading $k$ is sufficiently small relative to $L(\varepsilon)$, the degree of the Reeb current representing the generator of the group $ECH_{2k}^{L(\varepsilon)}(S^3,\lambda_{2,q,\varepsilon})$ is $N_k(2,q)$, which allows us to establish that
 \[
c_k(S^3,\lambda_{2,q}) = N_k({1/2}, {1/q}).
\]

The knot filtered ECH of $(S^3,\xi_{std})$ with respect to the standard transverse right handed $T(2,q)$ knot with rotation number $2q + \delta$, where $\delta$ is either 0 or a sufficiently small positive irrational number, is governed by the degree as well.  In \S\ref{ss:kECH}, we show that for any Reeb current $\alpha$ not including the standard right handed transverse $T(p,q)$ torus knot  ${b}$,
\[
\mathcal{F}_{{b}}({b}^B\alpha) =  B \rot(b) + \ell(\alpha,{b}) = d({b}^B \alpha) + B\delta_{L(\varepsilon)}.
\]
Theorem \ref{thm:kech-intro} then follows from the description of the filtered chain complex in \S \ref{s:ECHI} via successive approximations and the Morse-Bott direct limit arguments using the sequence of contact forms $\{ \lambda_{2,p,\ve} \}$, which is a knot admissible family by the computations preceding and summarized in Lemma \ref{lem:orbtrivCZ2}.  Finally, we establish Theorem \ref{thm:introok} in \S \ref{s:cobordismfun}.

\begin{remark}
Since $d({b}^B \alpha) =  2q \ \A_{\lambda_{2,q}}({b}^B \alpha)$, knot filtered ECH is able to realize the relationship between action and linking in this class of examples, cf. \cite[Prop.~1.3]{bhs}.  We will elucidate this relationship further in the sequel \cite{calabi}. Because our perturbation $\lambda_{p,q,\varepsilon}$ is not toric, our work also could be used to bound the systolic interval for a larger class of perturbations of ellipsoids {by carefully controlling the estimates on action appearing in Lemma \ref{lem:efromL}(ii)}. \end{remark}

\begin{remark}
With additional development, knot filtered embedded contact homology can be used as a means to obtain new obstructions of relative symplectic cobordisms between transverse knots in contact 3-manifolds.    To establish results for strong symplectic cobordisms the results of \cite{echsft} will be beneficial.   A better understanding of how knot filtered ECH changes with respect to ``large changes" in the rotation number will be helpful.  Presently, knot filtered ECH using widely varying irrational rotation numbers has only been computed in $S^3$ and lens spaces $L(n,n-1)$ with irrational rotation numbers in \cite{HuMAC, weiler, weiler2}.
 \end{remark}

  \noindent \textbf{Acknowledgements.} 
We wish to thank Gordana Mati\'c and Jeremy Van Horn-Morris for elucidating conversations during the Braids in Low-Dimensional Topology conference at the Institute for Computational and Experimental Research in Mathematics (ICERM) in Providence, RI; ICERM programming is supported by the National Science Foundation under DMS-1929284.  We also thank Michael Hutchings, Umberto Hryniewicz, Tom Mrowka, Peter Ozsv\'ath, Abror Pirnapasov, Jake Rassmussen and Sara Venkatesh for helpful conversations. Finally, we thank the anonymous referee for their thoughtful suggestions and comments on an earlier version of this paper. Jo Nelson is partially supported by NSF grants DMS-2104411 and CAREER DMS-2142694.  During her stay at the Institute for Advanced Study, she was supported as a {Stacy and James Sarvis Founders’ Circle Member}. Morgan Weiler is partially supported by an NSF MSPRF grant DMS-2103245.

\section{From open books to orbibundles}\label{s:topology}

In this section we review how to obtain the open book decomposition of $(S^3,\xi_{std})$ along a right handed $T(p,q)$ torus knot and identify the Seifert invariants.  We then review why the $T(p,q)$ fibrations of $(S^3,\xi_{std})$ are strictly contactomorphic to the prequantization orbibundles over complex one dimensional weighted projective space $\CP_{p,q}^1$ with Euler class $-\frac{1}{pq}$.  Using the latter description, we perturb the contact form induced by the orbibundle connection 1-form using (the lift of) an appropriate Morse-Smale function on the base orbifold. We then describe the associated perturbed Reeb dynamics.

\subsection{Open books along right handed torus knots}\label{ss:OBDs}
The open book decompositions of $S^3$ along the right handed $T(p,q)$ knots are obtained as ``stabilizations" of explicit open book decompositions of $S^3$ with annular pages and Hopf link bindings.  This process can be iterated as explained in \cite[\S 9]{osbook},\cite[\S 2]{ao} to obtain the genus $\frac{(p-1)(q-1)}{2}$ open book decomposition of $(S^3,\xi_{std})$ along the-right handed $T(p,q)$ knot. Torus links correspond to rational open books of lens spaces, see \cite{cabling}. We thus only consider torus knots, so thus assume $p$ and $q$ are relatively prime.

An \emph{open book decomposition} $(B,\pi)$ of a closed oriented 3-manifold $Y$ is an oriented link $B\subset Y$, called the \emph{binding}, together with a fibering $\pi:Y\setminus B\to S^1$ such that $\pi^{-1}(\theta)$, $\theta\in S^1$, is a Seifert surface for $B$. The closures $\overline{\pi^{-1}(\theta)}$ are called \emph{pages}. The \emph{monodromy} $\phi_\pi$ of $(B,\pi)$ is the isotopy class {(relative to the boundary)} of the return map of the flow of any vector field which is positively transverse to the pages and meridional near $B$.

An open book decomposition is entirely determined by the diffeomorphism type of its pages and isotopy class of its monodromy: an \emph{abstract open book} is a pair $(\Sigma,\phi)$ where $\Sigma$ is an oriented compact surface with nonempty boundary and $\phi$ is a diffeomorphism of $\Sigma$ which is the identity near $\partial\Sigma$. An abstract open book determines an open book decomposition $(B_\phi,\pi_\phi)$ of the manifold $Y_\phi:=\Sigma\times[0,1]/\sim_\phi$, where $(z,1)\sim_\phi(\phi(z),0)$ for all $z\in\Sigma$ and $(z,t)\sim_\phi(z,t')$ for all $z\in\partial\Sigma$. The binding $B_\phi$ is $\partial\Sigma\times[0,1]/\sim_\phi$ and the projection map $\pi_\phi$ is simply projection onto the $[0,1]$-coordinate. Abstract open books are equivalent if there is a diffeomorphism of their pages under which their monodromies are conjugate.

\begin{definition}
The \emph{stabilization} of an abstract open book $(\Sigma,\phi)$ is the abstract open book whose page $\Sigma'$ is obtained from $\Sigma$ by attaching a 1-handle and whose monodromy $\phi'$ is the composition $\phi\circ \tau_c$, where $c$ is a closed curve in $\Sigma'$ intersecting the co-core of the new 1-handle exactly once and $\tau_c$ is a Dehn twist along $c$; if $\tau_c$ is a right-handed Dehn twist then we say $(\Sigma',\phi')$ is a \emph{positive stabilization} of $(\Sigma,\phi)$ and if $\tau_c$ is a left-handed Dehn twist we say it is a \emph{negative stabilization}. The underlying 3-manifolds determined by $(\Sigma,\phi)$ and $(\Sigma',\phi')$ are diffeomorphic, no matter the choice of the curve $c$.
\end{definition}

One can view stabilizations (resp.~destablization) as plumbing (resp.~deplumbing) Hopf bands.  Since plumbing a Hopf band at the level of 3-manifolds is equal to taking a connected sum with $S^3$, by definition we do not change the topology of the underlying 3-manifold.   As detailed in \cite[\S 9]{osbook}, one can plumb two positive Hopf links to get the right-handed trefoil $T(2,3)$.  The resulting monodromy will be the product of two right-handed Dehn twists.  Iterating this plumbing operation allows one to express the monodromy of a right handed $T(2,q)$ torus knot as a product of $(q-1)$ right-handed Dehn twists.  

By attaching additional positive Hopf bands, we can construct the fibered surface of a right handed $(p,q)$-torus knot for arbitrary $p$ and $q$.  By \cite[Thm.~1, Fig.~4-5]{ao,aoe}, the monodromy of a right-handed $T(p,q)$ torus knot is a product of $(p-1)(q-1)$ nonseparating positive Dehn twists.   We denote the corresponding abstract open book by $(T(p,q),\pi)$.  The page of this open book is a surface of genus $\frac{(p-1)(q-1)}{2}$ and the monodromy is $pq$-periodic.

An open book decomposition of a 3-manifold $Y$ and a cooriented contact structure $\xi$ on $Y$ are called \emph{compatible} if $\xi$ can be represented by a contact form $\lambda$ such that the binding is a transverse link, $d\lambda$ is a volume form on every page, and the orientation of the transverse binding induced by $\lambda$ agrees with the boundary orientation of the pages.  We will call a contact form $\lambda$ \emph{adapted} to an open book if the above conditions hold.  That every open book decomposition of a closed and oriented 3-manifold admits a compatible contact structure is due to Thurston-Winkelnkemper \cite{tw}.  Giroux substantially refined this result \cite{giroux}, and showed that any two contact structures compatible with a given open book decomposition are isotopic; see also \cite[\S 9, \S 11]{osbook}. See also the recent proof of Breen-Honda-Huang \cite{bhh}, which extends Giroux' result to all dimensions.

As explained in \cite[Rem.~9.2.12]{osbook}, in the case of a positive stabilization of a compatible open book on $(Y,\xi)$, the resulting open book is obtained by plumbing a positive Hopf band to a page of the original open book.  The contact structure compatible with the resulting open book is a contact connected sum $\xi \# \xi_{std}$, which is isotopic to $\xi$.  The plumbing procedure  is a special case of the Murasugi sum.  Ambient stabilization is described in terms of an ambient Murasugi sum in \cite{etnlec}.  Moreover, the contact Murasugi sum induces the connect sum of contact manifolds \cite[Prop.~2.6]{cm}. Giroux's theorem \cite{bhh, giroux} states that given a closed 3-manifold $Y$, there is a one to one correspondence between oriented contact structures on $Y$ up to isotopy and open book decompositions of $Y$ up to stabilization.

We have the following relationship between abstract open books with periodic monodromy and Seifert fibrations.  Explicit contact forms $\lambda_{p,q}$ and $\lambda_{p,q,\varepsilon}$ adapted to the open book $(T(p,q),\pi)$ are described later.

\begin{theorem}[{\cite[Thm.~4.1]{CH}}]\label{thm:CH} Suppose $(\Sigma,\phi)$ is an abstract open book with periodic monodromy $\phi$. Let $c_i$ be the fractional Dehn twist coefficient of the $i^\text{th}$ boundary component and assume all $c_i>0$. Then there is a Seifert fibration on $Y_\phi$ and a contact form $\lambda_\phi$ on $Y_\phi$ adapted to the open book decomposition $(B_\phi,\pi_\phi)$ whose Reeb vector field generates the Seifert fibration.
\end{theorem}

{The \emph{fractional Dehn twist coefficient} measures the difference between a representative of the monodromy $\phi$ and its Nielsen-Thurston representative (which is not necessarily the identity along the boundary). See \cite[\S1.1]{CH} for a definition.} In particular, when all the $c_i$ are positive, then $\xi$ is an $S^1$-invariant contact structure which is transverse to the $S^1$-fibers of the Seifert fibration.  Moreover, by \cite[Lem.~4.3]{CH}, we have that for any Seifert fibered space $Y$ with a fixed fibering, any two $S^1$-invariant transverse contact structures are isotopic.  These results go through \cite{lm}.

\begin{remark}\label{cor:CHdetails} In the setting of Theorem \ref{thm:CH}, for the open book $(T(p,q),\pi)$, we have:
\begin{enumerate}[ (i)]
\itemsep-.25em
\item The (right handed) $pq$-periodic representative $\psi$ of the monodromy $\phi$  is the return map of the Reeb vector field $R_\phi$ of $\lambda_\phi$.
\item $R_\phi$ generates the $S^1$-actions on $S^3$ with fundamental domain given by an orbifold 2-sphere with two exceptional points, one with isotropy group $\Z/p\Z$ and the other with isotropy group $\Z/q\Z$.
\item The binding  $T(p,q)$ is a regular fiber of the Seifert fibration.
\end{enumerate}
With respect to a preferred trivialization induced by the page, the inverse of the fractional Dehn twist coefficient realizes the rotation number of the binding Reeb orbit, see \S \ref{ss:pagetau}. 
\end{remark}

To understand the associated Seifert invariants, it is helpful to recall how the open book decomposition $(T(p,q),\pi)$ can be written in coordinates.

\begin{remark}\label{rmrk:Milnor}
 We can realize the right handed torus knot $T(p,q)$ as
\[
T(p,q)=\left\{(z_1,z_2)\in S^3 \subset \C^2 \ |\ z_1^p+z_2^q=0\right\},
\]
and the projection map $\pi$ is the Milnor fibration\footnote{D. Dreibelbis created
visualizations of Milnor's fibration theorem for torus knots and links at: \\ \url{https://www.unf.edu/~ddreibel/research/milnor/milnor.html} \\ \indent \  \ H. Blanchette created an interactive model for the trefoil at: \\  \url{http://people.reed.edu/~ormsbyk/projectproject/posts/milnor-fibrations.html} \\
\indent \  \ B. Baker created an animation for $T(2,3)$: \\ \url{https://sketchesoftopology.wordpress.com/2012/08/24/bowman/}
}
\[
\pi:S^3\setminus T(p,q) \to S^1, \ (z_1,z_2)\mapsto\frac{z_1^p+z_2^q}{|z_1^p+z_2^q|}.
\]

There is a well-known $S^1$ action on $S^3$ inducing the associated Seifert fibration,  
which is the flow of the Reeb vector field discussed in \S \ref{ss:morse}. The $S^1$-action
\[
e^{{2\pi}it}\cdot(z_1,z_2)=\left(e^{{2\pi}pit}z_1,e^{{2\pi}qit}z_2\right)
\]
is positively transverse to the pages of $(T(p,q),\pi)$.

\end{remark}

\subsection{Seifert fiber spaces and $S^1$-invariant contact structures}\label{ss:Seifert}

We now provide some background on Seifert fiber spaces and $S^1$-invariant contact structures, primarily to fix our notational conventions, which agree with \cite{lm}.

\begin{definition}\cite[\S 14]{lrbook}\label{def:seifert} A \emph{Seifert fiber space} is a 3-manifold $Y$, which can be decomposed into a union of disjoint circles, the \emph{fibers}, such that each circle has a neighborhood which is fiber-preserving homeomorphic to an \emph{$(\mu,\nu)$-fibered solid torus} $T_{\mu,\nu}$ where $\gcd(\mu,\nu) =1$ and $\mu \geq 1$:\[
\mathbb{D}^2_z\times[0,2\pi]/(z,0)\sim\left(e^{2\pi i \nu/\mu}z,2\pi\right).
 \]
In other words, a $(\mu,\nu)$-fibered solid torus  $T_{\mu,\nu}$ is the mapping torus of a $\frac{\nu}{\mu}$ rotation.  The fiber  $\{0\}\times S^1$ is a core circle of the solid $(\mu,\nu)$-torus, while the other fibers represent $\mu$-times the core circle.\footnote{J. Bettencourt provides an approximate means of visualizing a fibration at: \\  \url{http://www.jessebett.com/TorusKnotFibration/index.html}}  Collapsing each of the fibers of the solid $(\mu,\nu)$-torus defines a quotient map to a disk $D^2$.  If $\nu=0$, then $T_{\mu,\nu}$ is an $S^1$-bundle over $\mathbb{D}^2$, but if $\nu \neq 0$ then $T_{\mu,\nu}$ is an $S^1$-bundle only over $\mathbb{D}^2 \setminus \{ 0 \}$ and the core circle is called an \emph{exceptional fiber}.

There is a homeomorphism from $T_{\mu, \nu+\mu}$ to $T_{\mu, \nu}$ that takes fibers to fibers, so we can always pick $\nu$ such that $0\leq \nu < \mu$.  If $\nu \neq 0$, define the \emph{unnormalized Seifert invariant} $(a,b)$ by
\[
\begin{array}{lcl}
a &=&\mu \\
b \nu &\equiv& 1 \mod \mu. \\
\end{array}
\]
(It is possible and sometimes of interest to normalize the Seifert invariant so that $0 \leq b < a$, but this is not strictly necessary.) We say that the core circle of $T_{\mu,\nu}$ is an exceptional fiber of type $(a,b)$.  The quotient space of $Y$ determined by collapsing each fiber to a point is a 2-orbifold $\cO$, and the quotient map $\fp: Y \to \Sigma$ is called a \emph{Seifert fibering} of $Y$ over $\Sigma$.
\end{definition}

\begin{remark}
Not all Seifert fiber spaces are orientable, nor is the base of an orientable Seifert fiber space necessarily an orientable orbifold. However, in this paper both the base and total space are orientable.
\end{remark}

\begin{definition}\label{def:seifert}
Let $\fp: Y \to \cO$ be an oriented three dimensional Seifert fibration with oriented base  of genus $g$ and normalized Seifert invariants $Y(g; b; (a_1,b_1),...,(a_r,b_r))$, in the sense of \cite{orlik, lrbook} and in agreement with \cite{lm}. The \emph{Euler class} of $Y$ is defined by the rational number $e(Y):= -b - \sum_{i=1}^r \frac{b_i}{a_i}.$  (One obtains the same number when using unnormalized Seifert invariants.)
\end{definition}

\begin{remark}
Since $\alpha$ and $\beta$ are reserved for Reeb currents, we have had to make a slight abuse of notation.  Each $(a_i, b_i)$ corresponds to the Seifert invariant $(a,b)$  in Definition \ref{def:seifert}. The standalone $b$ indicates that there is a fiber of type $(1,b)$ present.  Moreover, the normalization convention dictates that $0 \leq b_i<a_i$ with $\gcd(a_i,b_i)=1$.  \end{remark}

\begin{remark}
Lisca and Mati\'c gave a complete answer to the question of which Seifert fibered 3-manifolds admit a contact structure transverse to their fibers in \cite[Thm.~1.3, Cor.~2.2, Prop.~3.1]{lm}, which moreover is shown to be universally tight.  In particular, an oriented Seifert 3-manifold carries a positive, $S^1$-invariant transverse contact structure if and only if $e(Y) < 0$.  We use a negative Euler class, as in \cite{lm, lrbook}, to agree with Giroux's conventions for classifying tight contact structures transversal to the fibration of a circle bundle over a surface in terms of the Euler class of the fibration in \cite{giroux1}.
\end{remark}

 The classical Seifert fibering of $S^3$ along the torus knot $T(p,q)$, cf. \cite{ch-sfs, moser}, may be described as follows. It provides a partition of $S^3$ into orbits over the orbifold 2-sphere with the $z_1$ and $z_2$-axes given by singular fibers and each principal fiber given by a torus knot of type $T(p,q)$.  
  \begin{proposition}\label{prop:SeifertInvts}
The $S^1$-action $e^{2\pi it}\cdot(z_1,z_2)=\left(e^{2\pi pit}z_1,e^{2\pi qit}z_2\right),$
generates the Seifert fibration of $S^3$ given by ${Y\left(0; -1; (p, p-m), (q, n)\right),}$ 
where $m, n \in \Z$ such that $qm-pn=1$, and has Euler class ${e(Y)= -\frac{1}{pq}.}$
\end{proposition}

\subsection{Prequantization orbibundles}\label{ss:orbifolds}

We quickly define the notion of an orbifold and collect some calculations for weighted complex one dimensional projective space that will be used in our computation of the ECH index.  Further details can be found in \cite[\S 4]{bgbook}, which also describes complex and K\"ahler structures on orbifolds.

\begin{definition}
Let $O$ be a paracompact Hausdorff space.  An \emph{orbifold chart} or \emph{local uniformizing system} on $O$ is a triple $(\tilde{U}, \Gamma, \varphi)$ where 
\begin{itemize}
\itemsep-.35em
\item $\tilde{U}$ is a connected open subset of $\R^n$ containing the origin, 
\item $\Gamma$ is a finite group acting effectively on $\tilde{U}$,
\item $\varphi: \tilde{U} \to U$ is a continuous map onto an open set $U \subset O$ such that $\varphi \circ \zeta = \varphi$ for all $\zeta \in \Gamma$ and the induced natural map of $\tilde{U}/\Gamma$ onto $U$ is a homeomorphism.
\end{itemize}
An \emph{orbifold atlas} on $O$ is a family $\mathcal{U} = \{\tilde{U}_i, \Gamma_i, \varphi_i \}$ of orbifold charts subject to a compatibility condition on overlapping charts, cf. \cite[Def.~4.1.1]{bgbook}.   A \emph{smooth orbifold} is a paracompact Hausdorff space $O$ equipped with an equivalence class of orbifold atlases, which we denote by $\cO = (O,\mathcal{U})$. 

  \end{definition}

\begin{example}\label{symporb}
Let $\CP^1_{p,q}$ be the weighted complex one dimensional projective space defined by the quotient of the unit sphere $S^3 \subset C^2$ by the almost free action of $S^1 \subset \C$ of the form\footnote{Weighted projective space $\CP^1_{p,q}$ is actually an algebraic variety which admits two different orbifold structures; the other orbifold structure is realized by $\CP^1/ (\Z_{p} \times \Z_{q})$, cf.~\cite[\S 3.a.2]{mann}.  Typically one decorates the algebraic variety $\CP_{p,q}^1$ to denote that it is equipped with a specific orbifold structure, but we forgo this precision as we only consider one orbifold structure in this paper.}
\[
e^{{2\pi}it} \cdot (z_1,z_2) = \left(e^{p{2\pi}it}z_1,e^{q{2\pi}it}z_2 \right).
\]
The following 1-form is invariant under the above $S^1$-action
\begin{equation}
\lambda_{p,q} = \frac{\frac{1}{2\pi}\lambda_0}{p|z_1|^2 + q|z_2|^2}, \mbox{ where } \lambda_0 = \frac{i}{2} \sum_{j=1}^2 \left( z_j d\bar{z}_j - \bar{z}_jd z_j \right).
\end{equation}
Thus $\omega_{p,q} := d\lambda_{p,q}$ descends to an orbifold symplectic form on $\CP^1_{p,q}$.  By \cite[Lem.~4.2]{hong-cz}, its cohomology class satisfies ${[\omega_{p,q}] = \frac{1}{pq}} \mbox{ in }H^2(\CP^1_{p,q}, \Q)\cong \Q.$ (The orbifold chart that gives rise to this desired orbifold structure on $\CP^1_{p,q}$ is spelled out in \cite[Prop.~3.1]{mann}.)
\end{example}

\begin{definition}\label{def:orbX}
Following Boyer-Galicki  \cite[\S 4.3-4.4]{bgbook} we define the \emph{orbifold Euler characteristic} as
\[
\chi^{orb}(\cO) = \sum_S (-1)^{\dim(S)}\frac{\chi(S)}{|\Gamma(S)|}.
\]
where the sum is taken over all strata $S$ of the stratification of $\Sigma,$ $\chi(S)$ is the ordinary Euler characteristic, and $\Gamma(S)$ is the isotropy.
\end{definition}

\begin{example}\label{ex:orbX}
Let $(\Sigma; z_1,...,z_k)$ be a Riemann surface of genus $g$ with $k$ marked points.  We can give $(\Sigma; z_1,...,z_k)$ the structure of an orbifold by defining local uniformizing systems $(\tilde{U}_j, \Gamma_{a_j}, \varphi_j)$ centered at the point $z_j$, where $\Gamma_{a_j}$ is the cyclic group of order $a_j$ and $\varphi_j:\tilde{U}_j \to U_j = \tilde{U}_j/\Gamma_{a_j}$ is the branched covering map $\varphi_j(z)=z^{a_j}$. The orbifold Chern number\footnote{The rational orbifold first Chern class can be computed by way of the first rational Chern class of the canonical divisor.} agrees with the orbifold Euler characteristic:
\[
c_1^{orb}(\Sigma; z_1,...,z_k) = \chi^{orb}(\Sigma; z_1,...,z_k) = 2- 2g -k + \sum_{j=1}^k \frac{1}{a_j}
\]
It follows that for weighted projective space, 
\begin{equation}
{c_1^{orb}(\CP_{p,q}^1)=\frac{p+q}{pq}.}
\end{equation}

\end{example}

A \emph{principal $S^1$-orbibundle} or \emph{prequantization orbibundle} whose total space $Y$ is a manifold is the same as an almost free $S^1$-action on $Y$.  In this setting, the contact form is induced by the \emph{connection 1-form} $A$, which satisfies $iA(R) =1$ and $d(iA)(R, \cdot)=0$, where $R$ is the derivative of the $S^1$-action.  The \emph{curvature form} is the \emph{$S^1$-basic} 2-form $dA=i\mathfrak{p}^*\omega$, where $S^1$-basic means that the 2-form is $S^1$-invariant and vanishes in the the $S^1$-direction of $R$.  The basic cohomology class of $\omega$ can be canonically identified with an element in $H^2_{\mbox{\tiny \em orb}}(\cO;\R) $ via the equivariant de Rham theorem, cf. \cite[\S 3.2]{kl}.  The form $\omega$ being \emph{symplectic} means that it is a closed basic 2-form on $\cO$ satisfying $\omega^n \neq 0$.

We now review some results of Kegel and Lange, in particular their complete classification of closed Besse contact 3-manifolds up to strict contactomorphism in \cite{kl}.  In a strict contactomorphism, the Reeb dynamics are related by rescaling, which allows us to use either the prequantization orbibundle or open book description to compute various components of the ECH index and knot filtration.  

By \cite[Thm 1.4]{kl} we have that a Seifert fibration of a closed orientable 3-manifold $Y$ over an oriented orbifold $\cO$ can be realized by a Reeb flow if and only if the Euler class\footnote{Note that our Euler class is referred to as the real Euler class by Kegel-Lange.  They also use the opposite sign convention from us.}, as defined in Definition \ref{def:seifert}, of the fibration is nontrivial.  Combining \cite[Thm.~1.4]{kl} with \cite{cgm}, gives the following complete classification of Besse\footnote{A contact manifold $(Y, \lambda)$ is said to be \emph{Besse} whenever all its Reeb orbits are periodic, possibly with different periods. } contact 3-manifolds up to strict contactomorphism.   

\begin{theorem}{\em \cite[Cor.~1.6]{kl}}
The classification of closed Besse contact 3-manifolds up to strict contactomorphism coincides with the classification of Seifert fibrations $Y \to \cO$ of orientable, closed 3-manifolds with nonvanishing Euler class.
\end{theorem}

The Besse condition gives rise to the following orbifold Boothby Wang result, which is explicitly stated in the language of almost K\"ahler orbifolds and Sasakian geometry as \cite[Theorems 4.3.15, 6.3.8, 7.1.3, 7.1.6]{bgbook}.   However, it is more amenable to use the characterization   \cite[Theorems 1.1, 1.2]{kl} of Besse contact manifolds in terms of principal $S^1$-orbibundles over integral symplectic orbifolds satisfying a certain cohomological condition.  (In the below statement, we have restricted to dimension 3, renormalized the bundle so that the common period of the Reeb orbits is 1 instead of $2\pi$, and negated their Euler class convention. )

\begin{theorem}{\em \cite[Theorem 1.1]{kl} }
Let $(Y^{3},\lambda)$ be a Besse contact manifold.  Then after rescaling by a suitable constant, the Reeb flow of $\lambda$ has common period {$1$} and $\lambda$ is given by the connection 1-form of a corresponding principal $S^1$-orbibundle $\mathfrak{p}: M \to \cO$ over a symplectic orbifold $(\cO,\omega)$, with $\omega$ given by the the curvature form associated to $\lambda$ and ${-\frac{1}{2\pi}}[\omega]$ representing the  Euler class $e \in H^2_{\mbox{\tiny \em orb}}(\cO;\R)$ of the orbibundle.\end{theorem}

The converse construction of the above theorem is not needed for the purposes of this paper, but we direct the interested reader to \cite[Theorem 1.2]{kl}.  (In the setting where the base is smooth, prequantization bundles characterize \emph{Zoll} contact manifolds, which in addition to being Besse, also satisfy the requirement that all the periodic Reeb orbits have the same minimal period.)

\subsection{Morse-Smale functions for $T(2,q)$ fibrations}\label{ss:morse}

From Example \ref{symporb}, the Reeb vector field associated to $\lambda_{p,q}$ realizing the associated Seifert fiber space of Proposition \ref{prop:SeifertInvts} as a principal $S^1$-orbibundle over $\CP^1_{p,q}$ is given by
\[
R_{p,q} := {2\pi} 
i \left(p \left (z_1 \frac{\partial}{\partial z_1} - \bar{z}_1 \frac{\partial}{\partial \bar{z}_1}\right) + q \left (z_2 \frac{\partial}{\partial z_2} - \bar{z}_2 \frac{\partial}{\partial \bar{z}_2}\right)\right).
\]

A closed orbit $\gamma$ of a prequantization orbibundle is said to be a \emph{principal orbit} if it has the longest period among all the periodic orbits, provided one exists.   The orbits of the Reeb vector field $R_{p,q}$ that project to nonsingular points of the base $\CP^1_{p,q}$ are principal with action 1.  In particular, the $T(p,q)$ binding of the associated open book will be a principal orbit.

There will be two non-principal orbits of interest, which respectively project to each of the singular points of $\CP^1_{p,q}$. We will refer to these orbits as exceptional orbits; their actions are given by $1/|\Gamma_{x}|$ where $|\Gamma_{x}|$ is the {order} of the cyclic isotropy group at $x$. (Note that a nonsingular point has $|\Gamma_{x}|=1$.) The orbit $ ( e^{2\pi i t}, 0)$
 projects to the orbifold point  whose isotropy group is $ \Z/p\Z$.  Similarly, $(0, e^{2 \pi i t} )$ projects to the orbifold point whose isotropy group is $ \Z/q\Z$.

Next, we specialize to $p=2$, and construct the Morse-Smale functions ${H}_{2,q}$ on the orbifolds $\CP^1_{2,q}$, used to perturb the contact form
\begin{equation}\label{eqn:lambdapertR}
\lambda_{2,q,\varepsilon}:=  (1+\varepsilon \frak{p}^*H_{2,q})\lambda_{2,q},
\end{equation}
cf. Figure \ref{fig:morse23}. These Morse-Smale functions admit 3 critical points and is constructed from the punctured $2q$-gon (with opposite sides identified) representation of the page, which is a surface of genus $g=(q-1)/2$, familiar from the study of mapping class groups.  We obtain the right handed $2q$-periodic element of $\op{Mod}(\Sigma_g)$ by rotating the $2q$-gon by 1 click.

\begin{figure}[h]
 \begin{center}
 \begin{overpic}[width=.75\textwidth, unit=1.75mm]{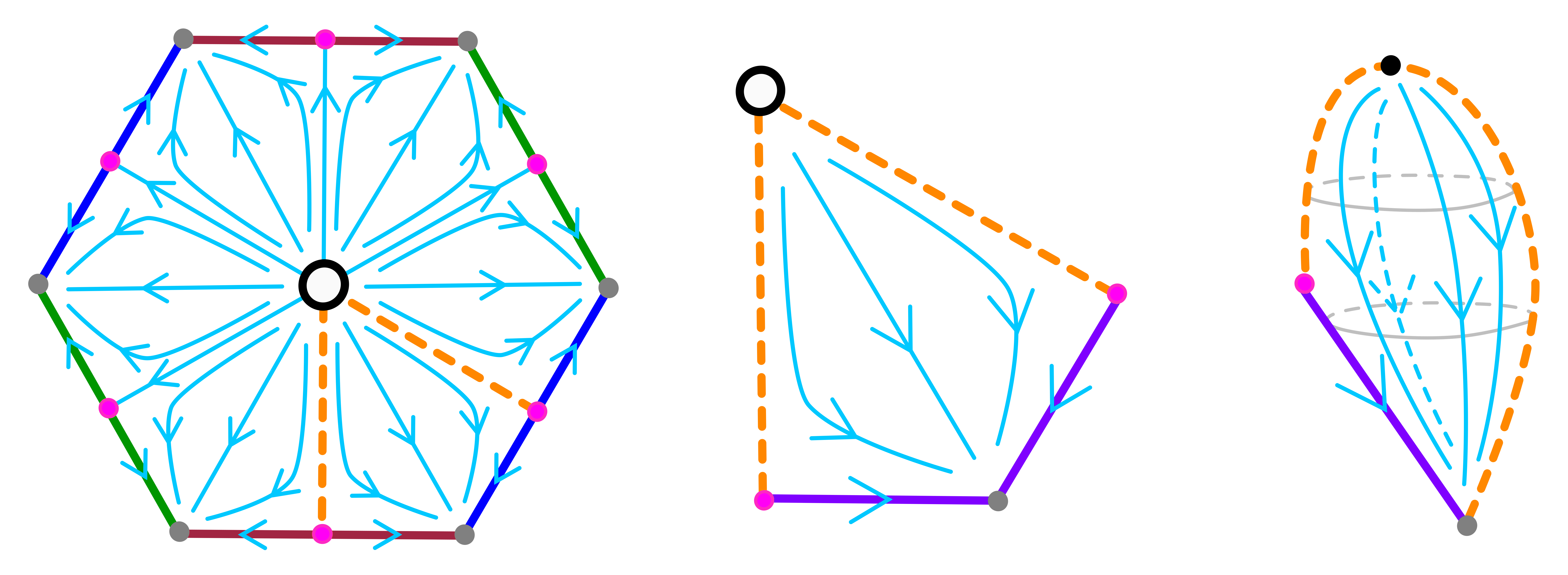}
\put(65.5,0){$e$}
\put(56,12.5){$h$}
\put(62.5,24){$b$}
\end{overpic}
\end{center}
\caption{(The construction and gradient flow of $H_{2,3}$ on $\CP^1_{2,3}$.) \\ \textbf{Left and Center:} On the left is the fundamental domain for the punctured torus, where edges are identified according to their color (with no twist).   The wedge is the fundamental domain for the action of the Reeb return map, which acts on  by a clockwise $\frac{2\pi}{6}$ rotation.   \\ \textbf{Right:} We depict $\CP^1_{2,3}$ obtained by gluing the center picture and closing up the puncture by collapsing it to a maximum (depicted by a black dot).  Gluing in the associated solid torus produces the binding  fiber $b$ of the open book. \\ \textbf{Linking of fibers:} When glued as indicated on the left, the gray dot appears $2=\ell(b,e)$ times while the pink dot appears $3=\ell(b,h)$ times.}
\label{fig:morse23}
\end{figure}

\begin{proposition}[Morse functions $H_{2,q}$]\label{prop:morse2}
Let $q$ be a fixed odd number.  There exists a Morse function $H_{2,q}$  on $\CP^1_{2,q}$, which is $C^2$ close to one, with exactly three critical points, {such that the binding projects to the nonsingular index 2 critical point}, the $\Z/2\Z$-isotropy point is the index 1 critical point, and the $\Z/q\Z$ isotropy point is the index 0 critical point. There are stereographic coordinates defined in a small neighborhood of $x\in X:=\op{Crit}(H_{2,q})$ in which $H_{2,q}$ takes the form  
\begin{enumerate}[\em (i)]
 \itemsep-.35em
    \item $r_0(u,v):=(u^2+v^2)/2-1$,\,\,\, if $x\in X_0$
    \item $r_1(u,v):=(v^2-u^2)/2$,\,\,\,\,\,\,\,\,\,\,\,\,\, if $x\in X_1$
        \item $r_2(u,v):=1-(u^2+v^2)/2$,\,\,\,\,\,if $x\in X_2$
\end{enumerate}
 {Moreover, $H_{2,q}$ is invariant under the the $2q$-periodic Reeb return map $\psi_{2,q}$ associated to $R_{2,q}$.}
\end{proposition}

\begin{proof}
This follows from the construction given in \cite[Lem. 2.16]{leo} via the following procedure.  Figure \ref{fig:morse23} provides a ``cartoon" realizing the construction of the desired Morse function.   First, we define a Morse function $\mathring{H}_{2,q}$ on the punctured genus $g= \frac{q-1}{2}$ surface $\mathring{\Sigma}_g$, which is realized as a $2q$-gon with opposite sides identified, with a puncture $\{ x_b \}$ at the center.\footnote{This polygon realizes the page of the open book decomposition of $\mathring{\Sigma}_g$ along $T(2,q)$. Rotation by 1 right handed click and subsequent identification of the $2q$-gon realizes the right handed $2q$-periodic element of the mapping class group of $\mathring{\Sigma}_g$.} We put a minimum at each vertex of the $2q$-gon and a saddle at each midpoint of the edge of the $2q$-gon;  after identification we obtain a single  minimum at a $\Z/q\Z$-isotropy point and a single saddle at a $\Z/2\Z$-isotropy point on $\CP^1_{2,q} \setminus \{ x_b \}$.  We can then extend $\mathring{H}_{2,q}$ to a Morse function on $\CP^1_{2,q}$ by ``completing" the flow over the puncture $\{ x_b \}$ with an index 2 critical point at $\{ x_b \}$. 
\end{proof}

\begin{remark}
When $p\neq 2$, a different class of orbifold Morse functions are needed as the previous construction cannot be used for elementary reasons: when $p=2$, the genus of $T(2,q)$ is $(q-1)/2$, and a genus $g$ surface can be represented by gluing opposite sides of a $(4g+2)$-gon. That $2q$ divides $4g+2=\frac{4(q-1)}{2}+2$ is reflected in the fact that the $2q$-periodic monodromy of the open book can be represented by rotating the polygon. 
\end{remark}

The following lemma guarantees that the Morse functions in Proposition \ref{prop:morse2}   are Smale with respect to $\omega_{2,q}(\cdot, j\cdot)$ restricted to $\CP^1_{2,q}$.

\begin{lemma}\label{lemma:smale}
If $H$ is a Morse function on a 2-dimensional {orbifold} $\Sigma$ with isolated quotient singularities such that $H(p_1)=H(p_2)$ for all $p_1, p_2\in\mbox{\em Crit}(H)$ with Morse index 1, then $H$ is Smale, given any metric on $S$.
\end{lemma}
\begin{proof}
Given metric $g$ on $\Sigma$, $H$ fails to be Smale with respect to $g$ if and only if there are two distinct critical points of $H$ of Morse index 1 that are connected by a gradient flow line of $H$. Because all such critical points have the same $H$ value, no such flow line exists.
\end{proof}

\subsection{Perturbed Reeb dynamics for $T(2,q)$ fibrations}\label{ss:Reeb}

Similarly to \cite{jo2, leo, preech}, we utilize the Morse functions $H_{2,q}$ to perturb the degenerate contact form $\lambda_{2,q}$, as in \eqref{eqn:lambdapertR}.  Note that the nonsingular point, to which the binding projects, realizing the right handed transverse torus knot $T(2,q)$, is always the unique maximum of the Morse functions.   A standard computation, cf. \cite[Prop.~4.10]{jo2}, adapted to the $S^1$-orbibundle framework, cf. \cite[Rem.~2.1]{LT}, yields the following results.

\begin{lemma}
The Reeb vector field of $\lambda_{2,q,\varepsilon}$ is given by
\begin{equation}
\label{perturbedreeb1}
R_{2,q,\varepsilon}=\frac{R_{2,q}}{1+\varepsilon \fp^*H_{2,q}} + \frac{\varepsilon \widetilde{X}_{H_{2,q}}}{{(1+\varepsilon \fp^*{H_{2,q}})}^{2}}, 
\end{equation}
where $\widetilde{X}_{H_{2,q}}$ denotes the horizontal lift of the Hamiltonian vector field\footnote{We use the convention $\omega(X_{{H}}, \cdot) = d{H}.$} $X_{H_{2,q}}$ on $\CP^1_{2,q}$.  
\end{lemma}

\begin{lemma}\label{lem:efromL}
For each odd $q$, let $H_{2,q}$ be the Morse function $H_{2,q}$ as constructed in Proposition \ref{prop:morse2}, which we further assume to be ${C^2}$ close to 1.  
\begin{enumerate}[\em (i)]
\itemsep-.25em
\item For each $L>0$, there exists $\varepsilon(L)>0$ such that for all $\varepsilon<\varepsilon(L)$, all Reeb orbits of $R_{2,q,\varepsilon}$ with $\mathcal{A}(\gamma) < L$ are nondegenerate and project to critical points of $H_{2,q}$, where $\mathcal{A}$ is computed using $\lambda_{2,q,\varepsilon}$.
\item The action of a Reeb orbit $\gamma_x^{k|\Gamma_x|}$ of $R_{2,q,\varepsilon}$ over a critical point $x$ of ${H_{2,q}}$ is proportional to the length of the fiber, namely
\[
\mathcal{A}\left(\gamma_x^{k|\Gamma_x|}\right) = \int_{\gamma_x^{k|\Gamma_x|}} \lambda_{2,q,\varepsilon} = k (1+\varepsilon \fp^*H_{2,q}(x)).
\]
\end{enumerate}
\end{lemma}

As in \cite[Lem. 3.1]{preech},  one can choose $\varepsilon(L)$ so that the embedded orbits realizing the generators of $ECC^L_*(S^3,\lambda_{2,q,\varepsilon(L)};J)$ consist only of fibers above critical points of $H_{2,q}$ and that $\varepsilon(L)\sim\frac{1}{L}$. 
To capture all these filtered complexes, we obtain the analogue of the result proven in  \cite[\S 3.4]{preech}.

\begin{proposition}\label{prop:directlimitcomputesfiberhomology} As discussed above, there is a direct system formed by $\{ ECH^L_*(Y,\lambda_{2,q,\varepsilon(L)}\}$. The direct limit $\lim_{L\to\infty}ECH^L_*(S^3,\lambda_{2,q,\varepsilon(L)})$ is the homology of the chain complex generated by Reeb currents $\{(\alpha_i,m_i)\}$ where the $\alpha_i$ are fibers above critical points of ${H_{2,q}}$.
\end{proposition}

Proposition \ref{prop:directlimitcomputesfiberhomology} provides a means of computing ECH by taking a direct limit, which involves passing to filtered Seiberg-Witten Floer cohomology explained in \cite[\S 7]{preech}.  The Conley-Zehnder computations in \S \ref{ss:CZ} yield the following classification of fiber Reeb orbits:

\begin{lemma}\label{lem:orbitseh} Up to large action $L(\varepsilon)$, as determined by Proposition \ref{lem:efromL}, the generators of $ECC_*^{L(\varepsilon)}(S^3,\lambda_{2,q,\varepsilon}, J)$ are of the form $b^Bh^He^E$, where $B,E \in \Z_{\geq 0}$ and $H=0,1$. Moreover,
\begin{itemize}
\itemsep-.25em
\item $b$ is elliptic projecting to the (nonsingular) index 2 critical point of $H_{2,q}$,
\item $h$ is negative hyperbolic projecting to the singular index 1 critical point of $H_{2,q}$ with isotropy $\Z/2\Z$,
\item $e$ is elliptic projecting to the singular index 0 critical point of $H_{2,q}$ with isotropy $\Z/q\Z$.
\end{itemize}
\end{lemma}

\begin{remark}\label{rem:obd}
The contact forms $\lambda_{2,q, \varepsilon}$ and $\lambda_{2,q}$ are  adapted to the open book decomposition $(T(2,q),\pi)$.  This is because $\widetilde X_{H_{2,q}}$ is tangent to the pages of $\pi$, so by (\ref{perturbedreeb1}), $R_{p,q,\varepsilon}$ is positively transverse to the pages while remaining tangent to the binding because $\widetilde X_{H_{2,q}}$ vanishes along the binding.
\end{remark}

From the above discussion, we obtain the following linking numbers relevant to our later computation of the knot filtration on ECH. 

\begin{corollary}\label{cor:linkingnumber} 
For $T(2,q)$, we have:
\begin{enumerate}[\em (i)]
\itemsep-.35em
\item $\ell(e,h)=1$,
\item $\ell(b,e)=2$,
\item $\ell(b,h)=q$.
\end{enumerate}
\end{corollary}
\begin{proof}
Conclusion (i) follows from the fact that the open book decomposition $(T(2,q),\pi)$ of Remark \ref{rmrk:Milnor} can be expressed via an $S^1$-action which is free except at the intersections of $S^3$ with the axes in $\C^2$. Thus $S^3$ is the union of two solid tori whose cores are $e$ and $h$, which link once.

Conclusions (ii, iii)  follow from the fact that $\fp(e)$ and $\fp(h)$ are periodic points of the monodromy of the open book $(T(2,q),\pi)$ of periods $2$ and $q$, respectively (the periods can be computed from Remark \ref{cor:CHdetails}(i, ii)). Their linking numbers with the binding equal their intersection numbers with the page, which are precisely these periods. 

\end{proof}

\section{Generalities regarding the ECH index}\label{s:generalities}

 In this section we provide some generalities on the ECH index, construct surfaces, and explain  different trivializations of $\xi_{std}$ along iterates of Reeb orbits which project to critical points of the Morse-Smale functions {on} $\CP^1_{p,q}$.   We also provide a change of trivialization formula {for these iterated Reeb orbits} and compute it in the desired settings. Our computation of the ECH index, completed in \S \ref{s:ECHI}, will make use of three different trivializations.  These are the constant trivialization, orbibundle trivialization, and page trivialization.    The ``orbibundle" trivialization provides a global trivialization of $\xi$, which makes computing the relative first Chern numbers and Conley-Zehnder index relatively straightforward, which we carry out in \S \ref{ss:chern} and \S \ref{ss:CZ}, respectively.  Understanding the relative intersection pairing in \S \ref{ss:Q} is more involved and requires a variety of different trivializations and surfaces.

\subsection{Properties of the ECH index}\label{ss:ECHI}

In this section we recall the necessary facts about the ECH index (and clarify some notation), and we prove Proposition \ref{prop:nonsimpleCOT}, which allows us to adjust the terms in the ECH index when using covers of simple Reeb orbits. We first define the ECH index.  To learn more about the wonders of the ECH index see \cite[\S 2]{Hrevisit}.

The definition of the ECH index depends on three components: the relative first Chern number $c_\tau$, which detects the contact topology of the curves; the relative intersection pairing $Q_\tau$, which detects the algebraic topology of the curves; and the Conley-Zehnder terms, which detect the contact geometry of the Reeb orbits. {For Reeb currents $\alpha$ and $\beta$ on $Y$, the set $H_2(Y,\alpha,\beta)$ denotes the 2-chains $Z$ with $\partial Z=\alpha-\beta$, modulo boundaries of 3-chains. }  If $Z \in H_2(Y,\alpha,\beta)$ and $\tau$ is an appropriate trivialization of $\xi$ over the Reeb orbits $\{\alpha_i\}$ and $\{\beta_j\}$, which is symplectic with respect to $d\lambda$, {we define the ECH index as follows:} 
\begin{definition}[ECH index]\label{defn:ECHI} 
Let $\alpha=\{(\alpha_i,m_i)\}$ and $\beta=\{(\beta_j,n_j)\}$ be Reeb currents  in the same homology class, $\sum_i m_i [\alpha_i]=\sum_j n_j [\beta_j]= \Gamma\in H_1(Y).$   Given $Z \in H_2(Y,\alpha,\beta)$, we define the \emph{ECH index} to be
\[
I(\alpha,\beta,Z) =  c_\tau(Z) + Q_\tau(Z) +  CZ^I_\tau(\alpha) -CZ^I_\tau(\beta),
\]
where $CZ^I(\gamma): =  \sum_i \sum_{k=1}^{m_i}CZ_\tau(\gamma_i^k)$.
When $\alpha$ and $\beta$ are clear from context, we use the notation $I(Z)$, and when $\beta=\emptyset$ {and $Z$ is clear from context} we use the notation $I(\alpha)$. The \emph{relative first Chern number} $c_\tau$ is defined in \S\ref{ss:chern}, the \emph{relative intersection pairing} $Q_\tau$ is defined in \S\ref{ss:Q}, and the \emph{Conley-Zehnder index} $CZ_\tau$ is defined in \S \ref{ss:CZ}.
\end{definition}

\begin{remark}
The first Chern number is linear in the multiplicities of the Reeb currents and the relative intersection term is quadratic (see Lemma \ref{lem:formulas}).   The ``total Conley-Zehnder" index term $CZ_\tau^I$ behaves in a complicated way with respect to the multiplicities, depending on the trivialization $\tau$, but is generally quadratic unless a very special trivialization is used (because it is a sum of linear terms, see (\ref{CZ:changetriv})). 
\end{remark}

We next recall the following general properties of the ECH index, cf. \cite[\S 3.3]{Hindex}.

\begin{theorem}[{\cite[Prop.~1.6]{Hindex}}]\label{thm:Iproperties}The ECH index has the following basic properties: \hfill
\begin{enumerate}[{\em (i)}]
\itemsep-.25em
\item {\em(Well Defined)} The ECH index $I(Z)$ is independent of the choice of trivialization.
\item\label{property:indexamb} {\em (Index Ambiguity Formula) }If $Z' \in H_2(Y,\alpha,\beta)$ is another relative homology class, {and $Z-Z'$ is as defined in Remark \ref{rmk:H2}(ii),} then
\[
I(Z)-I(Z')= \langle Z-Z', c_1(\xi) + 2 \mbox{\em PD}(\Gamma) \rangle.
\]
\item {\em (Additivity) } If $\delta$ is another orbit set in the homology class $\Gamma$, and if $W \in H_2(Y,\beta,\delta)$, then $Z+W \in  H_2(Y,\alpha,\delta)$ is defined as in Remark \ref{rmk:H2}(iii) and 
\[
I(Z+W)=I(Z)+I(W).
\]
\item {\em (Index Parity) }If $\alpha$ and $\beta$ are generators of the ECH chain complex (in particular all hyperbolic orbits have multiplicity 1), then $(-1)^{I(Z)} = \varepsilon(\alpha)\varepsilon(\beta),$ where $\varepsilon(\alpha)$ denotes $-1$ to the number of positive hyperbolic orbits in $\alpha$.
\end{enumerate}
\end{theorem}

Next we collect facts about the components of the ECH index, upon which we will rely heavily.  {We first fix some notation for trivializations, used throughout this section and the next. Given a nondegenerate Reeb orbit $\gamma:\R/T\Z\to Y$,  denote the set of homotopy classes of symplectic trivializations of the 2-plane bundle $\gamma^*\xi$ over $S^1=\R/T\Z$ by $\mathcal{T}(\gamma)$.   After fixing trivializations $\tau_i^+ \in \mathcal{T}(\alpha_i)$ for each $i$ and $\tau_j^- \in \mathcal{T}(\beta_j)$, we denote this set of trivialization choices by $\tau \in \mathcal{T}(\alpha,\beta)$.   The trivialization $\tau$ determines a trivialization of $\xi|_C$ over the ends of $C$ up to homotopy.  }  As spelled out in \cite[Rem.~2.4]{preech}, we use the sign convention that if $\tau_1, \ \tau_2: \gamma^*\xi \to S^1 \times \R^2$ are two trivializations then
\begin{equation}\label{eqn:trivchange}
\tau_1 - \tau_2 = \mbox{deg} \left(\tau_2 \circ \tau_1^{-1} : S^1 \to \mbox{Sp}(2,\R) \cong S^1\right).
\end{equation}

\begin{lemma}\label{lem:formulas} Let $Z, Z_1, Z_2\in H_2(Y,\alpha,\beta)$, $Z'\in H_2(Y,\alpha',\beta')$, and $W\in H_2(Y,\beta,\delta)$. We have:
\begin{enumerate}[{\em (i)}]
\itemsep-.25em
\item \emph{(Dependence on $Z$: \cite[(5), Lem.~2.5~(a)]{Hindex})} The relative first Chern number depends only on $\alpha, \beta, \tau,$ and $Z$, that is,
\begin{equation}
c_\tau(Z_1) - c_\tau(Z_2) = \langle c_1(\xi), Z_1 - Z_2 \rangle.
\end{equation}
Similarly, the relative intersection pairing depends only on $\alpha,\beta,\tau, Z$, and $Z'$, that is,
\[
Q_\tau(Z_1,Z')-Q_\tau(Z_2,Z')=[Z_1-Z_2]\cdot\Gamma.
\]

\item\label{item:clin} \emph{(Linearity with respect to concatenation)} Using trivializations which agree over the orbits in $\beta$, the relative first Chern number and relative intersection pairing are linear with respect to concatenation addition\footnote{See Remark \ref{rmk:H2}(iii) for a precise definition.}:
\[
\begin{split}
c_\tau(Z+W)=  & \ c_\tau(Z)+c_\tau(W), \\
Q_\tau(Z+W,Z'+W')= & \  Q_\tau(Z,Z')+Q_\tau(W,W').
\end{split}
\]

\item\label{item:linbil} \emph{(Linearity/bilinearity with respect to union: \cite[(3.11)]{Hrevisit})} The relative first Chern number is linear with respect to union addition\footnote{See Remark \ref{rmk:H2}(iv) for a precise definition.},
\[
c_\tau(Z\uplus Z')=c_\tau(Z)+c_\tau(Z'),
\]
while the relative intersection pairing is symmetric, and it is bilinear with respect to union addition,
\begin{align*}
Q_\tau(Z,Z')&=Q_\tau(Z',Z)
\\Q_\tau(Z\uplus W,Z')&=Q_\tau(Z,Z')+Q_\tau(W,Z').
\end{align*}

\item\label{item:cot} \emph{(Change of trivialization: \cite[(6), Lem.~2.5~(b)]{Hindex} and \cite[(2.12)]{Hrevisit})} Given another collection of trivialization choices $\tilde\tau = \left( \{ {\tilde\tau}_i^+ \}, \{ {\tilde\tau}_j^-\} \right) \in \mathcal{T}(\alpha,\beta)$ over the Reeb currents $\alpha = \{ (\alpha_i,m_i)\}$ and $\beta = \{ (\beta_j,n_j)\}$, we have
\begin{align}
c_\tau(Z) - c_{\tilde\tau}(Z) &= \sum_i m_i\left(\tau_i^+-{\tilde\tau}_i^+\right) - \sum_j n_j \left(\tau_j^--{\tilde\tau}_j^-\right), \label{cherntriv} \\
Q_\tau(Z,Z') - Q_{\tilde\tau}(Z,Z') & = \sum_i m_i m_i' (\tau_i^+ - \tilde\tau_i^+) - \sum_i n_j n_j' (\tau_i^- - \tilde\tau_i^-), \label{Qtriv}\\
CZ_\tau^I(\alpha) - CZ_{\tilde\tau}^I(\alpha) & = \sum_i (m_i^2+m_i)(\tilde \tau_i - \tau_i). \label{CZItriv} 
\end{align}
\end{enumerate}
\end{lemma}

In Lemma \ref{lem:formulas}, the properties that are given without citation (the linearity properties) follow immediately from the definitions of $c_\tau$ and $Q_\tau$.

\begin{remark}\label{rmk:formulas}
\begin{enumerate}[{ (i)}]
\itemsep-.25em
\item {When $Y$ is $S^3$}, Lemma \ref{lem:formulas} (i) will not be used because each $H_2(S^3,\alpha,\beta)$ contains only one element, as in Remark \ref{rmk:H2} (i).

\item A consequence of (\ref{item:linbil}) is that for union addition
\[
Q_\tau(Z\uplus Z',Z\uplus Z')=Q_\tau(Z)+2Q_\tau(Z,Z')+Q_\tau(Z'),
\]
thus $Q_\tau(mZ)=m^2Q_\tau(Z).$

\item In Lemma \ref{lem:formulas} (iii) it may be the case that $\gamma_i$ does not appear in $\alpha'$; in this case $m_i'=0$, and similarly when $\gamma_j$ does not appear in $\beta'$ then $n_j'=0$. The trivialization $\tau$ is a trivialization of $\xi$ over all Reeb orbits in the sets $\alpha, \alpha', \beta,$ and $\beta'$.  As a consequence, when $\alpha$ and $\alpha'$ share no underlying orbits and $\beta$ and $\beta'$ share no underlying orbits then $Q_\tau(Z,Z')$ is independent of $\tau$. 
\end{enumerate}
\end{remark}

\subsection{Trivializations and surfaces for $T(2,q)$ fibrations}\label{ss:trivs} 

The Reeb orbit $b$ realizing the binding is a regular fiber, so there are three trivializations which can be used: the constant $\tau_0$, page $\tau_\Sigma$, and orbibundle $\tau_{orb}$ trivialization.  The constant trivialization is not available for fibers which project to orbifold points.   The {constant} trivialization can be geometrically ``extended" from the binding $b$ to obtain trivializations along $h^2$ and $e^q$. {See Remarks \ref{rmk:0eh}-\ref{rmk:T2q0Sig}.}   

Since some trivializations are easier to work with than others, we establish some change in trivialization formulas to relate them.  Additional details on how various components of the ECH index are impacted by changes of trivialization are given in \S \ref{ss:ECHI}. 

\begin{remark} In our computations of the relative first Chern number, Conley-Zehnder index, and relative self intersection pairing, we minimize the number of subscripts used distinguish trivializations by suppressing $\tau$. This means we use the notation
\[
c_0:=c_{\tau_0}, \quad Q_0:=Q_{\tau_0}, \quad CZ_0:=CZ_{\tau_0},
\]
and the obvious analogues for other trivializations.
\end{remark}

We now define the trivializations which we will use throughout our computations.  We do this via surfaces, and thus first note several facts about $H_2(Y,\alpha,\beta)$ which will help us describe the classes in which those surfaces live.
\begin{remark}\label{rmk:H2}

{Relative} homology classes in $H_2(Y,\alpha,\beta)$ admit the following properties.
\begin{enumerate}[{(i)}]
\itemsep-.25em
\item Since $H_2(Y,\alpha,\beta)$ is affine over $H_2(Y)$, when $Y=S^3$ or $L(k,\ell)$ it contains only one element. When $\beta=\emptyset$ we denote this element by $Z_\alpha$.
\item Since $H_2(Y,\alpha,\beta)$ is affine over $H_2(Y)$, if $Z,Z'\in H_2(Y,\alpha,\beta)$ then $Z-Z'$ denotes their difference as an element of $H_2(Y)$.
\item Let $\delta$ be another Reeb current in the homology class $[\alpha]=[\beta]$. If $Z\in H_2(Y,\alpha,\beta)$ and $W\in H_2(Y,\beta,\delta)$, we can define their sum
\[
Z+W\in H_2(Y,\alpha,\delta),
\]
by gluing representatives along $\beta$. We will refer to this as ``\emph{concatenation addition}'' and  denote it by $+$.
\item When $Z\in H_2(Y,\alpha,\beta)$ and $Z'\in H_2(Y,\alpha',\beta')$, we can define their sum
\[
Z+Z'\in H_2(Y,\alpha\alpha',\beta\beta'),
\]
where concatenation of Reeb currents denotes union of the underlying orbits with corresponding multiplicities added. We will refer to this as ``\emph{union addition}'' and denote it by $\uplus$. Furthermore, we will use the notation
\[
mZ:=\underbrace{Z\uplus\cdots\uplus Z}_{m}.
\]
\end{enumerate}
\end{remark}

We now describe the surface representatives of the elements of relevant $H_2(S^3,\alpha,\beta)$ sets. We use $[S]$ to denote the equivalence class $Z$ in $H_2(Y,\alpha,\beta)$ of a surface $S$ in $Y$ with boundary on $\alpha-\beta$.

\begin{definition}[Surfaces for $T(2,q)$]\label{def:sfcs2} 

We define the following three surfaces in $S^3$, where $\cup_b$ means that we are attaching two surfaces along their common boundary $b$: 

\begin{itemize}
\itemsep-.25em
\item The surface $\Sigma$, which has one boundary component and genus $(q-1)/2$, is the page of the open book decomposition discussed in \S\ref{ss:OBDs}. Thus $\partial\Sigma=b=T(2,q)$ and $[\Sigma]=Z_b$.

\item The surface $S_e$ is the preimage under $\fp$ of a short ray connecting $\fp(e)$ to a regular fiber $b$. Thus $\partial S_e=e^q-b$ and $[S_e\cup_b\Sigma]=[S_e]+[\Sigma]=Z_{e^q}=qZ_e$. \item The surface $S_h$ is the preimage under $\fp$ of a short ray connecting $\fp(h)$ to a regular fiber $b$. Thus $\partial S_h={h^2}-b$ and $[S_h\cup_b\Sigma]=[S_h]+[\Sigma]=Z_{{h^2}}={2}Z_h$.

\end{itemize}

\end{definition}

Assume $\gamma\in\partial S$. A trivialization $\tau$ over an orbit $\gamma$ ``has linking number zero with respect to $S$" or ``is the $S$-trivialization" if the pushoff of $\gamma$ into $S$ is considered to be constant with respect to $\tau$; see Remarks \ref{rem:linkingtriv} and \ref{rem:c1sl}. We denote such a trivialization by $\tau_S$.

\begin{remark}[Trivializations for $T(2,q)$]\label{rmk:tau2q}
The trivializations we will use are:
\begin{itemize}
\itemsep-.25em
\item The page trivialization $\tau_\Sigma$ over $b$,

\item The \emph{constant trivialization} $\tau_0$ over $b$, {which} has two related surface trivializations (see Remark \ref{rmk:0eh}):
\begin{itemize}
\itemsep-.25em
\item[*] $\tau_e:=\tau_{S_e}$ over $b$ and $e^q$,
\item[*] $\tau_h:=\tau_{S_h}$ over $b$ and ${h^2}$.
\end{itemize}

\item The \emph{orbibundle trivialization} $\tau_{orb}$ over $b, e$, and $h$; this can be used as a black box and is a pullback trivialization as explained and utilized in \cite[\S 3]{hong-cz} for Conley-Zehnder computations of fibers in terms of the orbifold Chern class of the base $\CP^1_{2,q}$; cf. \S\ref{ss:orbitau}.
\end{itemize}
{We will use $\tau(\gamma)$ to indicate the orbit to which we restrict the trivialization.}
\end{remark}

See \S\ref{ss:consttau} for details on the constant trivialization, including the computation of the Conley-Zehnder index. It has the following topological relationships to the surfaces. 
\begin{definition}\label{def:consttriv}
The \emph{constant trivialization} $\tau_0$ can be defined for any orbit $\gamma$ which is a fiber of a prequantization bundle (as in \cite{preech}) or regular fiber of a prequantization orbibundle $\fp:Y\to\Sigma$. It is the trivialization in which the unperturbed linearized Reeb flow is the identity, meaning an identification of $\xi|_\gamma$ with $T_{\fp(\gamma)}\Sigma\times S^1$, which is homotopic to the trivialization which restricts to $\fp_*$ on each contact plane.
\end{definition}

\begin{remark}\label{rmk:0eh} In Definition \ref{def:consttriv}, the condition that the linearized Reeb flow be the identity makes sense over any fiber orbit whose neighbors are all regular fibers as well, that is, even over exceptional fibers. The only necessary adjustment is that we take a cover of the orbit whose covering multiplicity is a multiple of the corresponding orbifold point's isotropy. However, it is simpler to define these as the trivializations having linking number zero with respect to a surface that is a union of nearby fibers, such as $S_e$ or $S_h$. One of the boundaries of this surface will be a cover of an exceptional fiber whose order is that fiber's orbifold point's isotropy. For example, in the case of $S_e$, it will be a $q$-fold cover, and in the case of $S_h$, a 2-fold cover. For a visualization of such a surface, see the mesh surface in \cite[Fig.~3]{weiler}. 
\end{remark}
\begin{remark}[$T(2,q)$ constant and surface trivializations]\label{rmk:T2q0Sig}
Lemma \ref{lem:taudiff2}(i) shows that $\tau_0(b)=\tau_e(b)=\tau_h(b)$. Moreover, we can think of the trivializations $\tau_e$ and $\tau_h$ as extensions of the constant trivialization over $e^q$ and $h^2$. This heuristic is supported by computing the Conley-Zehnder indices. Combining Proposition \ref{prop:nonsimpleCOT}, Lemma \ref{lem:orbtrivCZ2}, and Lemma \ref{lem:e3h2diff} shows that
\begin{align*}
CZ_e(e^{qk})&=CZ_{orb}(e^{qk})+2k(\tau_{orb}(e^q)-\tau_e(e^q))=2(2+q)k-1+2k(-2-q)=-1;
\\CZ_h({h^{2k}})&=CZ_{orb}({h^{2k}})+2k(\tau_{orb}({h^2})-\tau_h({h^2}))=2({2+q})k+2k({-2-q})=0.
\end{align*}
These are the values taken if they were regular fibers, analogous to \cite[Lem.~3.9]{preech}.
\end{remark}

\subsection{General change of trivialization formulae}

We now review and compute some changes of trivialization.  There are many natural trivializations only defined over covers of simple orbits: see \cite[\S5.1]{weiler}, where trivializations of contact structures over rationally (but not necessarily integrally) nullhomologous knots are used to define the knot filtration when $b_1=0$ but $H_1\neq0$. If $\gamma$ is an embedded orbit, then for every trivialization of $\xi$ over $\gamma^k$ which does arise as the $k$-fold cover of a trivialization of $\xi$ over $\gamma$, there will be $k-1$ trivializations which do not.

For example, in the $T(2,q)$ setting, our trivializations $\tau_e$ and $\tau_h$ over $e^q$ and $h^2$, respectively (see Remark \ref{rmk:tau2q}), are of this type, while $\tau_{orb}$ is a trivialization over the underlying embedded orbits $e$ and $h$. Since it is much easier to find $\tau$-representatives when $\tau=\tau_e, \tau_h$ rather than when $\tau=\tau_{orb}$, we require change-of-trivialization formulas for trivializations over covers of simple orbits in order to prove Lemma \ref{lem:QeQh}, which computes the relative self-intersection numbers of $e$ and $h$.

\begin{proposition}\label{prop:nonsimpleCOT} Modifications of the formulas in Lemma \ref{lem:formulas} and (\ref{CZ:changetriv}) hold for trivializations over covers of embedded orbits. Specifically, assume
\begin{itemize}
\itemsep-.25em
\item $\gamma$ is an embedded orbit,
\item $\tau(\gamma^k)$ and $\tilde\tau(\gamma^k)$ are trivializations over a cover $\gamma^k$ of a simple orbit $\gamma$,
\item $\alpha=\{(m_i,\gamma_i)\}\cup\{(mk,\gamma)\}$ and $\alpha'=\{(m_i',\gamma_i')\}\cup\{(m'k,\gamma)\}$, {where no $\gamma_i$ or $\gamma_i'$ is $\gamma$,}
\item $\tau$ and $\tilde\tau$ extend to elements of $\mathcal{T}(\alpha,\beta)$, and
\item $Z\in H_2(Y,\alpha,\beta)$ and $Z'\in H_2(Y,\alpha',\beta')$.
\end{itemize}
Then 
\begin{enumerate}[{\em (i)}]
\itemsep-.25em
\item $c_\tau(Z)-c_{\tilde\tau}(Z)=m\left(\tau(\gamma^k)-\tilde\tau(\gamma^k)\right)+\sum_im_i\left(\tau^+_i-\tilde\tau^+_i\right)-\sum_jn_j\left(\tau^-_j-\tilde\tau^-_j\right)$
\item $Q_\tau(Z,Z')-Q_{\tilde\tau}(Z,Z')=mm'k\left(\tau(\gamma^k)-\tilde\tau(\gamma^k)\right)+\sum_im_im'_i\left(\tau^+_i-\tilde\tau^+_i\right)-\sum_jn_jn'_j\left(\tau^-_j-\tilde\tau^-_j\right)$
\item $CZ_\tau(\gamma^{mk})-CZ_{\tilde\tau}(\gamma^{m'k})=2m(\tilde\tau(\gamma^k)-\tau(\gamma^k))$
\end{enumerate}
\end{proposition}

{Note that the proof of Proposition \ref{prop:nonsimpleCOT} relies on the definitions of $c_\tau, Q_\tau$, and $CZ_\tau$ in \S\ref{ss:chern}-\ref{ss:Q}. We include it here because of the similarities between the result and those in \S \ref{ss:ECHI} and because it is motivated by the trivializations defined in \S\ref{ss:trivs}.}

\begin{proof}
The proofs of (i) and (iii) are the same as in the case of a simple orbit. For (i), replace the curve bounded by $\gamma$ with one bounded by $\gamma^k$; for (iii), as the Conley-Zehnder index is defined for any Reeb orbit, simply replace the Reeb orbit $\gamma$ with $\gamma^k$.

To prove (ii), let $S$ and $S'$ be admissible representatives of $Z$ and $Z'$, respectively. It is enough to consider a single pair of ends, on each of $S, S'$, on $\gamma^k$ (the coefficient $mm'$ arises from the number of ways to match these ends of $S$ and $S'$), determining braids $\zeta$ and $\zeta'$. Our goal is to show
\[
\ell_{\tilde\tau}(\zeta,\zeta')=\ell_\tau(\zeta,\zeta')+k(\tau(\gamma^k)-\tilde\tau(\gamma^k)).
\]

In the tubular neighborhood of $\gamma$ identified by $\tau$ with $S^1\times\D^2$, arrange $\zeta$ and $\zeta'$ so that there is an interval (in the $S^1$ factor) on which the $k$ strands of $\zeta$ project to a set $\{r_1,0\},\dots,\{r_k,0\}$ of points in $\D^2$, the $k$ strands of $\zeta'$ project to $\{(r_{k+1},0),\dots,(r_{2k},0)\}$, and $r_1<\cdots<r_{2k}$. The effect of changing the trivialization from $\tau$ to $\tilde\tau$ can be expressed by adding $\tau(\gamma^k)-\tilde\tau(\gamma^k)$ copies of the meridian of $S^1\times\D^2$ to the strands projecting to both $(r_k,0)$ and $(r_{k+1},0)$. This adds exactly $k(\tau(\gamma^k)-\tilde\tau(\gamma^k))$ to the linking number of $\zeta$ and $\zeta'$, because the strand projecting to $(r_{k+1},0)$ now links $\tau(\gamma^k)-\tilde\tau(\gamma^k)$ times with each strand of $\zeta$ (including the one projecting to $(r_k,0)$, which has also been twisted). We have thus proved (ii).
\end{proof}

  We now collect the various change in trivialization formulas. Some proofs require the relative first Chern number and Conley-Zehnder computations carried out in \S \ref{ss:chern}-\ref{ss:CZ} as a black box.

\subsection{Changes of trivializations for $T(2,q)$ fibrations}\label{ss:cotTpq}
We now collect the change in trivialization formulas for the binding.  

\begin{lemma}\label{lem:taudiff2} For trivializations defined along the binding orbit $b$ realizing $T(2,q)$, we have:
\begin{enumerate}[{\em (i)}]
\itemsep-.25em
\item $\tau_e(b)=\tau_h(b)=\tau_0(b)$,
\item $\tau_0(b)-\tau_\Sigma(b)=2q$,
\item $\tau_0(b)-\tau_{orb}(b)=2+q$, 
\item $\tau_{orb}(b)-\tau_\Sigma(b)=q-2$.
\end{enumerate}
\end{lemma}
\begin{proof}
\begin{enumerate}[{ (i)}]
\itemsep-.25em
\item In all these trivializations, a nearby fiber is a constant pushoff. Since the homotopy class of a constant pushoff determines the trivialization, they agree.

\item The homotopy equivalence $\operatorname{Sp}(2,\R)\approx S^1$ sends the rotation matrix by angle $\theta$ to $\theta\in S^1=\R/2\pi\Z$. Thus it suffices to show that $\tau_0\circ\tau_\Sigma^{-1}:\R/\Z\to\operatorname{Sp}(2,\R)$ is homotopy equivalent to the map sending $t$ to rotation by $-4q\pi t$.

In solid torus coordinates $S^1_t\times\mathbb{D}^2_{r,\theta}$ near $b=(t,0,0)$, because the fractional Dehn twist coefficient of the open book decomposition is $1/2q$, a nearby fiber can be parametrized as $(t,r,4q\pi t)$ so that it wraps positively $2q$ times around the binding $b$. The pullback $b^*\xi$ can be identified with $T\mathbb{D}^2\cong\R^2$, and under this identification the trivialization $\tau_\Sigma$ is simply the identity, as is its inverse. Thus the trivialization $\tau_0$, which sends $(t,r,4q\pi t)$ to the curve $(t,r,0)$ in $S^1\times\R^2$ which does not wrap at all around the central fiber $S^1\times\{(0,0)\}$, must subtract $4q\pi t$ from the $\theta$ coordinate, or in other words, rotate $2q$ times negatively around the central fiber. Thus $\tau_0-\tau_\Sigma=-(\tau_\Sigma-\tau_0)=2q$.

\item Using the change-of-trivialization formula (\ref{CZ:changetriv}) for the Conley-Zehnder index referenced above, Corollary \ref{cor:CZ0b}, and Lemma \ref{lem:orbtrivCZ2} computing $CZ_\tau(b)$ when $\tau=\tau_0, \tau_{orb}$ respectively, we obtain
\[
2(\tau_0(b)-\tau_{orb}(b))=CZ_{orb}(b)-CZ_0(b)=2(2+q)+1-1.
\]

\item Using conclusions (ii) and (iii), we obtain
\[
\tau_{orb}(b)-\tau_\Sigma(b)=\left(\tau_0(b)-\tau_\Sigma(b)\right)+\left(\tau_{orb}(b)-\tau_0(b)\right)=2q-2-q.
\]
\end{enumerate}
\end{proof}

We now collect the change of trivialization formulas along the exceptional fibers.  First we consider $T(2,q)$.

\begin{lemma}\label{lem:e3h2diff} 
The change of trivialization formulae along $e^q$ and $h^2$ are
\begin{enumerate}[{\em (i)}]
\itemsep-.25em
\item $\tau_e(e^q)-\tau_{orb}(e^q)=2+q$,
\item $\tau_h(h^2)-\tau_{orb}(h^2)=2+q$.
\end{enumerate}
\end{lemma}
\begin{proof}

Using the change-of-trivialization formula for the relative first Chern number concerning trivializations over multiply covered orbits (Proposition \ref{prop:nonsimpleCOT}(i) mentioned above), we have
\[
\tau_e(e^q)-\tau_{orb}(e^q)=c_e(Z_{e^q})-c_{orb}(Z_{e^q})=2+q-0.
\]
Lemma \ref{lem:relfirstCherncalc2}(iv) computes $c_e(Z_{e^q})$, while $c_{orb}(Z_{e^q})=qc_{orb}(Z_e)=0$ by the linearity formula for the relative first Chern number with respect to union addition explained in Lemma \ref{lem:formulas} (\ref{item:linbil}). This proves (i); the proof of (ii) is similar.
\end{proof}

\section{Components of the ECH index}\label{s:components}
 In this section we compute the components of the embedded contact homology index for the $T(2,q)$ fibrations of $S^3$.  First we compute relative first Chern numbers, then the Conley-Zehnder indices, and finally  the relative intersection pairing.   
 
\subsection{Relative first Chern numbers}\label{ss:chern}

The \emph{relative first Chern number} of the complex line bundle $\xi|_C$ with respect to the trivialization $\tau \in \mathcal{T}(\alpha,\beta)$, is denoted by
\[
c_\tau(C)  = c_1(\xi|_C,\tau),
\]
and defined as follows.  Let $\pi_Y:\R\times Y\to Y$ denote projection onto $Y$. We define  $c_1(\xi|_C,\tau)$ to be the algebraic count of zeros of a generic section $\psi$ of $\xi|_{[\pi_YC]}$ which on each end is nonvanishing and constant with respect to the trivialization on the ends.  In particular, given a class $Z\in H_2(Y,\alpha,\beta)$ we represent $Z$ by a smooth map $f: S \to Y$ where $S$ is a compact oriented surface with boundary.  Choose a section $\psi$ of $f^*\xi$ over $S$ such that $\psi$ is transverse to the zero section and $\psi$ is nonvanishing over each boundary component of $S$ with winding number zero with respect to the trivialization $\tau$.  We define 
\[
c_\tau(Z) : = \# \psi^{-1}(0),
\]
where `\#' denotes the signed count.

Thus, given another collection of trivialization choices up to homotopy $\tau' = \left( \{ {\tau'}_i^+ \}, \{ {\tau'}_j^-\} \right) \in \mathcal{T}(\alpha,\beta)$, by the convention \eqref{eqn:trivchange}, we have
\begin{equation}\label{eq:cherntriv}
c_\tau(Z) - c_{\tau'}(Z) = \sum_i m_i\left(\tau_i^+-{\tau}_i^{+'}\right) - \sum_j n_j \left(\tau_j^--{\tau}_j^{-'}\right);
\end{equation}
this is also reviewed in Lemma \ref{lem:formulas} (\ref{item:cot}). Moreover, we will also use the fact that $c_\tau$ is linear under both $+$ and $\uplus$; see Lemma \ref{lem:formulas}.

We now compute the relative first Chern number $c_\tau(Z)$ where $Z=[S]$ for the surfaces $S$ of Definition \ref{def:sfcs2}. {(Note that we compute the ECH index in \S\ref{s:ECHI} using only $c_{orb}$, but we required the values of $c_\tau$ for other $\tau$ for our computations in \S\ref{ss:cotTpq}.)}

\begin{lemma}\label{lem:relfirstCherncalc2} For $T(2,q)$, we have
\begin{enumerate}[{\em (i)}]
\itemsep-.25em
\item $c_{orb}([\Sigma])=0$,
\item $c_0([\Sigma])=2+q$,
\item\label{ctauS} $c_\Sigma([\Sigma])=2-q$,
\item $c_e(Z_{e^q})=c_h(Z_{h^2})=2+q$,
\item $c_{orb}(Z_e)=c_{orb}(Z_h)=0$.
\end{enumerate}
\end{lemma}
\begin{proof}
\begin{enumerate}[{ (i)}]
\itemsep-.25em
\item The trivialization $\tau_{orb}$ is global, so it extends from $b$ over the interior of any representative of the class $Z_b$. Thus $c_{orb}([\Sigma])=0$. 

\item The degree of the covering $\fp:\Sigma\to\CP^1_{2,q}$ is $2q$ and the orbifold Euler characteristic of the base of $\fp$ is $(2+q)/2q$, thus in analogy to \cite[Lem.~3.12]{preech}, %
\[
c_0([\Sigma])=2q\left(\frac{2+q}{2q}\right)=2+q.
\]

\item Using the change of trivialization formula and Lemma \ref{lem:taudiff2}(ii),
\[
c_\Sigma([\Sigma])=c_0([\Sigma])+\tau_\Sigma(b)-\tau_0(b)= 2+q -2q 
\]

\item Both $S_e$ and $S_h$ are a union of fibers, so the degree of the restriction of $\fp$ to each of them is zero and $c_e([S_e])=c_h([S_h])=0$. Thus by the linearity of the relative first Chern number under concatenation addition, Lemma \ref{lem:formulas}(\ref{item:clin}), as well as the fact that $\tau_e(b)=\tau_h(b)=\tau_0(b)$ by Lemma \ref{lem:taudiff2}(i),
\[
c_e(Z_{e^q})=c_e([S_e\cup\Sigma])=c_e([S_e])+c_0([\Sigma])=0+2+q
\]
and similarly for $c_h(Z_{h^2})$.

\item The trivialization $\tau_{orb}$ is global, so it extends from $e$ (respectively, $h$) over the interior of any representative of the class $Z_e$ (respectively, the class $Z_h$). Thus $c_{orb}(Z_e)=c_{orb}(Z_h)=0$. However, we may also verify this using the change-of-trivialization formula for covers of embedded orbits, Proposition \ref{prop:nonsimpleCOT}(i).

\end{enumerate}
\end{proof}

\subsection{Conley-Zehnder indices}\label{ss:CZ}
We now compute the Conley-Zehnder index of $b$ using the constant and orbibundle trivializations, and the Conley-Zehnder index of $e^q$ and $h^2$ using the orbibundle trivialization, which also provides the formulae for other iterates of $e$ and $q$.  To switch to the page trivialization, we will make use of the change of trivialization formulae computed previously.   We first review some basic definitions and properties of the Conley-Zehnder index.

Given a Reeb orbit $\gamma: \R / T\Z \to Y$, the linearized Reeb flow along $\gamma$ with respect to a choice of a trivialization $\tau \in \mathcal{T}(\gamma)$ from time 0 to time $t \in \R$ defines a symplectic map $P_{\gamma(t)}:\xi_{\gamma(0)} \to \xi_{\gamma(t)}$.  The \emph{symplectic return map} 
is defined to be $P_{\gamma(T)}$.  (Note that $P_{\gamma(0)}$ is the identity matrix.) 
If we assume $\gamma$ is nondegenerate then the path of symplectic matrices $\{ P_{\gamma(t)} \ | \ 0 \leq t \leq T \}
$
has a well-defined \emph{Conley-Zehnder} index, which we denote by
\[
CZ_\tau(\gamma) := CZ(P_{\gamma(t)}) \in \Z.
\]
In three dimensions, the Conley-Zehnder index can be more explicitly described in the {nondegenerate} setting as follows:
\begin{itemize}
\itemsep-.25em
\item If $\gamma$ is \emph{hyperbolic}, meaning the eigenvalues of the linearized return map are real, then there is an integer $n\in \Z$ such that the linearized Reeb flow along $\gamma$ rotates the eigenspaces of the linearized return map by angle $n\pi$ with respect to $\tau$.  We have:
\[
CZ_\tau(\gamma^k) = kn.
\]
The integer $n$ is always even when $\gamma$ is positive hyperbolic and always odd when $\gamma$ is negative hyperbolic.  We call $n$ the \emph{monodromy angle} of $\gamma$.
\item If $\gamma$ is \emph{elliptic}, meaning the eigenvalues of the linearized return map lie on the unit circle, then $\tau$ is homotopic to a trivialization in which the linearization of the time $t$ Reeb flow $\xi_{\gamma(0)} \to \xi_{\gamma(t)}$ along $\gamma$ rotates by angle $2\pi\theta_t$ for each $t \in [0,T]$, where $\theta:[0,T] \to \R$ is continuous and $\theta_0=0$. The nondegeneracy assumption forces $\theta_T$ to be irrational.  We have:
\begin{equation}\label{CZ:elliptic}
CZ_\tau(\gamma^k) = 2 \lfloor k\theta_T \rfloor + 1
\end{equation}
We call $\theta_T$ the \emph{monodromy angle} of $\gamma$.  (In some literature, $\theta_T$ is called the rotation angle, but we will use the terminology `rotation number' to designate the $\theta_T$ obtained from a specific homotopy class of trvializations, see Remark \ref{rem:linkingtriv} below.)
\end{itemize}

The Conley-Zehnder index depends only on the Reeb orbit $\gamma$ and the homotopy class of $\tau \in \mathcal{T}(\gamma)$.  If $\tau' \in \mathcal{T}(\gamma)$ is another trivialization then we have 
\begin{equation}\label{CZ:changetriv}
CZ_\tau(\gamma^k) - CZ_{\tau'}(\gamma^k) = 2k(\tau' - \tau).
 \end{equation}

\begin{remark}\label{rem:linkingtriv}
Our computation of knot filtered ECH with respect to the binding in \S \ref{s:spectral} will make use of the page trivialization $\tau_\Sigma: \xi|_\gamma \simeq  \R^2$,  as when $\gamma$ is the binding, a pushoff of $\gamma$ via this trivialization has linking number zero with $\gamma$.  If $H_1(Y)=0$, then we can associate to any elliptic Reeb orbit $\gamma$ a well-defined \emph{rotation number} 
\[
\op{rot}(\gamma):=\theta_T \in \R
\]
 in terms of the symplectic trivialization which, when used to push the elliptic Reeb orbit $\gamma$ off itself, has linking number zero.
\end{remark}

\begin{remark}\label{rem:c1sl}
When defining the Conley-Zehnder index of a nondegenerate Reeb orbit, it is typical to use a different homotopy class of trivialization $\tau$ than the one described in Remark \ref{rem:linkingtriv}, which extends to a trivialization of $\xi$ over a surface bounded by $\gamma$, rather than one yielding a zero linking number with respect to the pushoff of the Reeb orbit.  These two trivializations differ by the self-linking number of the transverse knot $\gamma$.   The \emph{self-linking number} of a simple Reeb orbit $\gamma$, or more generally, of any transverse knot (oriented and positively transverse to $\xi$) is defined as follows.  Let $\tau$ denote the homotopy class of a symplectic trivialization of $\xi|_\gamma$ for which a pushoff of $\gamma$ has linking number 0 with $\gamma$.  Let $\Sigma$ be a Seifert surface for $\gamma$.  Then 
\[
\op{sl}(\gamma):=-c_1(\xi|_\Sigma, \tau).
\]
We have $c_\Sigma([\Sigma]) = 2+q-2q$, so $2q-2-q$ is the self linking which corresponds to the rotation angle of the binding being $2q$.  Using the trivialization that extends over a disk, we will show that the monodromy angle is $2+q$.

\end{remark}

We have the following formula for the Conley-Zehnder indices of iterates of Reeb orbits associated to $\lambda_{2,q,\epsilon}$ which project to critical points $x$ of $H_{2,q}$.  We denote the $k$-fold iterate of an orbit which projects to $x \in \mbox{Crit}(H_{2,q})$ by $\gamma_x^k$.  This formula relies on an extension of the Conley-Zehnder index to degenerate orbits, which are the fibers of the prequantization orbibundle associated to $\lambda_{2,q}$.  namely that of the Robbin-Salamon index as in \cite{RS}.  Detailed background, various technicalities, and associated proofs can be found in \cite[\S 4]{jo2}, \cite{vknotes}, \cite[\S 3]{hong-cz}. The following {lemma} provides an overview of how these results are used in the setting at hand.

\begin{lemma}\label{lempre}
Fix $L>0$ and let ${H_{2,q}}$ be a Morse-Smale function on $\CP^1_{2,q}$ which is $C^2$ close to 1 as in Proposition \ref{prop:morse2}.  Then there exists $\varepsilon >0$ such that all periodic orbits $\gamma$ of $R_{2,q,\varepsilon}$ with action $\mathcal{A}(\gamma) <L$ are nondegenerate and project to critical points of $H_{2,q}$.  The Conley-Zehnder index such a Reeb orbit over $x \in \mbox{\em Crit}(H_{2,q})$ is given by
\[
\begin{array}{lcl}
CZ_\tau(\gamma_x^k) &=& {RS}_\tau(\gamma^k) -1 + \mbox{\em index}_xH_{2,q},\\
\end{array}
\]
where $RS$ stands for the Robbin-Salamon index, which can be associated to a degenerate Reeb orbit \cite{RS}.
\end{lemma}

\subsubsection{Constant trivialization}\label{ss:consttau}
The constant trivialization $\tau_0$ in Definition \ref{def:consttriv}, as for example considered in \cite[\S 3]{preech} can be used to compute the Conley-Zehnder index of the fibers which project to nonsingular points as follows. 
Let $x\in \CP^1_{2,q} $ be any point with with trivial isotropy.  Then for any point $y \in \fp^{-1}(x)$, a fixed trivialization of $T_x\CP^1_{2,q}$  allows us to trivialize $\xi_y$ because $\xi_y \cong T_x\CP^1_{2,q}$ This trivialization is invariant under the linearized Reeb flow and can be thought of as a  \emph{constant trivialization} over the orbit $\gamma_x$ because the linearized Reeb flow, with respect to this trivialization, is the identity map.

Using this constant trivialization, we have the following result regarding the Robbin-Salamon index, see \cite[Lem. 3.3]{vknotes}, \cite[Lem. 4.8]{jo2}.

\begin{lemma}\label{consttrivlem}
Let $x\in \CP^1_{2,q} $ be any point with with trivial isotropy and let $\gamma_x = \fp^{-1}(x) $ be the $S^1$ fiber realizing a Reeb orbit of $\lambda_{2,q}$, which projects to $x$.   Then for the constant trivialization $\tau_0$ we obtain $RS_{0}(\gamma_x)=0$ and $RS_{0}(\gamma_x^k) =0$, where $RS$ denotes the Robbin-Salamon index. \end{lemma}

In conjunction with  Lemma \ref{lempre} we obtain the following computations.
\begin{corollary}\label{cor:CZ0b}
Fix $L>0$, and $H_{2,q}$ as in Proposition \ref{prop:morse2}.  Then there exists an $\varepsilon >0$ such that all $k$-fold iterates of $b$ with action $\A(b^k) <L$ are nondegenerate.  Then 
\[
CZ_{0}(b^k) = 1.
\]
\end{corollary}

\subsubsection{Orbibundle trivialization}\label{ss:orbitau}
In order to compute the Conley-Zehnder indices of Reeb orbits which project to critical points we appeal to a global orbibundle trivialization which relates the Robbin-Salamon indices of the degenerate Reeb orbit fibers to the orbifold Chern class of the base $\CP^1_{2,q}$.   A bit of clarification may be helpful prior to stating Hong's result on this relationship that we will employ below.  

First, we note that that this global orbibundle trivialization is done in terms of pullback bundles, which is provided in full detail in \cite[\S 3]{hong-cz} to establish the below result, which generalizes the construction in \cite{vknotes}.  We use change of coordinate formulae to convert our relative intersection pairing terms using the aforementioned constant and page trivializations to this orbibundle trivialization, which obiviates the need to work explicitly with the orbibundle trivialization beyond its below application to the computation of the Robbin-Salamon and Conley-Zehnder indicies of the fiber orbits, and thus their monodromy angles, and the prior use of deducing that the relative Chern number vanishes.  

Second, the orbifold fundamental group is deserving of (more than) a few words; see also \cite[\S 1.5, 4]{bgbook} and \cite[\S 2.2]{orbnotes}.  The orbifold fundamental group $\pi_1^{orb}$ was first conceived by Thurston in terms of the group of deck transformations of a universal covering orbifold \cite[\S 13]{bill1}, though this did not fully make its way into \cite{bill2}.  The modern presentation realizing Thurston's construction comes by way of homotopy classes of loops on the pseudogroups representing the orbifold, though a more involved notion of homotopy classes is necessary to capture the local nature of pseudogroups, as pseudogroups are in essence groups of transformations where there ``may be some problems" with domains of definition.  Haeflinger gave an alternate \cite{hae} but equivalent approach \cite{hd}, by providing an explicit Borel-type construction of the classifying space of an orbifold; the latter permits the definitions of all higher homotopy groups as well.

\begin{theorem}{\em \cite[Thm.~3.1]{hong-cz}}
Let $(\cO, \omega, J)$ be a K\"ahler orbifold so that it admits an $S^1$-orbibundle $\fp: Y \to \cO$ such that $Y$ has a Besse contact structure,\footnote{In \cite{hong-cz}, a Besse contact structure is called a $K$-contact structure.} where 
 $d\lambda = \fp^*\omega$.  If 
\begin{enumerate}[\em (i)]
\itemsep-.25em
\item $c_1^{orb}(T\cO) = \nu_\cO [\omega] \in H^2(\cO, \Q) \mbox{ for some integer } \nu_\cO \in \Z;$
\item $\pi_1^{orb}(\cO) = 0;$
\item $Y$ is a manifold.
\end{enumerate}
Then the Robbin-Salamon index $RS_{orb}$ of the $|\Gamma_x|$-th iterate of a Reeb orbit $\gamma$ is given by
\[
RS_{orb}\left(\gamma^{|\Gamma_x|}\right) = 2 \nu_\cO,
\]
where  $\fp(\gamma)=x$ and $\Gamma_x$ is the isotropy group at $x$.
\end{theorem}

\begin{remark}
Let $(S^3, \lambda_{p,q})$ be the {prequantization} $S^1$-orbibundle over the orbifold $(\CP^1_{p,q}, \omega_{p,q})$ of real Euler class $-\frac{1}{pq}$.  Then for the orbibundle trivialization $\tau_{orb}$ along the fiber $\gamma_x^{|\Gamma_x|}$, we have $RS_{{orb}}(\gamma_x^{|\Gamma_x|k})=2(p+q)k$.  (Recall that Example \ref{symporb} computed $[\omega_{p,q}] = [d\lambda_{p,q}]=\frac{1}{pq}$ while Definition \ref{ex:orbX} computed $c_1^{orb}(\CP_{p,q}) = \frac{p+q}{pq}$.)  
\end{remark}

In combination with Lemma \ref{lempre}, the above yields the following formulae for the Conley-Zehnder indices of iterates of orbits which project to critical points $x$ of $H_{2,q}$.

\begin{lemma}\label{lem:orbtrivCZ2}
Fix $L>0$ and $H_{2,q}$ a Morse-Smale function as in Proposition \ref{prop:morse2} on $\cpq$ which is $C^2$ close to 1.  Then there exists $\varepsilon>0$ such that all periodic orbits of $\gamma$ of $R_{2,q,\varepsilon}$ with action $\mathcal{A}(\gamma)<L$ are nondegenerate and project to critical points of $H_{2,q}$.  The Conley-Zehnder index of such a Reeb orbit over $x \in \op{Crit}(H_{2,q})$ is given by
\[
CZ_{orb}(\gamma_x^{k|\Gamma_x|}) = 2(2+q)k + \mbox{\em index}_x H_{2,q} -1. 
\]
In particular, we have:
\[
\begin{array}{lcl}
CZ_{orb}(b^k) & = & 2(2+q)k + 1, \\
CZ_{orb}(h^{2k}) &= & 2(2+q)k, \\
CZ_{orb}(e^{qk}) & = & 2(2+q)k -1. \\
\end{array}
\]
Thus,
\begin{itemize}
\itemsep-.35em
\item $b$ is elliptic of {monodromy angle} ${2+q} + \delta_{b,L}$, where $0<\delta_{b,L} \ll 1 $ is irrational;
\item $h$ is negative hyperbolic with rotation number $2+q$; 
\item $e$ is elliptic of {monodromy angle} $(2+q)/{q} - \delta_{e,L}$, where $0<\delta_{e,L} \ll 1 $ is irrational.
\end{itemize}
\end{lemma}

\begin{remark}
The computations in Table \ref{table:trefoil} show that our construction recovers the action filtered cylindrical contact homology of $(S^3, \ker \lambda_{2,3})$ via an action filtered chain complex with vanishing differential by way of \cite{jo1, jo2, leo}.  Since $h$ is a negative hyperbolic orbit, we discard all of its even iterates, thus there is always exactly one generator in every odd degree of the action filtered chain complex.  In contrast to the scheme utilized in \cite{mlyob}, we have that the binding and its iterates are generators of cylindrical contact homology.
\end{remark}

\begin{table}[h!]
\centering
\begin{tabular}{ || c | c | c | c | c | c | c | c  | c  | c  | c  | c  | c | c | c  | c | c | c | c || } 
\hline orbit  & $e$ & $h$  & $e^2$ & $e^3$ & $h^2$ & $b$ & $e^4$ & $h^3$ & $e^5$ & $e^6$ & $h^4$ & $b^2$ & $e^7$ & $h^5$ & $e^8$ & $e^9$ & $h^6$ & $b^3$ \\ 
\hline  $CZ_{orb}$  &  3 & 5  & 7 & 9 & 10 & 11 & 13 & 15 & 17 & 19 & 20 & 21 & 23  & 25 & 27 & 29 & 30 & 31\\
\hline  
\end{tabular}
\caption{Conley-Zehnder indices for the $T(2,3)$ open book decomposition}
\label{table:trefoil}
\end{table}

\subsubsection{Page trivialization aka push-off linking zero trivialization}\label{ss:pagetau}
Finally we consider the Conley-Zehnder indices and  monodromy angles of the bindings with respect to the the page trivialization $\tau_\Sigma$, described in \S \ref{ss:trivs}.  Since the page trivialization is the push-off linking zero trivialization, these {monodromy angles are the} rotation numbers  used in our computation of knot filtered ECH.  
\begin{lemma}\label{lem:pagetrivCZ2} For the $T(2,q)$ binding $b$, $CZ_\Sigma(b^B)=4qB+1$.  Thus with respect to the page trivialization  $b$ is elliptic with $\op{rot}(b)=2q+\delta_{b,L}$, where $\delta_{b,L}$ is an irrational number such that $0 < \delta_{b,L} \ll 1.$
\end{lemma}
\begin{proof} Using (\ref{CZ:changetriv}) and Lemmas \ref{lem:orbtrivCZ2} and \ref{lem:taudiff2}(iv), we have
\[
CZ_\Sigma(b^B)=2B(\tau_{orb}(b)-\tau_\Sigma(b))+CZ_{orb}(b^B)=2B(q-2)+(2(2+q)B+1) = 4qB+1.
\]
\end{proof}

It may be of interest to note that a direct computation for the Conley-Zehnder index of the binding of an open book decomposition with disk like pages is given in \cite[Thm.~3.11]{bh}.  

\subsection{Relative intersection pairing}\label{ss:Q}

In order to compute the relative intersection pairing $Q_\tau(Z)$, we need to choose specific surfaces $S\subset[-1,1]\times Y$ representing a class $Z\in H_2(Y,\alpha,\beta)$.

\begin{definition}[{\cite[Def.~2.11]{Hrevisit}}]
Given $Z\in H_2(Y,\alpha,\beta)$ we define an \emph{admissible representative} of $Z$ to be a smooth map $f: S \to [-1,1] \times Y$, where $S$ is an oriented compact surface such that
\begin{enumerate}
\itemsep-.25em
\item The restriction of $f$ to the boundary $\partial S$ consists of positively oriented  covers of $\{ 1\} \times \alpha_i$ whose total multiplicity is $m_i$ and negatively oriented covers of $\{ -1\} \times \beta_j$ whose total multiplicity is $n_j$.
\item The projection $\pi_Y: [-1,1] \times Y \to Y$ satisfies $[\pi_Y(f(S))]=Z$. 
\item The restriction of $f$ to $\mbox{int}(S)$ is an embedding and $f$ is transverse to $\{-1,1\} \times Y$.
\end{enumerate}
Such an $S$ is said to be an {admissible surface} for $Z\in H_2(Y,\alpha,\beta)$.

\end{definition}

{An admissible representative $S$ of $Z\in H_2(Y,\alpha,\beta)$ determines braids around the component Reeb orbits $\alpha_i$ and $\beta_j$. Let $\zeta_i^+$ denote the braid around $\alpha_i$ given by $S\cap(\{1-\varepsilon\}\times Y)$ for $\varepsilon>0$ sufficiently small; it is well-defined up to isotopy in a tubular neighborhood of $\alpha_i$ chosen to not intersect the tubular neighborhood of to any other simple Reeb orbit in $\alpha$.  Note that $\zeta_i^+$ will have $m_i$ strands. Define $\zeta_j^-$ analogously in a tubular neighborhood of $\beta_j$.}

We now define the linking number of admissible representatives.  If $S'$ is an admissible representative of $Z' \in H_2(Y,\alpha',\beta')$ such that the interior of $S'$ does not intersect the interior of $S$ near the boundary, and with braids $\zeta_i^{+ '}$ and $\zeta_j^{- '}$, we can define the \emph{linking number of $S$ and $S'$} to be
\[
\ell_\tau(S,S') := \sum_i \ell_\tau(\zeta_i^+,\zeta_i^{+ '}) - \sum_j \ell_\tau(\zeta_j^-,\zeta_j^{- '}).
\]
  
Using this, we may now define the relative intersection pairing. Let  $S$ and $S'$ be two surfaces which are admissible representatives of $Z \in H_2(Y,\alpha,\beta)$ and $Z'\in H_2(Y,\alpha',\beta')$  whose interiors $\dot{S}$ and $\dot{S}'$ are transverse and do not intersect near the boundary.  We define 
the \emph{relative intersection pairing} by the following signed count
\begin{equation}
Q_\tau(Z,Z'):= \# \left( \dot{S} \cap \dot{S}'\right) - \ell_\tau(S,S').
\end{equation}
Moreover, $Q_\tau(Z,Z')$ is an integer which depends only on $\alpha, \beta, Z, Z'$ and $\tau$.  If $Z=Z'$ then we write $Q_\tau(Z):=Q_\tau(Z,Z)$ and call this the \emph{relative self-intersection number} of $Z$.

The above definition of the relative intersection pairing comes from \cite[\S2.7]{Hrevisit}, and it is particularly useful when we can find admissible representatives with $\#\left(\dot{S}\cap\dot{S}'\right)=0$ (see the proofs of Lemmas \ref{lem:crossQ}). However, often it is also desirable to compute $Q_\tau$ when the linking number term is zero, and this is how the relative intersection pairing was originally defined in \cite[\S2.4]{Hindex}. Those surfaces with $\ell_\tau(S,S')=0$ can be characterized geometrically, as in the following definition. Note that instead of following \cite[\S2.4]{Hindex}, we use the version for embedded surfaces appearing in \cite[\S3.3]{lecture}.\footnote{The definition in \cite[\S2.4]{Hindex} is only for immersed curves, and it requires that the boundary of $S$ consists of single covers of the $\alpha_i$ and $\beta_j$, which we do not want to be restricted to.}

\begin{definition} Assume $S$ is an admissible representative of $Z$, and furthermore that the following conditions hold:
\begin{enumerate}
\itemsep-.25em
\item $\pi_Y\circ f$ is an embedding near $\partial S$.
\item The $m_i$ (respectively, $n_j$) nonvanishing intersections of these embedded collars of $\partial S$ with $\xi$ lie in distinct rays emanating from the origin and do not rotate (with respect to $\tau$) as one goes around $\alpha_i$ (respectively, $\beta_j$).
\end{enumerate}
Then we say $S$ is a \emph{$\tau$-representative} of $Z$, and $
Q_\tau(Z,Z'):= \# \left( \dot{S} \cap \dot{S}'\right).$
\end{definition}

We start by noting that the surfaces defined in \S\ref{ss:trivs} are convenient for computing $Q_\tau$.

\begin{lemma}\label{lem:taurep} The following surfaces are $\tau$-representatives:
\begin{enumerate}[{\em (i)}]
\itemsep-.25em
\item The surface $\Sigma$ is a $\tau_\Sigma$-representative of $Z_b$.
\item The surface $S_e$ is a $\tau_e$-representative of the class in $H_2(S^3,e^q,b)$.
\item The surface $S_h$ is a $\tau_h$-representative of the class in $H_2(S^3,h^2,b)$.
\end{enumerate}
\end{lemma}
\begin{proof} The conclusions follow immediately from the definitions of the trivializations in \S\ref{ss:trivs}.
\end{proof}

We prove several relative intersection pairings are zero.

\begin{lemma}\label{lem:zero} The following relative intersection pairings are zero.
\begin{enumerate}[{\em (i)}]
\itemsep-.25em
\item $Q_e([S_e])=Q_h([S_h])=0$,
\item $Q_\tau([S_e],[S_h])=0$, where $\tau(e)=\tau_e(e), \tau(h)=\tau_h(h)$, and $\tau(b)=\tau_0(b)$.
\end{enumerate}
\end{lemma}

\begin{proof}

The surfaces $S_e$ and $S_h$ are $\tau$-representatives of the single element in $H_2(S^3,e^q,b)$ and $H_2(S^3,h^2,b)$, respectively, by Lemma \ref{lem:taurep} (ii, iii). Thus the relative intersection pairings in all three cases equal the intersection numbers of the surfaces in $\R\times S^3$. These are all zero because even the intersection numbers between the projections of the surfaces to $S^3$ are zero.

\end{proof}

We compute the remaining relative intersection pairings, separated by proof method.

\begin{lemma}\label{lem:Qb} The following formulas hold.
\begin{enumerate}[{\em (i)}]
\itemsep-.25em
\item $Q_\Sigma(Z_b)=0$,
\item $Q_{orb}(Z_b)=q-2$,
\item $Q_0(Z_b)=2q$.
\end{enumerate}
\end{lemma}
\begin{proof}
\begin{enumerate}[{ (i)}]
\itemsep-.25em

\item Recall that $\Sigma$ is a $\tau_\Sigma$-representative of $Z_b$. Thus $Q_\Sigma([\Sigma])=0$ because $\Sigma$ can be pushed off of itself in $S^3$ via the Reeb flow, which sends it to another surface whose boundary is constant with respect to $\tau_\Sigma$. 

\item We use the change-of-trivialization formula Lemma \ref{lem:formulas} (iv) and Lemma \ref{lem:taudiff2}(iii):
\[
Q_{orb}([\Sigma])=Q_\Sigma([\Sigma])+\tau_{orb}(b)-\tau_\Sigma(b)=0+q-2.
\]

\item We use the change-of-trivialization formula Lemma \ref{lem:formulas}(iv) and Lemma \ref{lem:taudiff2}(ii):
\[
Q_0([\Sigma])=Q_\Sigma([\Sigma])+\tau_0-\tau_\Sigma=0+2q.
\]
Note that this agrees with the quantity $-ed^2$ computed in \cite[Lem.~3.13]{preech} for the case of a prequantization bundle, because $e=-\frac{1}{2q}$  while $d=d(b)=2q$; we expect this agreement because $b$ is a regular fiber of a prequantization orbibundle.
\end{enumerate}
\end{proof}

\begin{lemma}\label{lem:QeQh} The following formulas hold.
\begin{enumerate}[{\em (i)}]
\itemsep-.25em
\item $Q_{orb}(Z_e)=-1$,
\item $Q_{orb}(Z_h)=-1$.
\end{enumerate}
\end{lemma}
\begin{proof}
\begin{enumerate}[{ (i)}]
\itemsep-.25em
\item Using the change-of-trivialization formula from Proposition \ref{prop:nonsimpleCOT}(ii) and the value of $\tau_{orb}(e^q)-\tau_e(e^q)$ computed in Lemma \ref{lem:e3h2diff}(i), we have
\begin{align*}
Q_{orb}(qZ_e)&=Q_e(qZ_e)+q(\tau_{orb}(e^q)-\tau_e(e^q))
\\&=Q_e(Z_{e^q})+q(-2-q)
\\&=Q_e([S_e]+[\Sigma])-2q-q^2
\\&=-q^2
\end{align*}
where we have used linearity with respect to concatenation, Lemma \ref{lem:formulas} (\ref{item:clin}), and the fact that $\tau_e(b)=\tau_0(b)$, Lemma \ref{lem:taudiff2}(i), in the fourth line, and Lemma \ref{lem:zero} (i) to compute $Q_e([S_e])$ in the fifth line. Thus by bilinearity with respect to union addition, see Remark \ref{rmk:formulas}~(ii), $Q_{orb}(Z_e)=\frac{Q_{orb}(qZ_e)}{q^2}=-1.$

\item By the change-of-trivialization formula from Proposition \ref{prop:nonsimpleCOT} (ii) and the value of $\tau_{orb}(h^2)-\tau_h(h^2)$ computed in Lemma \ref{lem:e3h2diff} (ii), we have
\begin{align*}
Q_{orb}(2Z_h)&=Q_h(2Z_h)+2(\tau_{orb}(h^2)-\tau_h(h^2))
\\&=Q_h(Z_{h^2})+2(-2-q)
\\&=Q_h([S_h]+[\Sigma])-2(2+q)
\\&=-4,
\end{align*}
where in the fourth line we have used the fact that $Q_\tau$ is linear with respect to concatenation addition and the fact that $\tau_h(b)=\tau_0(b)$, Lemma \ref{lem:taudiff2}(i). We have also used Lemma \ref{lem:zero}(i) to compute $Q_h([S_h])$ in the fifth line. Thus by bilinearity with respect to union addition, see Remark \ref{rmk:formulas}(ii), $
Q_{orb}(Z_h)=\frac{Q_{orb}(2Z_h)}{4}=-1.$

\end{enumerate}
\end{proof}

\begin{lemma}\label{lem:crossQ} The following formulas hold.
\begin{enumerate}[{\em (i)}]
\itemsep-.25em
\item $Q_{orb}(Z_e,Z_h)=1$,
\item $Q_{orb}(Z_e,Z_b)=2$,
\item $Q_{orb}(Z_h,Z_b)=q$.
\end{enumerate}
\end{lemma}

\begin{proof}
Whenever the ends of $Z$ and $Z'$ are disjoint and connected, the $\ell_\tau$ term in the definition of $Q_\tau$ is zero. Thus for any trivialization $\tau$ and admissible representatives $S$ and $S'$ of $Z$ and $Z'$, we have
\[
Q_\tau(Z,Z')=\#(\dot{S}\cap\dot{S'})=\#(S\cap S')=\ell(\partial S,\partial S').
\]
Corollary \ref{cor:linkingnumber} computes the relevant linking numbers.

\end{proof}

\section{Computation of embedded contact homology}\label{s:ECHI}

In this section we compute the  action filtered ECH chain complex for the contact form $\lambda_{2,q,\varepsilon}$. 
We establish our ECH index theorem and then summarize the results of this section in \S\ref{ss:2}. In \S\ref{sss:cc2} we finish the description of the ECH chain complex of $\lambda_{2,q,\varepsilon}$ up to action $L(\varepsilon)$. In particular, we prove that the differential disappears for index reasons, enabling us to compute
\[
\lim_{\varepsilon\to0}ECH^{L(\varepsilon)}_*(S^3,\lambda_{2,q,\varepsilon},J)=\begin{cases}\Z/2&\text{ if }* \in 2\Z_{\geq 0}
\\0&\text{else}\end{cases}
\] 
 combinatorially, including the $T(2,q)$ knot filtration in \S\ref{s:spectral}.

Other {well-known} means of recovering $ECH_*(S^3, \xi_{std})$ can be achieved via the irrational ellipsoid in \cite[\S3.7]{lecture}, or even the prequantization bundle over $S^2$ with Euler class -1 in \cite[\S7.2.2]{preech}. {The utility of our computation is not to obtain the homology, but to prove that our chain complex consists precisely of one generator per index, and to identify the degree of that generator as a function of its index. This identification will be heavily relied on in \S\ref{s:spectral}.}

\subsection{Summary of calculations for $T(2,q)$}\label{ss:2}

Recall that the generators of $ECC_*^{L(\varepsilon)}(S^3,\lambda_{2,q,\varepsilon}, J)$ are of the form $b^Bh^He^E$, where $B,E \in \Z_{\geq 0}$ and $H=0,1$.\footnote{Recall that the Reeb vector field $R_{2,q,\varepsilon}$ admits two orbits realizing singular fibers of the Seifert fibration, called $e$ and $h$ and of isotropy $\Z/q$ and $\Z/2$, respectively, with $T(2,q)$ as a regular fiber $b$.  We obtain $\lambda_{2,q,\varepsilon}$ by perturbing $\lambda_{2,q}$ using the the Morse function $H_{2,q}$ of Proposition \ref{prop:morse2} on the base orbifold $\CP^1_{2,q}$.  As described in Lemma \ref{lem:orbitseh}, up to large action the only Reeb orbits of $R_{2,q,\varepsilon}$ are iterates of $e$ (elliptic), $h$ (negative hyperbolic), and $b$ (elliptic). }

\begin{notation}
When specifying a particular Reeb current with multiplicative notation, we will follow the convention that $m>0$, omitting the term $(\alpha, m)$ which in multiplicative notation is expressed as $\alpha^m$ if $m=0$. For an unspecified or general Reeb current, however, we will allow $m=0$, and it will correspond to an omitted Reeb current in the usual notation.
\end{notation}

Our computation of the ECH index is as follows. 

\begin{theorem}\label{thm:ECHI}
The ECH index $I$ for $(S^3,\lambda_{2,q,\varepsilon})$ satisfies the following formula.  For any Reeb current $e^Eh^Hb^B$ with action less than $L(\varepsilon)$, where $E=qm+r, 0\leq r\leq q-1$: 
\begin{multline}
I(b^Bh^He^E)=2EH-H^2+(q+2)H+2qB^2+(q+3)B+4EB+2qHB \label{eqn:IEHB} \\ 
+\frac{2}{q}E^2+(q+1)m-\frac{2}{q}r^2+2r+\left\lfloor\frac{2r}{q}\right\rfloor(2r-q+1)
\end{multline}
where $E=qm+r, 0\leq r\leq q-1$. 
\end{theorem}

\begin{proof}
If we can show
\begin{equation}\label{eqn:CZIorbeE}
CZ^I_{orb}(e^E)=\frac{q+2}{q}E^2+(q+1)m-\frac{2}{q}r^2+2r+\left\lfloor\frac{2r}{q}\right\rfloor(2r-q+1)
\end{equation}
and
\[
CZ^I_{orb}(b^B)=(q+2)B^2+(q+3)B,
\]
then (\ref{eqn:IEHB}) follows from Definition \ref{defn:ECHI} of the ECH index, linearity of $c_\tau$ and bilinearity of $Q_\tau$ with respect to union addition, Lemma \ref{lem:formulas} (\ref{item:linbil}), as well as Lemmas \ref{lem:orbtrivCZ2}, \ref{lem:relfirstCherncalc2}(i,v), \ref{lem:Qb}(ii), \ref{lem:QeQh}, and \ref{lem:crossQ}.

The formula for $CZ^I_{orb}(b^B)$ can be quickly obtained from Lemma \ref{lem:orbtrivCZ2}:
\[
CZ^I_{orb}(b^B)=\sum_{i=1}^b(2(q+2)i+1)=(q+2)(B+1)B+B.
\]

The formula for $CZ^I_{orb}(e^E)$ where $E=qm+r$ is elementary but slightly complicated to obtain. First we compute a formula for $CZ(e^E)$ in terms of $m$ and $r$. Let $(q+2)/q-\delta_{e,\varepsilon}$ denote the monodromy angle of $e$ with respect to the contact form $\lambda_\varepsilon$ and the trivialization $\tau_{orb}$; recall $\delta_{e,\varepsilon}>0$ and approaches zero as $\varepsilon\to0$ (see Lemma \ref{lem:orbtrivCZ2}). Then
\begin{align*}
CZ_{orb}(e^{qm+r})&=2\left\lfloor(qm+r)\left(\frac{q+2}{q}-\delta_{e,\varepsilon}\right)\right\rfloor+1
\\&=2(q+2)m+2r+2\left\lceil\frac{2r}{q}\right\rceil-1.
\end{align*}
To obtain $CZ^I_{orb}(e^E)$ we first compute $CZ^I_{orb}(e^{qm})$:
\begin{align*}
CZ^I_{orb}(e^{qm})&=\sum_{i=0}^{m-1}\sum_{j=1}^qCZ_{orb}(e^{qi+j})
\\&=\sum_{i=0}^{m-1}\left(\sum_{j=1}^{\lfloor q/2\rfloor}(2(q+2)i+2j+1)+\sum_{j=\lceil q/2\rceil}^{q-1}(2(q+2)i+2j+3)+(2(q+2)(i+1)-1)\right)
\\&=\sum_{i=0}^{m-1}2q(q+2)i+q^2+3q+1
\\&=\frac{q+2}{q}E^2+(q+1)m,
\end{align*}
using the $r=0$ case of
\[
E^2=(qm+r)^2=q^2m^2+2qmr+r^2.
\]
This proves (\ref{eqn:CZIorbeE}) when $r=0$.  Now for $r=1,\dots,\lfloor q/2\rfloor$, we have
\begin{align*}
CZ^I_{orb}(e^{qm+r})&=CZ^I_{orb}(e^{qm})+\sum_{j=1}^rCZ_{orb}(e^{qm+j})
\\&=q(q+2)m^2+(q+1)m+\sum_{j=1}^r\left(2(q+2)m+2j+1\right)
\\&=\frac{q+2}{q}E^2+(q+1)m-\frac{2}{q}r^2+2r,
\end{align*}
while when $r=\lceil q/2\rceil,\dots,q-1$, we have
\begin{align*}
CZ^I_{orb}(e^{qm+r})&=CZ^I_{orb}(e^{qm+\lfloor q/2\rfloor})+\sum_{j=\lceil q/2\rceil}^rCZ_{orb}(e^{qm+j})
\\&=q(q+2)m^2+(q+1)m+(2(q+2)m+1)\left\lfloor\frac{q}{2}\right\rfloor+\left\lfloor\frac{q}{2}\right\rfloor\left(\left\lfloor\frac{q}{2}\right\rfloor+1\right)
\\&\;\;\;\;+\sum_{j=\lceil q/2\rceil}^r\left(2(q+2)m+2j+3\right)
\\&\;\;\;\;+(2(q+2)+3)\left(r-\left\lfloor\frac{q}{2}\right\rfloor\right)+r(r+1)-\left\lfloor\frac{q}{2}\right\rfloor\left(\left\lfloor\frac{q}{2}\right\rfloor+1\right)
\\&=\frac{q+2}{q}E^2+(q+1)m-\frac{2}{q}r^2+4r-q+1.
\end{align*}
\end{proof}

\begin{table}[h!]
\begin{subtable}[c]{0.45\textwidth}
\centering
\begin{tabular}{||c|c|c||}
\hline
degree&generator&index
\\\hline 0& $\emptyset$ &0
\\\hline 2&$e$&2
\\\hline 3&$h$&4
\\\hline 4&$e^2$&6
\\\hline 5&$he$&8
\\\hline 6&$e^3$&10
\\\hline 6&$b$&12
\\\hline 7&$he^2$&14
\\\hline 8&$e^4$&16
\\\hline 8&$be$&18
\\\hline 9&$he^3$&20
\\\hline 9&$bh$&22
\\\hline 10&$e^5$&24
\\\hline 10&$be^2$&26
\\\hline 11&$he^4$&28
\\\hline 11&$bhe$&30
\\\hline 12&$e^6$&32
\\\hline 12&$be^3$&34
\\\hline 12&$b^2$&36
\\\hline 13&$he^5$&38
\\\hline 13&$bhe^2$&40
\\\hline 14&$e^7$&42
\\\hline 14&$be^4$&44
\\\hline 14&$b^2e$&46
\\\hline 15&$he^6$&48
\\\hline 15&$bhe^3$&50
\\\hline 15&$b^2h$&52
\\\hline 16&$e^8$&54
\\\hline 16&$be^5$&56
\\\hline 16&$b^2e^2$&58
\\\hline 17&$he^7$&60
\\\hline 17&$bhe^4$&62
\\\hline 17&$b^2he$&64
\\\hline 18&$e^9$&66
\\\hline 18&$be^6$&68
\\\hline
\end{tabular}
\subcaption{generators for $T(2,3)$}
\end{subtable}
\begin{subtable}[c]{0.45\textwidth}
\centering
\begin{tabular}{||c|c|c||}
\hline
degree&generator&index
\\\hline 0& $\emptyset$ &0
\\\hline 2&$e$&2
\\\hline 4&$e^2$&4
\\\hline 5&$h$&6
\\\hline 6&$e^3$&8
\\\hline 7&$he$&10
\\\hline 8&$e^4$&12
\\\hline 9&$he^2$&14
\\\hline 10&$e^5$&16
\\\hline 10&$b$&18
\\\hline 11&$he^3$&20
\\\hline 12&$e^6$&22
\\\hline 12&$be$&24
\\\hline 13&$he^4$&26
\\\hline 14&$e^7$&28
\\\hline 14&$be^2$&30
\\\hline 15&$he^5$&32
\\\hline 15&$bh$&34
\\\hline 16&$e^8$&36
\\\hline 16&$be^3$&38
\\\hline 17&$he^6$&40
\\\hline 17&$bhe$&42
\\\hline 18&$e^9$&44
\\\hline 18&$be^4$&46
\\\hline 19&$he^7$&48
\\\hline 19&$bhe^2$&50
\\\hline 20&$e^{10}$&52
\\\hline 20&$be^5$&54
\\\hline 20&$b^2$&56
\\\hline 21&$he^8$&58
\\\hline 21&$bhe^3$&60
\\\hline 22&$e^{11}$&62
\\\hline 22&$be^6$&64
\\\hline 22&$b^2e$&66
\\\hline 23&$he^9$&68
\\\hline
\end{tabular}
\subcaption{generators for $T(2,5)$}
\end{subtable}
\caption{Note the differences in the number of generators of fixed degree between the $\lambda_{2,3,\varepsilon}$ and $\lambda_{2,5,\varepsilon}$ contact forms.  Neither $\lambda_{2,3,\varepsilon}$  nor $\lambda_{2,5,\varepsilon}$  admit a generator of degree 1, and for $\lambda_{2,5,\varepsilon}$ there are no generators in degree 3. }
\label{table:gen2}
\end{table}

\begin{remark}[Evenness of ECH index] We note that the ECH index appearing in Theorem \ref{thm:ECHI} is always an even integer, as it must be by the Index Parity property, Theorem \ref{thm:Iproperties} (iv), because $h$ is negative hyperbolic by Lemma \ref{lem:orbitseh}. Because $q$ is odd, the only terms which are not even integers are
\[
-H^2,\quad (q+2) H,\quad \frac{2}{q}E^2,\quad \mbox{and } -\frac{2}{q}r^2.
\]
If $H=1$ then the sum $-H^2+(q+2)H=q+1$ is even. Moreover,
\[
\frac{2}{q}(E^2-r^2)=\frac{2}{q}(q^2m^2+2qmr)=2qm^2+4mr,
\]
which is an even integer.  Thus in order to obtain the ECH of $S^3$ as computed in \cite[\S3.7]{lecture}, the ECH differential must vanish; this also follows as a direct consequence of the combinatorial results in \S\ref{sss:cc2}.
\end{remark}

 The formula for the ECH index appearing in Theorem \ref{thm:ECHI} is rather opaque.  To make sense of it, recall the notion of degree of a generator, adapted to the case $\beta=\emptyset$ in Definition \ref{def:degreep}:
 \[
 d(b^Bh^He^E) = \frac{B+\frac{1}{2}H+\frac{1}{q}E}{| e|} = 2qB + qH +2E.
 \]

\begin{remark} We now summarize the key properties of the relationship between the ECH index of a generator and its degree so as to better elucidate the formula for the ECH index in Theorem \ref{thm:ECHI} and its consequences in \S\ref{sss:cc2}.
\begin{itemize}
\itemsep-.25em
\item If $d(\alpha)>d(\beta)$ then $I(\alpha)>I(\beta)$; moreover, the lowest index appearing for a generator of degree $d+1$ is exactly two more than the highest index appearing for a generator of degree $d$, so long as $d$ is large enough. See Lemma \ref{lem:dd+1}.
\item Within a given degree, the ECH index increases by twos in lexicographic order on triples $(B,H,E)$.\footnote{Note that within a given degree $H$ is constant, because its contribution to the degree has odd parity while the contributions of $E$ and $B$ have even parity.}
\item A generator of the form $b^B$ has degree $d=2qB$ and index $I=2qB^2+(q+3)B$, thus
\[
d=\frac{q+3+\sqrt{(q+3)^2+8Iq}}{2}.
\]
{The function $I\mapsto d$ in general is a more complicated approximation of the above; see \S \ref{ss:index-degree}.}
\end{itemize}
We provide a list of generators organized by degree and index for  $\lambda_{2,3,\varepsilon}$ in
Table \ref{table:gen2}(a) and for  $\lambda_{2,5,\varepsilon}$ in Table \ref{table:gen2}(b).

\end{remark}

We prove the following proposition  in \S\ref{sss:cc2}, which will be key to our computation of the knot filtration in \S\ref{s:spectral}.  This is primarily done combinatorially, with a minimal invocation of results on ECH cobordism maps.

\begin{proposition}\label{prop:bijectionZ2}
For any $\lambda$-compatible almost complex structure $J$, we have
\[
ECC_*^{L(\varepsilon)}(S^3,\lambda_{2,q,\varepsilon},J)=\begin{cases}
\Z/2&\text{ if }*\in 2\Z_{\geq0}
\\0&\text{ otherwise},
\end{cases}
\]
so long as $*$ is small enough relative to $L(\varepsilon)$ (see Lemma \ref{lem:orbitseh}).

Moreover, for $\varepsilon>\varepsilon'$, there is an exact symplectic cobordism from $(S^3,\lambda_{2,q,\varepsilon})$ to $(S^3,\lambda_{2,q,\varepsilon'})$, and the compositions of the inclusion-induced maps $\iota^{L(\varepsilon),L(\varepsilon')}$ and the cobordism maps $\Phi^L$  defined in \cite[Thm.~2.17]{preech} are the canonical bijection on generators.
\end{proposition}

After taking direct limits, Proposition \ref{prop:bijectionZ2} immediately provides a combinatorial computation of the ECH of $(S^3,\xi_{std})$:\footnote{For a simpler computation of $ECH(S^3,\xi_{std})$, see \cite[\S3.7]{lecture}.}

\begin{corollary}\label{thm:ECHS3} The direct limit of the homologies of the chain complexes in Proposition \ref{prop:bijectionZ2} recovers the ECH of $(S^3,\xi_{std})$:
\[
\lim_{\varepsilon\to0}ECH_*^{L(\varepsilon)}(S^3,\lambda_{2,q,\varepsilon},J)=\begin{cases}
\Z/2&\text{ if }* \in 2\Z_{\geq0}
\\0&\text{ otherwise}.
\end{cases}
\]
\end{corollary}

\subsection{The ECH chain complex}\label{sss:cc2} 

 In this section we prove Proposition \ref{prop:bijectionZ2} and Corollary \ref{thm:ECHS3}.  We first prove several lemmas. Recall that by Definition~\ref{def:degree} the formula for the degree of a generator is $
d(b^Bh^He^E)=2qB+qH+2E$. Our first lemma sorts the generators by degree.

\begin{lemma}\label{lem:dqxy}
\begin{enumerate}[{\em (i)}]
\itemsep-.25em
\item In degrees $2i, i=0,\dots,q-1$, there is one generator, $e^i$.
\item In degrees $q+2i, i=0,\dots,q-1$, there is one generator, $he^i$.
\item Each other generator can be written in the form $b^yh^He^{qx+i}$ where $h^He^i$ is as described in (i) or (ii) above. Given two such generators $b^yh^He^{qx+i}$ and $b^{y'}h^{H'}e^{qx'+i'}$, they have the same degree if and only if $i=i'$ and $x+y=x'+y'$. This degree is $d(h^He^i)+2q(x+y)$.
\end{enumerate}
\end{lemma}

\begin{proof}
For all three conclusions, the fact that the described generators have the given degrees is a simple computation.

The reverse implication in (i), that the $e^i$ are the only generators of degree $2i, i=0,\dots,q-1$, follows from the fact that the degree is even and less than $2q$, thus $H=B=0$. Then $E=i$ follows from the formula for degree.

The reverse implication in (ii), that the $he^i$ are the only generators of degree $q+2i, i=0,\dots,q-1$, follows from the fact that the degree is odd and $q+2i-q$ is less than $2q$, thus $H=1$ and $B=0$. Then $E=i$ follows from the formula for degree.

Given a generator $b^Bh^He^E$, the first claim in (iii) follows by setting $i=E\mod q, x=\lfloor E/q\rfloor$, and $y=B$.

The second claim in (iii) has two implications. Assuming the degrees of $b^yh^He^{qx+i}$ and $b^{y'}h^{H'}e^{qx'+i'}$ are equal, we conclude $H=H'$ by parity of degree. Using the fact that
\[
d(b^Bh^He^E)=d(b^{B'}h^He^{E'}) \iff d(b^Be^E)=d(b^{B'}e^{E'}),
\]
we then compute
\[
d(e^i)+2q(x+y)=d(e^{i'})+2q(x'+y').
\]
Because $d(e^i), d(e^{i'})<2q$, this implies $d(e^i)=d(e^{i'})$ and $x+y=x'+y'$; part (i) implies $i=i'$.

The opposite implication in (iii) is immediate from the formula for degree.
\end{proof}

Our next lemma sorts generators with the same degree by their ECH index.

\begin{lemma}\label{lem:dy} If $d(b^yh^He^{qx+i})=d(b^{y'}h^He^{qx'+i})$ then
\[
I(b^yh^He^{qx+i})=I(b^{y'}h^He^{qx'+i})+2(y-y').
\]
\end{lemma}

\begin{proof}
This is a straightforward computation using Theorem \ref{thm:ECHI}. Note that it is crucial to use the fact that $x+y=x'+y'$, or equivalently, that $x-x'=y-y'$.
\end{proof}

Lemma \ref{lem:dy} can also be interpreted as the fact that, within a given degree, the ECH index increases with respect to the lexicographic order on triples $(B,H,E)$.

The final lemma explains how to sort generators of adjacent degrees by their ECH index.

\begin{lemma}\label{lem:dd+1}
For $d\geq 2q-1$, the lowest index occurring for a generator of degree $d+1$ is two higher than the highest index occurring for a generator of degree $d$.
\end{lemma}

\begin{proof}
We split the proof, based on degree, into three pairs of identities, each of which compares two ECH indices. We first explain where the three pairs of identities come from.

From Lemma \ref{lem:dqxy}, we know that generators appear in an alternating ascending pattern according to degree, that the parity of degree and multiplicity $H$ of $h$ match, and that within a given degree, the sum of the multiplicity of $b$ and the floor of the multiplicity of $e$ divided by $q$ is constant. Thus we have two basic junctures at which we must verify the conclusion of the lemma: when degree changes from odd to even and when degree changes from even to odd. Each of Tables \ref{table:1}-\ref{table:3} thus presents a scenario where we are checking two such junctures.

The reason we present three different tables, determined by $E\mod q$, is because the form the generators take depends on degree, meaning the specific identities we must check to prove the lemma differ. Thinking of the generators $\emptyset, e, \dots, e^{q-1}$ as the basic set of generators, Lemma \ref{lem:dqxy} indicates how all other generators are obtained from this basic set, and how their degrees correspond. Essentially, for the lower (respectively, upper) half of values of $E\mod q$, the adjacent odd generators are of the form $b^Bhe^{E'}$ with $E'$ in the upper (respectively, lower) half of numbers modulo $q$. As there are precisely $q$ possible values of $E\mod q$, we must also handle the case when $E\mod q$ takes on the middle value, $(q-1)/2$, separately (Table \ref{table:2}).

This is most easily seen in Table \ref{table:gen2}(b), where $q=5$. For example, Table $\ref{table:1}$ includes the cases of degrees 9, 10, and 11 as well as 11, 12, and 13; Table \ref{table:2} addresses degrees 13, 14, and 15; finally Table \ref{table:3} handles degrees 15, 16, and 17 as well as 17, 18, and 19. At this point, we have reached a degree equal to $9\mod2q$ (equivalently, value of $E\mod q$), whereupon we repeat the argument, starting with Table \ref{table:1} and degrees 19, 20, and 21.

Note that within each case, we also need to identify the highest and lowest ECH index representatives of the shared degree; this identification relies on Lemma \ref{lem:dy}, which is used implicitly when constructing the tables to sort the generators vertically by ECH index. The only odd degree generators presented are those with the highest and lowest indices in their respective degrees (whichever we need to prove has ECH index adjacent to that of the lowest/highest ECH index representatives of the even degree under consideration).

The cases presented prove the lemma because they represent each pair of consecutive degrees where one entry in the pair has degree $2(qm+i)$ for $i=0,\dots,q-1$ and any $m\geq1$; this covers all adjacent pairs of degrees starting with $2q-1$ and $2q$ (by setting $m=1$ and $i=0$).

\medskip

\noindent\textbf{Case 1, $\mathbf{E\equiv_q 0,\cdots,(q-3)/2}$:} By Lemma \ref{lem:dqxy}, for $i=0,\dots,(q-3)/2$, the generators in consecutive degrees $2(qm+i)-1, 2(qm+i), 2(qm+i)+1$ are (sorted vertically within degree by increasing ECH index, using Lemma \ref{lem:dy}) as depicted in Table \ref{table:1}.

\begin{table}[h!]
\centering
\begin{tabular}{||c|c||}
\hline
degree & generator
\\\hline
$2(qm+i)-1$ & $b^{m-1}he^{i+\frac{q-1}{2}}$
\\\hline
\multirow{2}{*}{$2(qm+i)$} & $e^{qm+i}$ \\ & $\vdots$\\ & $b^me^i$
\\\hline
$2(qm+i)+1$ & $he^{q(m-1)+i+\frac{q+1}{2}}$
\\\hline
\end{tabular}
\caption{Here we have shown the generator of degree $2(qm+i)-1$ of highest ECH index, the generators of degree $2(qm+i)$ of lowest and highest ECH index, and the generator of degree $2(qm+i)+1$ of lowest ECH index.}
\label{table:1}
\end{table}

As shown in Table \ref{table:1}, in the case of consecutive degrees $2(qm+i)-1, 2(qm+i)$, and $2(qm+i)+1$ with $i=0,\dots,(q-3)/2$, we must prove that
\[
I(b^{m-1}he^{i+\frac{q-1}{2}})+2=I(e^{qm+i}) \mbox{ and } I(b^me^i)+2=I(he^{q(m-1)+i+\frac{q+1}{2}}).
\]

\medskip

\noindent\textbf{Case 2, $\mathbf{E\equiv_q (q-1)/2}$:} When $i=(q-1)/2$, the generators in these consecutive degrees are displayed in Table \ref{table:2}.

\begin{table}[h!]
\centering
\begin{tabular}{||c|c||}
\hline
degree & generator
\\\hline
$2qm+q-2$ & $b^{m-1}he^{q-1}$
\\\hline
\multirow{2}{*}{$2qm+q-1$} & $e^{qm+\frac{q-1}{2}}$ \\ & $\vdots$\\ & $b^me^{\frac{q-1}{2}}$
\\\hline
$2qm+q$ & $he^{qm}$
\\\hline
\end{tabular}
\caption{Here we have shown the generator of degree $2qm+q-2$ of highest ECH index, the generators of degree $2qm+q-1$ of lowest and highest ECH index, and the generator of degree $2qm+q$ of lowest ECH index.}
\label{table:2}
\end{table}

As shown in Table \ref{table:2}, in the case of consecutive degrees $2qm+q-2, 2qm+q-1$, and $2qm+q$, we must prove that
\[
I(b^{m-1}he^{q-1})+2=I(e^{qm+\frac{q-1}{2}}) \mbox{ and } I(b^me^{\frac{q-1}{2}})+2=I(he^{qm}).
\]

\medskip

\noindent\textbf{Case 3, $\mathbf{E\equiv_q (q+1)/2,\dots,q-1}$:} When $i=(q+1)/2,\dots,q-1$, the generators in consecutive degrees $2(qm+i)-1, 2(qm+i), 2(qm+i)+1$ are given in Table \ref{table:3}.

\begin{table}[h!]
\centering
\begin{tabular}{||c|c||}
\hline
degree & generator
\\\hline
$2(qm+i)-1$ & $b^mhe^{i-\frac{q+1}{2}}$
\\\hline
\multirow{2}{*}{$2(qm+i)$} & $e^{qm+i}$ \\ & $\vdots$\\ & $b^me^i$
\\\hline
$2(qm+i)+1$ & $he^{qm+i-\frac{q-1}{2}}$
\\\hline
\end{tabular}
\caption{Here we have shown the generator of degree $2(qm+i)-1$ of highest ECH index, the generators of degree $2(qm+i)$ of lowest and highest ECH index, and the generator of degree $2(qm+i)+1$ of lowest ECH index.}
\label{table:3}
\end{table}

As shown in Table \ref{table:3}, for the consecutive degrees $2(qm+i)-1, 2(qm+i)$, and $2(qm+i)+1$ with $i=(q+1)/2,\dots,q-1$, we must show that
\[
I(b^mhe^{i-\frac{q+1}{2}})+2=I(e^{qm+i}) \mbox{ and } I(b^me^i)+2=I(he^{qm+i-\frac{q-1}{2}}).
\]

All six identities can be checked using Theorem \ref{thm:ECHI} (using computer algebra system helps us avoid mistakes).
\end{proof}

To conclude this section, we first prove Proposition \ref{prop:bijectionZ2}, which computes the action-filtered chain complexes $ECC^{L(\varepsilon)}_*(S^3,\lambda_{2,q,\varepsilon},J)$ and the maps between them necessary to set up a directed system with $\varepsilon\to0$.

\begin{proof}[Proof of Proposition \ref{prop:bijectionZ2}] Lemma \ref{lem:dqxy}(i, ii) provides us with one generator of each even degree from zero to $q-1$, one generator of each degree from $q-1$ to $2q-1$, and one generator of each odd degree from $2q-1$ to $3q-2$. Using Lemma \ref{lem:dqxy}(iii), we can obtain generators of every degree $d\geq2q$, since each such $d$ can be written as $m+2qn$ for some positive integer $n$ and $m$ one of the degrees described in the previous sentence.

Lemmas \ref{lem:dqxy}(iii) and \ref{lem:dy} show that generators with the same degree have different (even) ECH indices. Lemma \ref{lem:dd+1} shows that for degrees at least $2q$, not only does ECH index increase by degree, but it increases by the smallest amount possible (two). Thus to show that the set of action filtered generators is in bijection with some subset of the nonnegative even integers, it remains to show that the generators of degrees $0,\dots,2q-1$ have ECH indices $0,\dots,I(\alpha)-2$, where $\alpha$ is the generator of degree $2q$ with the smallest ECH index.

By Lemma \ref{lem:dqxy}, these generators are $e^i, i=0,\dots,q-1$ or $he^j, j=0,\dots,(q-1)/2$, and $\alpha=e^q$. By Theorem \ref{thm:ECHI},
\begin{align*}
&I(e^i)=2i+\left\lfloor\frac{2i}{q}\right\rfloor(2i-q+1)=\begin{cases}2i&\text{if }0\leq i \leq(q-1)/2
\\4i-q+1&\text{if }(q+1)/2\leq i \leq q-1,
\end{cases}
\\&I(he^j)=q+4j+1, \mbox{ and }
\\&I(e^q)=3q+1.
\end{align*}
These indices are precisely the even numbers between $0$ and $q^2+q+2$, inclusive: the first quarter are the indices $2i$ of $e^i, i=0,\dots,(q-1)/2$, and then the rest come in consecutive pairs with $I(he^j)+2=I(e^i)$, as it is easy to check from the above formulas.

It remains to show that the inclusion-induced maps and cobordism maps compose to the canonical bijection (in either direction of the commutative diagram \cite[(2.16)]{preech}): the proof is practically identical to the analogous part of the proof of \cite[Prop.~3.2]{preech}, which appears in \cite[\S3.4]{preech}. This necessitates checking a list of conditions from \cite{cc2}, which are summarized in \cite[Lem.~2.18]{preech} and contained in \S \ref{s:cobordismfun}.
\end{proof}

\section{Spectral invariants of ECH}\label{s:spectral}

In this section we compute the ECH spectrum for $(S^3,\lambda_{2,q})$ and knot filtered ECH for $(S^3,\xi_{std})$ with respect to the standard transverse right handed $T(2,q)$ knot with rotation angle $2q + \delta$, where $\delta$ is either 0 or a sufficiently small positive irrational number. 
First, we establish the relationship between the ECH index of a generator and its degree in \S \ref{ss:index-degree}; this governs the behavior of the ECH spectral invariants.  In \S \ref{ss:ECHspectrum} we review basic properties of the ECH spectrum. In Proposition \ref{prop:ck} we prove that 
\[
c_k(S^3,\lambda_{2,q}) = N_k({1/2}, {1/q})
\]
using our chain complex described in \S\ref{s:topology}-\ref{s:ECHI}.

 In \S \ref{ss:kECH} we review the basics of knot filtered ECH and compute the knot filtration by relating it to the degree of Reeb current generators.  We then establish Theorem \ref{thm:kech-intro}.

\subsection{Relationship between index and degree}\label{ss:index-degree}
We first compute the function $I=2k\mapsto d$, which will govern the computations of both the ECH spectrum and knot filtered ECH.  Recall that from Definition \ref{def:degreep} and Remark \ref{def:degree} that the degree of a Reeb current is  $ d(b^Bh^He^E)= 
2qB+qH+2E.$

For $a,b\in\R$, let $N(a,b)$ denote the sequence $(am+bn)_{m,n \in \Z_{\geq0}}$ of nonnegative integer linear combinations of $a$ and $b$, written in increasing order with multiplicity. 
We use $N_k(a,b)$ to denote the $k^\text{th}$ element of this sequence, including multiples and starting with $N_0(a,b)=0$.

\begin{lemma}\label{lem:degNk} When $k$ is small enough relative to $L(\varepsilon)$, cf. Lemma \ref{lem:orbitseh}, the degree of any Reeb current whose homology class represents the generator of the group $ECH_{2k}^{L(\varepsilon)}(S^3,\lambda_{2,q,\varepsilon})$ is $N_k(2,q)$.
\end{lemma}
\begin{proof}
The fact that each group $ECH_{2k}^{L(\varepsilon)}(S^3,\lambda_{2,q,\varepsilon})$ is generated by the homology class of a single Reeb current, all of which are cycles, follows from the computation of the differential in \S\ref{s:ECHI}.  The differential vanishes (Proposition \ref{prop:bijectionZ2}), and therefore the homology equals the chain complex.  (For $T(p,q)$ there will be a nonvanishing differential, and all even index generators will be closed.)

Let $(m,n)=(E,2B+H)$. This defines a bijection between ECH generators and $\Z^2_{\geq0}$ (to show surjectivity, note that $b^{\lfloor n/2\rfloor}h^{n-2\lfloor n/2\rfloor}e^m\mapsto(m,n)$; injectivity is easy to show). The fact that composing this bijection with $(m,n)\mapsto 2m+nq$ is monotonically increasing with respect to index follows from the fact that degree increases with respect to index, which is proven in Lemma \ref{lem:dd+1}.

\end{proof}

\subsection{The ECH spectrum }\label{ss:ECHspectrum}

We summarize the definitions and properties of action filtered ECH and the ECH spectrum, 
and we compute the ECH spectrum of $(S^3, \lambda_{2,q})$.

\begin{remark} Embedded contact homology contains a canonical class, called the \emph{contact invariant} $c(\xi) \in ECH(Y,\xi,0)$, which is the homology class of the cycle given by the empty set of Reeb orbits. 
\end{remark}

In the following, when $\Gamma$ is not specified, we define
\begin{equation}\label{e:decomp}
ECH_*(Y,\lambda) := \bigoplus_{\Gamma \in H_1(Y)} ECH_*(Y,\lambda,\Gamma).
\end{equation}
Recall that the symplectic action of a Reeb current $\A(\alpha)$ was defined in \eqref{eq:action} and that the ECH differential decreases symplectic action.  Thus for each $L\in \R$, there is a subcomplex $ECC^L(Y,\lambda,J)$ generated by the Reeb currents $\alpha$ for which $\A(\alpha) < L$, whose homology is the \emph{action filtered embedded contact homology} $ECH_*^L(Y,\lambda)$.  In \cite[Thm.~1.3]{cc2}, it is shown that action filtered ECH does not depend on $J$; it does however depend on the choice of contact form $\lambda$.  (In {the} literature, the word action is omitted; we include it to more readily distinguish action filtered ECH from knot filtered ECH.)

If $r>0$ is a constant then there is a canonical scaling isomorphism
\begin{equation}\label{e:scale}
ECH^L(Y,\lambda) = ECH^{rL}(Y,r\lambda)
\end{equation}
because $\lambda$ and $r \lambda$ have the same Reeb orbits up to reparametrization.  Moreover, for any generic $\lambda$-compatible almost complex structure $J$ there exists a unique $r\lambda$-compatible almost complex structure $J^r$ which agrees with $J$ on the contact planes, thus the bijection on Reeb orbits gives an isomorphism at the level of chain complexes:
\begin{equation}\label{eq:scale}
ECC^L(Y,\lambda, J) = ECC^{rL}(Y,r\lambda, J^r).
\end{equation}
For $L\leq L'$ there are also maps induced by the inclusion of chain complexes: 
\begin{equation}\label{eq:incl}
\begin{split}
\iota: ECH^L(Y,\lambda) \longrightarrow & \ ECH(Y,\lambda), \\
\iota^{L,L'}: ECH^L(Y,\lambda) \longrightarrow & \ ECH^{L'}(Y,\lambda).
\end{split}
\end{equation}
None of the maps in \eqref{eq:scale} and \eqref{eq:incl} depend on $J$ as a result of \cite[Thm.~1.3]{cc2}.

\begin{remark}[$U$-map]
If $Y$ is connected there is a degree -2 map
\[
U: ECH(Y,\lambda,\Gamma) \to ECH(Y,\lambda,\Gamma),
\]
which is induced by a chain map which is defined similarly to the differential.  However, instead of counting $J$-holomorphic curves in $\R \times Y$ with ECH index one modulo translation, it counts ECH index two curves that pass through a chosen generic point $z \in \R \times Y$, see \cite[\S 2.5]{shs}.  Taubes proved in \cite{taubesechswf5} that $U$ agrees with an analogous map on Seiberg-Witten Floer cohomology defined in \cite{KMbook} under the isomorphism of Theorem \ref{thm:taubes}.
\end{remark}

We now have all the necessary ingredients to define the ECH spectrum.

\begin{definition}\cite[\S 4]{qech}
Let $(Y, \lambda)$ be a closed connected\footnote{One can define the ECH spectrum for disconnected contact 3-manifolds, cf.~\cite[\S 1.5]{lecture}.} contact 3-manifold and assume that the contact invariant $c(\xi) \neq 0 \in ECH(Y,\xi,0)$.  The \emph{ECH spectrum} of $(Y,\lambda)$ is the sequence of real numbers
\[
0=c_0(Y,\lambda) < c_1(Y,\lambda) \leq c_2(Y,\lambda) \leq ... \leq \infty,
\]
defined as follows.  First suppose that $\lambda$ is nondegenerate. Then $c_k(Y,\lambda)$ is the infimum of $L$ such that there exists a class $\eta \in ECH^L(Y,\lambda,0)$ with $U^k\eta = c(\xi) = [\emptyset]$.  This is equivalent to $c_k(Y,\lambda)$ realizing the infimum over $L$ such that the generator of $ECH_{2k}(Y,\lambda,0)$ is in the image of the first inclusion induced map \eqref{eq:incl}.  If no such class exists $c_k(Y,\lambda) = \infty$, while $c_k(Y,\lambda) < \infty $ if and only if $c(\xi)$ is in the image of $U^k$ on $ECH(Y,\xi,0)$.

If $\lambda$ is degenerate, define 
\[
c_k(Y,\lambda):= \lim_{n \to \infty}c_k(Y,f_n\lambda),
\]
where $f_n:Y \to \R_{>0}$ are functions on $Y$ such that $f_n\lambda$ is nondegenerate  for each $n$ and $\lim_{n\to \infty}f_n =1$ in the $C^0$-topology.  In this setting, the spectral numbers $c_k$ still take values in the action spectrum of $\lambda$ and remain infimums over actions of admissible Reeb currents. 

\end{definition}

The ECH spectrum satisfies a number of nice properties, such as spectrality, monotonicity, and scaling. Thus it obstructs symplectic embeddings of symplectic manifolds with contact type boundary, and in many interesting cases the obstructions are sharp.  

We now compute the ECH spectrum of the degenerate contact form $\lambda_{2,q}$. 
\begin{proposition}\label{prop:ck} We have $c_k(S^3,\lambda_{2,q}) = N_k(1/2,1/q)$.
\end{proposition}
\begin{proof} Let $f_n=1+\frac{1}{n}\fp^*H_{2,q}$. Once $n$ is large enough so that $c_k(Y,f_n\lambda_{2,q})<L(1/n)$, the capacities $c_k(Y,f_n\lambda_{2,q})$ are constant in $n$, therefore it suffices to compute these for $k$ small enough with respect to $L(1/n)$ so that all orbits are of the form $b^Bh^He^E$ (see Lemma \ref{lem:orbitseh}).

By Lemma \ref{lem:efromL}(ii) 
\[
\A(b^B)=B, \ \ \  \A(h)=\frac{1}{2}, \ \ \ \A(e^E)=\frac{E}{q}.\]
Therefore
\begin{equation}\label{eqn:Ad}
\A(b^Bh^He^E)=\frac{d(b^Bh^He^E)}{2q}.
\end{equation}
The result follows from Proposition \ref{prop:bijectionZ2} and Lemma \ref{lem:degNk}. 

\end{proof}

\begin{remark} The result of Proposition \ref{prop:ck} could be indirectly obtained by \cite[Prop.~1.2]{qech}, because $\lambda_{2,q}$ is strictly contactomorphic (up to rescaling by a constant) to the standard contact form on $E(2,q)$. 
The strict contactomorphism follows from \cite{cgm,kl}.
\end{remark}

\subsection{Knot filtered ECH}\label{ss:kECH}

We collect the key properties of the knot filtration when $H_1(Y)=0$ from \cite{HuMAC}. See the introduction of the knot filtration in \S\ref{ss:overviewECH} for discussion of the consequences of these facts.

Let $b^m \alpha$ be a Reeb current where $m \in \Z_{\geq 0}$ and $\alpha$ be any Reeb current\footnote{When defining the knot filtration $\fb$, we need not assume that hyperbolic orbits have multiplicity 1, e.g. $\alpha$ does not have to be a generator of ECH.} not including $b$.   We define the \emph{knot filtration} with respect to $(b, \rot(b))$ by 
\begin{equation}\label{eq:filt}
\mathcal{F}_b(b^m\alpha) = m \rot(b) + \ell(\alpha,b),
\end{equation}
where $\ell(\alpha, b)$ is given by
\[
{\ell(\alpha,b)=\sum_im_i\ell(\alpha_i,b).}
\]

The ECH differential $\partial$ does not increase the knot filtration $\mathcal{F}_b.$
\begin{lemma}{\em \cite[Lem.~5.1]{HuMAC}}\label{lem:kdiff}
Let $(Y^3,\lambda)$ be a closed nondegenerate contact manifold. If $b^{m_+}\alpha_+$ and $b^{m_-}\alpha_-$ are Reeb currents, and there exists a $J$-holomorphic current \\ $\cur \in \M^J(b^{m_+}\alpha_+, b^{m_-}\alpha_-)$, then
\begin{equation}\label{eq:kdiff}
\mathcal{F}_b(b^{m_+}\alpha_+) \geq \mathcal{F}_b(b^{m_-}\alpha_-).
\end{equation}
In particular, the ECH differential $\partial$ does not increase the knot filtration $\mathcal{F}_b$.
\end{lemma}

If $K \in \R$  let $ECH_*^{\mathcal{F}_b \leq K}(Y,\lambda,J)$ denote the homology of the subcomplex generated by admissible Reeb currents $b^m\alpha$ with $\mathcal{F}_b (b^m\alpha)  \leq K$.  We have that $ECH_*^{\mathcal{F}_b\leq K}(Y,\lambda,J)$ is a topological invariant in the following sense.  

\begin{theorem}{\em \cite[Thm.~5.3]{HuMAC}}\label{thm:kech}
Let $(Y, \xi)$ be a closed contact 3-manifold with $H_1(Y)=0$, $b\subset Y$ be a transverse knot and $K\in \R$.  Let $\lambda$ be a contact form with $\ker \lambda = \xi$ such that $b$ is an elliptic Reeb orbit with rotation number $\rot(b)\in\R / \Q$.  
Let $J$ be any generic $\lambda$-compatible almost complex structure.  Then $ECH_*^{\mathcal{F}_b\leq K}(Y,\lambda,J)$ depends only on $Y, \xi, b, \rot(b),$ and $K$.  
\end{theorem}
In \S\ref{s:cobordismfun}, we generalize Theorem \ref{thm:kech} to allow for rational rotation numbers and provide a Morse-Bott direct limit means of its computation via an extension of the arguments employed in \cite{preech}. The proof relies upon a doubly filtered Morse-Bott direct limit argument and requires a knot admissible sequence of contact forms (see Definition \ref{def:knotadmissible}).  

This allows us to compute the knot filtered embedded contact homology of $(S^3,\xi_{std},T(2,q),2q)$ via successive approximations using the sequence $\{ \lambda_{2,q,\ve} \}$, which is knot admissible family by Lemma \ref{lem:orbtrivCZ2} and its surrounding discussion. We briefly review the computation for the nondegenerate unknot in the irrational ellipsoid.
\begin{example}\label{ex:kechellipsoid}\cite[\S 5]{HuMAC}
Let $Y=\partial E(a,b) = \left\{ (z_1,z_2) \in \C^2 \ | \ \pi \left( \frac{|z_1|^2}{a} + \frac{|z_2|^2}{b} \right)  =1 \right\},$ with $a,b \geq 0$.   Then for $\lambda_0= \frac{i}{2}\left( \sum_{j=1}^2 z_j d\bar{z}_j - \bar{z}_jd z_j \right)$ restricted to $Y$, 
\[
R_0 = 2\pi \left( \frac{1}{a}\frac{\partial}{\partial \theta_1} + \frac{1}{b}\frac{\partial}{\partial \theta_2}  \right).
\]
If $a/b \in \R \setminus \Q$, there are exactly two embedded Reeb orbits: $\gamma_1$ in the $z_2=0$ plane with action $a$ and $\gamma_2$ in the $z_1=0$ plane with action $b$.  The ECH generators are of the form $\alpha = \gamma_1^{m_1}\gamma_2^{m_2} $ and the ECH index\footnote{When $H_1(Y)=0$, the chain complex has an absolute $\Z$-grading, which we indicate by $ECC_*$.}  
 defines a bijection from the set of generators of $ECC_*(Y,\lambda_0,J)$ to the set of nonnegative even integers, cf. \cite[Lem.~4.1]{HuMAC}, where the grading of $\emptyset$ is 0. Thus,
 \[
ECH_*(Y,\xi_{std})=\begin{cases}
\Z/2&\text{ if $*\in2\Z_{\geq0}$,}
\\0&\text{ else.}
\end{cases}
\]

We have that $\rot(\gamma_2) = b/a$ and $\ell( \gamma_1,\gamma_2)=1$.  Thus the knot filtration of an ECH generator $\gamma_1^{m_1}\gamma_2^{m_2}$ with respect to $\gamma_2$ is given by
\[
\mathcal{F}_{\gamma_2}(\gamma_1^{m_1},\gamma_2^{m_2}) = m_2b/a + m_1 = a^{-1}\A(\gamma_1^{m_1},\gamma_2^{m_2}).
\] 
Thus if $\alpha$ is an arbitrary ECH generator with $I(\alpha) = 2k$ then 
\[
\mathcal{F}_{\gamma_2}(\alpha) = N_k(1,b/a).
\]
It follows that if $k$ is a nonnegative integer, $b_0$ is the standard transverse unknot\footnote{In tight contact 3-manifolds, \cite{yasha, torus} demonstrate that the self-linking number is a complete invariant of transversal isotopy for transversal unknots and torus knots, respectively.} given by a Hopf circle, and $\rot(b_0) \in \R \setminus \Q$ then
\[
ECH_{2k}^{\mathcal{F}_{b_0} \leq K}(S^3,\xi_{std},b_0,\rot(b_0))=\begin{cases}\Z/2&K\geq N_k(1,\rot(b_0)),
\\0&\text{otherwise,}\end{cases}
\]
and in all other gradings $*$, $ECH_*^{\mathcal{F}_{\gamma_2} \leq K}(S^3,\xi_{std},b_0,\rot(b_0))=0.$
\end{example}

Before computing knot filtered ECH  for $(S^3, \xi_{std})$ with respect to the right handed $T(p,q)$ torus knot, we first compute the knot filtration.

\begin{proposition}\label{prop:Fb}
For $(S^3,\lambda_{2,q,\varepsilon})$ and $\varepsilon(L)$ as in  Proposition \ref{prop:morse2},  for any Reeb current $\alpha$ not including the right handed $T(2,q)$ torus knot  ${b}$, 
\[
\fb({b}^B\alpha) = d({b}^B \alpha) + B\delta_{b,L}.
\]
\end{proposition}

\begin{proof}
By Lemma \ref{lem:pagetrivCZ2}, we know $\rot(b)=2q+\delta_{b,L}$.  We have that $\alpha = h^He^E$ for $H,E \in \Z_{\geq 0}$, thus
\[
\begin{split}
\fb(b^B \alpha) = & \ B \rot(b) + \ell(\alpha,b) \\
= & \ B(2q + {\delta_{b,L}})  + \ell(e^E,b) + \ell(h^H,b) \\
=& \ B(2q + {\delta_{b,L}}) + qH + 2E,
\end{split}
\]
where the last line follows from Corollary \ref{cor:linkingnumber}. Recall that 
\[
d(b^Bh^He^E) = 2qB + qH + 2E.
\]

\end{proof}
\begin{remark}
Note that $d(b^B \alpha) =  2q \ \A_{\lambda_{2,q}}(b^B \alpha)$, thus  knot filtered ECH is able to realize the relationship between action and linking in these examples, cf. \cite{bhs}. 
\end{remark}

Recall that $T(2,q)$ is standard as a transverse knot in the sense of Etnyre \cite{torus}, i.e. it has maximal self-linking $2q-2-q$.\footnote{We computed in \S \ref{ss:chern} that $c_\Sigma([\Sigma]) = 2+q-2q$, thus  $2q-2-q$ is the self linking which corresponds to rotation $2q+\delta_L$ with respect to the pushoff linking number zero trivialization, as explained in \S \ref{ss:pagetau}.}  We now compute $T(2,q)$ knot filtered ECH:

\begin{theorem}\label{thm:kECH} Let $\xi_{std}$ be the standard tight contact structure on $S^3$.  
Let ${b_0}$ be the standard positive ({right-handed}) transverse $T(2,q)$ torus knot for $q$ odd and positive.  Then for $k \in \N$,
\[
ECH_{2k}^{\fb \leq K}(S^3,\xi_{std},{b_0},2q)=\begin{cases}\Z/2&K\geq{ N_k(2,q)  },
\\0&\text{otherwise,}\end{cases}
\]
and in all other gradings $*$,
\[
ECH_*^{\fb \leq K}(S^3,\xi_{std},{b_0},2q)=0.
\]
If $\delta$ is a sufficiently small positive irrational number, then up to grading $k \in \N$ and knot filtration threshold $K$ inversely proportional to $\delta$,
\[
ECH_{2k}^{\fb \leq K}(S^3,\xi_{std},{b_0},2q+\delta)=\begin{cases}\Z/2&K\geq{ N_k(2,q) + \delta {(\$N_k(2,q) -1)}},
\\0&\text{otherwise,}\end{cases}
\]
where { $\$N_k(2,q)$ is the number of repeats in $\{ N_j(2,q)\}_{j\leq k}$ with value $N_k(2,q)$,} and in all other gradings $*$, up to the threshold inversely proportional to $\delta$,
\[
ECH_*^{\fb \leq K}(S^3,\xi_{std},{b_0},2q+\delta)=0.
\]

\end{theorem}

\begin{remark}
The relationship between the threshold of the grading $2k$ and the size of $\delta$ is as follows.  We require $\delta$ to be small enough so that $N_k(2,q)+\delta(\$N_k(p,q)-1) \leq N_{k+1}(2,q)$ for all $k$. \end{remark}

\begin{proof}

Since $\delta$ is small and the knot filtration is invariant of the contact form (so long as $b$ is an elliptic Reeb orbit with irrational rotation number $2q+\delta$ in its Seifert surface trivialization), we use one of our preferred contact forms $\lambda_{2,q,\varepsilon}$ where $\delta:=\delta_{b,L(\varepsilon)}=$ to compute the ECH chain complex up to the action and index thresholds determined by Lemma \ref{lem:orbitseh}.  We require $\delta$ to be small enough so that $N_k(2,q)+\delta(\$N_k(2,q)-1) \leq N_{k+1}(2,q)$ for all $k$.

Proposition \ref{prop:bijectionZ2} tells us that the lower bound on $K$ is precisely the knot filtration level of the generator of $ECC^{L(\varepsilon)}_{2k}(S^3,\lambda_{2,q,\varepsilon})$, which is $N_k(2,q)+\delta(\$N_k(2,q)-1)$ by Proposition \ref{prop:Fb} and Lemma \ref{lem:degNk}. By Theorem \ref{thm:introok} and the discussions in \S \ref{ss:direct}-\ref{ss:directlimit}, we can take direct limits to realize $\delta=0$.
\end{proof}

\section{Cobordism maps on embedded contact homology}\label{s:cobordismfun}

In this section we establish Theorem \ref{thm:introok}, which extends the definition and invariance of knot filtered embedded contact homology to knots with rational rotation numbers.  This is accomplished by a generalization and refinement of the direct systems established for our action filtered Morse-Bott arguments in \cite[\S 7.1]{preech}, which utilized the cobordism maps induced by filtered perturbed Seiberg-Witten Floer cohomology.

To make this section more self contained, we first briefly review a number of results regarding the existence and properties of cobordism maps for action filtered embedded contact homology and the relation to their counterpart in Seiberg-Witten Floer cohomology as established by Hutchings and Taubes \cite{cc2} in \S \ref{ss:def}-\ref{ss:action}. In \S \ref{ss:direct}, we explain how to construct action filtered direct systems via cobordism maps coming from filtered perturbed Seiberg-Witten Floer cohomology.\footnote{The contact and symplectic form perturbations of monopole Floer cohomology originated in Taubes' proof of the Weinstein conjecture in dimension three \cite{wconj, wconj2} and were also used in his proof of the isomorphism between embedded contact homology and Seiberg-Witten Floer cohomology \cite{taubesechswf}-\cite{taubesechswf5}.}  In \S \ref{ss:directlimit}, we set up the doubly filtered direct system, complete the direct limit argument, and establish invariance, which proves Theorem \ref{thm:introok}.

\subsection{Exact symplectic cobordisms and broken currents}\label{ss:def}
We now collect a number of definitions and some deep facts about the existence of certain broken holomorphic currents in exact symplectic cobordisms.

An \emph{exact symplectic cobordism} from $(Y_+,\lambda_+)$ to $(Y_-,\lambda_-)$ is a pair $(X,\lambda)$ where $X$ is a compact 4-dimensional oriented manifold with $\partial X = Y_+ - Y_-$ and $d\lambda$ is a symplectic form on $X$ with $\lambda|_{Y_\pm} = \lambda_\pm$. Given an exact symplectic cobordism $(X,\lambda)$, we form its \emph{completion}
\[
\overline{X} = ((-\infty,0] \times Y_-) \sqcup_{Y_-} X\sqcup_{Y_+} ( [0,\infty) \times Y_+)
\]
using the gluing under the following identifications.  A neighborhood of $Y_+$ in $(X,\lambda)$ can be canonically identified with $(-\epsilon, 0]_s \times Y_+$ for some $\epsilon>0$ so that $\lambda$ is identified with $e^s\lambda_+$.  Moreover, this identification is defined so that $\partial_s$ corresponds to the unique vector field $V$ such that $\iota_Vd\lambda = \lambda$.  Similarly, a neighborhood of $Y_-$ in $X$ can be canonically identified with $[0,\epsilon) \times Y_-$ so that $\lambda$ is identified with $e^s\lambda_-$.

Let $(X,\lambda)$ be an exact symplectic cobordism from $(Y_+,\lambda_+)$ to $(Y_-,\lambda_-)$.  An almost complex structure $J$ on the completion $\overline{X}$ is said to be \emph{cobordism compatible} if $J$ is $d\lambda$-compatible on $X$ (meaning that $d\lambda(\cdot, J \cdot)$ is a Riemannian metric on $X$), and there are $\lambda_\pm$-compatible almost complex structures $J_\pm$ on $\R \times Y_\pm$ such that $J$ agrees with $J_+$ on $[0,\infty) \times Y_+$ and with $J_-$ on $(-\infty,0] \times Y_-$.

{In order to define direct systems, we will need to compose cobordisms, $X_- \circ X_+$, compose cobordism compatible almost complex structures, and understand the maps they induce on embedded contact homology. 

\begin{definition}
Given exact symplectic cobordisms $(X_+,\lambda_+)$ from $(Y_1,\lambda_1)$ to $(Y_0,\lambda_0)$ and $(X_-,\lambda_-)$ from $(Y_0,\lambda_0)$ to $(Y_{-1},\lambda_{-1})$, we glue along $Y_0$ to define their \emph{composition}
\[
X_- \circ X_+ := X_- \sqcup_{Y_0} X_+,
\]
which we equip with the exact symplectic form $\lambda_- \circ \lambda_+$ obtained by gluing $\lambda_-$ and $\lambda_+$ together.  This produces an exact symplectic cobordism from $(Y_1,\lambda_1)$ to $(Y_{-1},\lambda_{-1})$.
 
For $R \geq 0$, we can construct the \emph{stretched composition}
\[
X_- \circ_{R} X_+ := X_- \sqcup_{Y_0} ([-R , R] \times Y_0) \sqcup X_+.
\]
We glue $e^{\pm R} \lambda_\pm$ on $X_\pm$ and $e^s \lambda_0$ on $[-R , R] \times Y_0$ to obtain a one-form $\lambda_- \circ_R \lambda_+$ on $X_- \circ_{R} X_+$, thereby producing an exact symplectic cobordism from $(Y_1,e^R\lambda_1)$ to $(Y_{-1},e^{-R}\lambda_{-1})$.

(Our notation here is not entirely ideal because, previously for an exact symplectic cobordism $X$ from $Y_+$ to $Y_-$, $\lambda_\pm$ were contact forms on $Y_\pm$, but in this definition they are primitives of the exact symplectic forms on the exact symplectic cobordisms being composed.)
 \end{definition}

\begin{notation}\label{n:cobord}
For positive $\varepsilon'<\varepsilon$ and $s \in [\varepsilon',\varepsilon]$, we consider exact symplectic cobordisms from $(S^3,\lambda_{2,q,\varepsilon})$ to $(S^3,\lambda_{2,q,\varepsilon'})$, which we denote by
 \[
 (X_{[\varepsilon',\varepsilon]}, \lambda_{2,q, [\varepsilon',\varepsilon]}) :=([\varepsilon',\varepsilon] \times S^3, (1+s\frak{p}^*H_{2,q})\lambda_{2,q}).
\]
For $\varepsilon''<\varepsilon'<\varepsilon$ we will consider the exact symplectic cobordism from $(S^3,\lambda_{2,q,\varepsilon})$ to $(S^3,\lambda_{2,q,\varepsilon''})$ formed from the composition of two  exact symplectic cobordisms:
\[
(X_{[\varepsilon'',\varepsilon']\circ[\varepsilon',\varepsilon]}, \lambda_{2,q, [\varepsilon'',\varepsilon']\circ[\varepsilon',\varepsilon]}):=(X_{[\varepsilon'',\varepsilon']}, \lambda_{2,q, [\varepsilon'',\varepsilon']}) \circ  (X_{[\varepsilon',\varepsilon]}, \lambda_{2,q, [\varepsilon',\varepsilon]}).
\]

\end{notation}

\begin{remark}
We can compose cobordism compatible almost complex structures as follows.  Let $J_i$ be a $\lambda_i$-compatible almost complex structure on $\R \times Y_i$ for $i=-1,0,1$.  Let $J_\pm$ be cobordism compatible almost complex structures on the completions $\overline{X}_\pm$ that restrict to $J_{\pm1}$ and $J_0$ on the ends.  For each $R \geq 0$ we glue $J_-$, $J_0$, and $J_+$ to define an almost complex structure $J_- \circ_R J_+$ on $\overline{X_- \circ_{R} X_+}$; when $R=0$ we obtain a cobordism compatible almost complex structure on $\overline{X_- \circ X_+}$, which we denote by $J_- \circ J_+$.
\end{remark}

Following \cite[\S 5.1]{cc2} we will need to consider a \emph{strong homotopy} $\left(X,\{ \lambda_t\}_{t\in[0,1]}\right)$ of exact symplectic cobordisms from $(Y_+,\lambda_+)$ to $(Y_-,\lambda_-)$, where $X$ is a compact four-manifold with boundary $\partial X = Y_+ - Y_-$ and $\{\lambda_t\}$ is a smooth family of 1-forms on $X$ that is independent of $t$ near $\partial X$, such that for each $t$, the form $d\lambda_t$ is symplectic and $\lambda_t|_{Y_\pm} = \lambda_\pm$.

Finally, we define the notion of a \emph{broken $J$-holomorphic current} on $(\overline{X},J)$.  These arise in connection with the maps induced by cobordisms on embedded contact homology. 

\begin{definition}\label{d:bcur}
Let $\alpha_+$ and $\beta_-$ be Reeb currents respectively associated to $\lambda_+$ and $\lambda_-$.   Let $\M^J(\alpha_+,\beta_-)$ be the set of $J$-holomorphic currents in $(\overline{X},J)$ from $\alpha_+$ to $\beta_-$.  A \emph{broken $J$-holomorphic current} from $\alpha_+$ to $\beta_-$ is a tuple $\cur=(\cur_{N_-},...,\cur_0,...,\cur_{N_+})$ where ${N_-}\leq0\leq{N_+}$ such that there are distinct Reeb currents $\beta_-=\beta_-(N_-),...,\beta_-(0)$ for $(Y_-,\lambda_-)$ and $\alpha_+(0),...,\alpha_+(N_+)=\alpha_+$ for $(Y_+,\lambda_+)$ such that:
\begin{itemize}
\itemsep-.35em
\item If $k>0$ then $\cur_k\in \M^{J_+}(\alpha_+(k),\alpha_+(k-1))/\R$.
\item $\cur_0 \in \M^J(\alpha_+(0),\beta_-(0))$;
\item If $k<0$ then $\cur_k\in \M^{J_-}(\beta_-(k+1),\beta_-(k))/\R$;
\end{itemize}
The currents $\cur_k$ are called the \emph{levels} of (the broken holomorphic current) $\cur$.  We denote the set of broken holomorphic currents by $\overline{\M^J(\alpha_+,\beta_-)}$.  The ECH index of the broken $J$-holomorphic current is defined to be the sum of the ECH indices of its levels,
\[
I(\cur) = \sum_{i=N_-}^{N_+}I(\cur_i)
\]
(See \cite[\S 4.2]{Hrevisit} for the definition of the ECH index in cobordisms.)
\end{definition}

Let $\lambda_\pm$ be nondegenerate and $J_\pm$ be generic so that the chain complexes $ECC(Y_\pm,\lambda_\pm,J_\pm)$ are defined.  We say that a linear map
\begin{equation}\label{eq:chain}
\phi: ECC(Y_+,\lambda_+,J_+) \to ECC(Y_-,\lambda_-,J_-) 
\end{equation}
\emph{counts (broken) $J$-holomorphic currents} if $\langle \phi \alpha_+,\beta_- \rangle \neq 0$ implies that the set $\overline{\M^J(\alpha_+,\beta_-)}$ is nonempty.

The linear map $\phi$ in \eqref{eq:chain} of interest is a noncanonical chain map defined by counting solutions to the Seiberg-Witten equations on $\overline{X}$, where the Riemannian metric is determined by $\lambda$ and $J$.    There is a perturbation of the four dimensional Seiberg-Witten equations on an exact symplectic cobordism \cite[\S 4.2]{cc2}, which is closely related to the contact form perturbation of the three dimensional Seiberg-Witten equations.  In particular, one uses a 2-form $r \hat{\omega} = r \sqrt{2}\widetilde{\omega}/{|\widetilde{\omega}|}$ on $\overline{X}$ where $r$ is a very large positive constant and $ \widetilde{\omega} = d\widetilde{\lambda}$ with $\widetilde{\lambda}$ being a slightly nonstandard choice of 1-form.\footnote{This is for the sake of consistency with Taubes' work \cite{wconj} and \cite{taubesechswf}-\cite{taubesechswf5}, as there are some factors of $2$ that appear; see \cite[Rem.~2.2, 4.2]{cc2}.  Otherwise one may have instead expected to have defined $\widetilde{\lambda}$ by extending the 1-form $\lambda$ on $\overline{X}$ to agree with $e^s\lambda_+$ on $[0,\infty) \times Y_+$ and with $e^s\lambda_-$ on $(-\infty,0] \times Y_-$. }  Up to insignificant factors, $\hat{\omega} $ agrees with $d\lambda$ on $X$ and with $d\lambda_\pm$ on its ends.  In \cite[\S 7]{cc2}, it is explained how for $r$ sufficiently large, the Seiberg-Witten solutions in a cobordism give rise to broken holomorphic currents, which are counted by the map $\phi$.  In particular, it is shown that given a sequence of solutions to the $r$-perturbed Seiberg-Witten equations with $r \to \infty$, the zero set of one component of the spinor converges to a broken holomorphic current.  
 
 However, the chain map $\phi$ is noncanonical because it is not unique; it depends on $r$, and even for fixed $r$, additional perturbations of the Seiberg-Witten equations are needed.\footnote{One uses the exact 2-forms $\mu_\pm$ on $Y_\pm$ and $\mu$ on $\overline{X}$, which agrees with $\mu_\pm$ on the ends, which appear in the contact form perturbation of the three dimensional Seiberg-Witten equations.  One must choose $\mu$ so that its derivatives up to some sufficiently large, but constant order have absolute value less than 1/100.} In particular, two different perturbations may give rise to different chain maps.  Fortunately, when this happens, Hutchings and Taubes have shown in \cite[Prop.~5.2(b)]{cc2} that they are capable of choosing a homotopy between the two perturbations, and that the two chain maps will then differ by a chain homotopy, which counts solutions to the Seiberg-Witten equations for perturbations in the homotopy, provided that $r$ is sufficiently large.  The proof of this is carried out in \cite[\S 7.6]{cc2} and is similar to the proof that the chain maps count holomorphic currents.

The noncanonical chain map $\phi$ from \eqref{eq:chain}, induced by an exact symplectic cobordism $(X,\lambda)$ from $(Y_+,\lambda_+)$ to $(Y_-,\lambda_-)$, induces canonical cobordism maps
\begin{equation}\label{eq:Lchain}
\Phi^L(X,\lambda): ECH^L(Y_+,\lambda_+) \to ECH^L(Y_-,\lambda_-),
\end{equation}
satisfying a number of wonderful properties, as established in \cite[Thm.~1.9]{cc2}, which we will soon review. 

\begin{remark}
 Our favorite property is the ``Holomorphic Curves Axiom", which guarantees that $\Phi^L(X,\lambda)$ is induced by a noncanonical chain map $\phi$, which counts (broken) $J$-holomorphic currents. The existence of a $J$-holomorphic curve allows us to draw geometric conclusions, like intersection positivity, which will be key to the proof that the knot filtration is preserved by $\phi$, and allows us to conclude that knot filtered ECH is a topological invariant.
\end{remark}

 In particular, for any cobordism compatible $J$, the cobordism map $\Phi^L(X,\lambda)$ is induced by a noncanonical chain map $\phi$ such that the coefficient $\langle \phi \gamma_+, \gamma_- \rangle$  is nonzero only if there exists a broken $J$-holomorphic current from $\gamma_+$ to $\gamma_-$.  Moreover, the coefficient $\langle \phi \gamma_+, \gamma_- \rangle$ is nonzero only if $\A(\gamma_+) \geq\A(\gamma_-)$, which is why the cobordism maps $\Phi^L(X,\lambda)$ preserve the symplectic action filtration.  

We define $\Phi^L(X,\lambda,J)$ to be the set of chain maps, which are the fruits of the labors of Hutchings and Taubes \cite{cc2} for $r \geq r_0$, where $r_0$ is chosen to be sufficiently large\footnote{The constant $r_0$ also depends on $L, \ X, \ \lambda, $ and $J$.}, so that any such chain map in fact counts (broken) $J$-holomorphic currents, and any two chain maps differ by a chain homotopy that counts $J$-holomorphic currents.  (In \cite[\S 3]{cc2}, it is explained why the energy filtered contact form perturbation of Seiberg-Witten Floer cohomology $\widehat{HM}^{*}_L(Y,\frak{s};\lambda,J,r)$ does not depend on $J$ or $r$.)

\begin{remark}
To study questions pertaining to the existence of symplectic cobordisms between transverse knots a weaker notion of symplectic cobordism will be desirable, namely that of a strong symplectic cobordism, as considered in \cite{beyond, echsft}.
\end{remark}

\subsection{Action filtered chain maps}\label{ss:action}

First, we recall from \eqref{e:scale} that if $r>0$ is a constant then there is a canonical scaling isomorphism
\begin{equation}\label{eq:scale}
ECH^L(Y,\lambda) = ECH^{rL}(Y,r\lambda)
\end{equation}
because $\lambda$ and $r \lambda$ have the same Reeb orbits up to reparametrization.  Moreover, after a unique appropriate choice of almost complex structure, the bijection on Reeb orbits gives an isomorphism at the level of chain complex \eqref{eq:scale}.  This scaling isomorphism preserves the knot filtration $\fb$ in \eqref{eq:filt}.

Second, recall from \eqref{eq:incl}, that given $L\leq L'$ there are homomorphisms induced by the inclusion of chain complexes: 
\[
\begin{split}
\iota: ECH^L(Y,\lambda, \Gamma) \longrightarrow & \ ECH(Y,\lambda, \Gamma), \\
\iota^{L,L'}: ECH^L(Y,\lambda, \Gamma) \longrightarrow & \ ECH^{L'}(Y,\lambda, \Gamma).
\end{split}
\]
(That there is independence of $J$ is shown in \cite[Thm.~1.3]{cc2}.) The homomorphisms $\iota^{L,L'}$ fit together into a direct system $(\{ECC_*^L(Y,\lambda,\Gamma)\}_{L \in \R}, \iota^{L,L'})$.  Since taking direct limits commutes with taking homology, we have
\[
ECH_*(Y,\lambda,\Gamma) = H_*\left( \lim_{L \to \infty }ECC^L_*(Y,\lambda,\Gamma;J) \right) = \lim_{L \to \infty }ECH^L_*(Y,\lambda,\Gamma) 
\]

An exact symplectic cobordism $(X,\lambda)$ from $(Y_+,\lambda_+)$ to $(Y_-,\lambda_-)$,  where $\lambda_\pm$ are nondegenerate, induces canonical maps on action filtered ECH
\[
\Phi^L(X,\lambda): ECH^L(Y_+,\lambda_+) \to ECH^L(Y_-,\lambda_-),
\]
satisfying various properties, which are proven using Seiberg-Witten theory and reviewed below.

\begin{theorem}[{\cite[Thm.~1.9, proof of Prop.~5.2(b), proof of Prop.~5.4]{cc2}, \cite[Prop.~6.2]{HuMAC}}]\label{thm:cobmaps} Let $\lambda_\pm$ be nondegenerate contact forms on closed 3-manifolds $Y_\pm$, with $(X,\lambda)$ an exact symplectic cobordism from $(Y_+,\lambda_+)$ to $(Y_-,\lambda_-)$.  Let $J$ be a cobordism compatible almost complex structure on $\overline{X}$, which restricts to generic $\lambda_\pm$-compatible almost complex structures $J_\pm$ on the ends. Then for each $L>0$, there exists a nonempty set ${\Phi}^L(X,\lambda,J)$ of maps of ungraded $\Z/2$-modules
\[
\hat{\phi}:ECH^L_*(Y_+,\lambda_+,J_+)\to ECH^L_*(Y_-,\lambda_-,J_-),
\]
induced by a nonempty set $\Theta^L(X,\lambda,J)$ of chain maps, satisfying the following properties.
\begin{itemize}

\item {\em (Holomorphic Curves)} The (noncanonical) chain map
\[
{\phi}: ECC^L_*(Y_+,\lambda_+,J_+)\to ECC^L_*(Y_-,\lambda_-,J_-),
\]
inducing $\hat{\phi} \in {\Phi}^L(X,\lambda,J)$, counts $J$-holomorphic currents.  More precisely, if $\gamma_\pm$ are the respective admissible Reeb currents for $(Y_\pm,\lambda_\pm)$ with $\A(\gamma_\pm)<L$, then:
\
\begin{enumerate}[\em (i)]
\item If there are no broken $J$-holomorphic curves in $\overline{X}$ from $\gamma_+$ to $\gamma_-$, then $\langle \hat{\phi} \gamma_+,\gamma_-\rangle = 0$.
\item If the only broken $J$-holomorphic curve  in $\overline{X}$ from $\gamma_+$ to $\gamma_-$ is a union of covers of product cylinders, then $\langle \hat{\phi} \gamma_+,\gamma_-\rangle = 1$.
\end{enumerate}

\item {\em (Inclusion)} If $L<L'$ then the following diagram commutes:
\begin{equation}\label{eqn:ECHcd}
\xymatrixcolsep{3pc}\xymatrix{
ECH^L(Y,\lambda_+,\Gamma) \ar[r]^{\Phi^L(X,\lambda)} \ar[d]_{\iota^{L,L'}} & ECH^L(Y,\lambda_-,\Gamma) \ar[d]^{\iota^{L,L'}}
\\ECH^{L'}(Y,\lambda_+,\Gamma) \ar[r]_{\Phi^{L'}(X,\lambda)} & ECH^{L'}(Y,\lambda_-,\Gamma)
}
\end{equation}

\item {\em (Homotopy Invariance)} Let $\left(X,\{ \lambda_t\}_{t\in[0,1]}\right)$ be a strong homotopy of exact symplectic cobordisms from $(Y_+,\lambda_+)$ to $(Y_-,\lambda_-)$.  Let $\{J_t\}_{t\in[0,1]}$ be a family of almost complex structures on $\overline{X}$ such that for each $t,$ $J_t$ is cobordism compatible for $\lambda_t$ and $J_t$ restricts to $J_\pm$.  Given
\[
\phi_i \in \Theta^L(X,\lambda_i,J_i)
\]
for $i=0,1$, there is a map
\[
\mathcal{K}: ECC^L(Y_+,\lambda_+,J_+) \to ECC^L(Y_-,\lambda_-,J_-)
\]
which counts $J_t$-holomorphic currents\footnote{This means that if $\langle\mathcal{K}\alpha(1),\gamma(0)\rangle \neq 0$, then for some $t$, the moduli space $\overline{\mathcal{M}^{J_t}(\alpha_+,\beta_-)}$ is nonempty.  Here $\partial_\pm$ denotes the differential on the chain complex $ECC(Y_\pm,\lambda_\pm,J_\pm)$.} such that 
\[
\partial_-\mathcal{K} + \mathcal{K}  \partial_+= \phi_0 - \phi_1
\]
\item {\em (Composition)} Let
\[
\phi_\pm \in \Theta^L(X_\pm,\lambda_\pm,J_\pm)
\]
and
\[
\phi \in \Theta^L(X_-\circ X_+,\lambda_- \circ \lambda_+, J_- \circ J_+).
\]
Then there exists a chain homotopy
\[
\mathcal{K}: ECC^L(Y_1,\lambda_1,J_1) \to ECC^L(Y_{-1},\lambda_{-1},J_{-1})
\]
such that
\[
\partial_{-1}\mathcal{K} + \mathcal{K}  \partial_{1}= \phi_- \circ \phi_+ - \phi
\]
and $\mathcal{K}$ counts $J_- \circ_R J_+$-holomorphic currents.
\item {\em (Trivial Cobordisms)} Let $\lambda_0$ be a nondegenerate contact form on $Y_0$ and suppose that 
\[
(X,\lambda)=([a,b] \times Y_0, e^s\lambda_0).
\]
Let $J_0$ be a generic $\lambda_0$-compatible almost complex structure on $\R \times Y_0$ and
\[
\phi_0: ECC^L\left(Y_0,e^b\lambda_0,J_0^{e^b}\right) \to  ECC^L\left(Y_0,e^a\lambda_0,J_0^{e^a}\right)
\]
denote the chain map induced in \eqref{eq:scale}.  Let $f: \R \to \R$ be a positive function such that $f(s)=e^a$ when $s\leq a$ and $f(s)=e^b$ when $s\geq b$.  Then
\[
\Theta^L\left([a,b] \times Y_0,e^s\lambda_0,J_0^{f}\right) = \{ \phi_0 \}
\]
\end{itemize}

\end{theorem}

Furthermore, as explained in \cite[Rem.~1.10]{cc2} the maps $\Phi^L(X,\lambda)$ respect the decomposition (\ref{e:decomp}) in the following sense: the image of $ECH_*(Y_+,\lambda_+,\Gamma_+)$ has a nonzero component in $ECH_*(Y_-,\lambda_-,\Gamma_-)$ only if $\Gamma_\pm\in H_1(Y_\pm)$ map to the same class in $H_1(X)$.

Next, we collect some additional facts about cobordism maps on the chain level in special exact cobordisms. These cobordism maps allowed us to compute the ECH of prequantization bundles in \cite[\S 7]{preech}, via successive action filtrations, and we will also employ them in our definition and computation of knot filtered ECH with respect to rational rotation numbers.

\begin{lemma}\label{lem:nicecobmap1}{\em \cite[Lem.~3.4(d), 5.6 and Def.~5.9]{cc2}} Given a real number $L$, let $\lambda_s$ and $J_s$ be smooth 1-parameter families of contact forms on $Y$ and $\lambda_s$-compatible almost complex structures such that
\begin{itemize}
\item The contact forms $\lambda_s$ are of the form $f_s\lambda_0$, where $f:[0,1]\times Y\to\R_{>0}$ satisfies $\frac{\partial f}{\partial s}<0$ everywhere.
\item All Reeb orbits of each $\lambda_s$ of length less than $L$ are nondegenerate, and there are no Reeb currents of $\lambda_s$ of action exactly $L$. (This condition is referred to in \cite{cc2} as $\lambda_s$ being ``$L$-nondegenerate.")
\item Near each Reeb orbit of length less than $L$ the pair $(\lambda_s,J_s)$ satisfies the conditions of \cite[(4.1)]{taubesechswf}. (This condition 
 is referred to in \cite{cc2} as $(\lambda_s,J_s)$ being ``$L$-flat.")
\item For Reeb currents of action less than $L$, the ECH differential $\partial$ is well-defined on admissible Reeb currents of action less than $L$ and satisfies $\partial^2=0$. (This is a condition on the genericity of $J_t$ described in \cite{obg2}, and referred to in \cite{cc2} as $J_s$ being ``$ECH^L$-generic.")
\end{itemize}
Then $([-1,0]\times Y,\lambda_{-s})$ is an exact symplectic cobordism from $(Y,\lambda_0)$ to $(Y,\lambda_1)$, and for all $\Gamma\in H_1(Y)$, the cobordism map $\Phi^L([-1,0]\times Y,\lambda_{-s})$ is induced by the isomorphism of chain complexes 
\[
ECC^L_*(Y,\lambda_0,\Gamma;J_0)\to ECC^L_*(Y,\lambda_1,\Gamma;J_1),
\]
determined by the canonical bijection on generators.

\end{lemma}

\begin{remark}\label{rem:nicecobmap1}
In \cite{cc2} and \cite{HuMAC} the notation conventions for the ``directionality" of the cobordisms in Lemma \ref{lem:nicecobmap1} disagree; we made use of the former in \cite{preech}.   To align with \cite{HuMAC}, we subsequently switch to using the exact symplectic cobordism $([0,1]\times Y,\lambda_{s})$, which gives rise to a cobordism map $\Phi^L([0,1]\times Y,\lambda_{s})$ induced by the isomorphism of chain complexes
\[
 ECC^L_*(Y,\lambda_1,\Gamma;J_1)\to ECC^L_*(Y,\lambda_0,\Gamma;J_0).
\]
determined by the canonical bijection on generators in Lemma \ref{lem:nicecobmap1}.   This latter convention is also taken in Theorem \ref{thm:cobmaps} (Trivial Cobordisms).  

Additionally, in \cite[\S 3.1, Lem.~3.6]{cc2}, it is explained that an arbitrary pair $(\lambda,J)$, where $\lambda$ is an $L$-nondegenerate contact form and $J$ is an $ECH^L$ generic $\lambda$-compatible almost complex structure on $\R \times Y$, can always be approximated by an $L$-flat pair $(\lambda_1,J_1)$, which is the endpoint of a well-behaved smooth homotopy as in \cite[Def.~3.2]{cc2}.  As a result there is a canonical isomorphism of chain complexes induced by the canonical identification of generators
\[
ECC_*^L(Y,\lambda,\Gamma; J) \overset{\simeq}{\to} ECC_*^L(Y,\lambda_1,\Gamma; J_1).  \]
Combining this with the isomorphism established in \cite[Prop.~3.1]{cc2}, we can conclude that if  $(\lambda_1,J_1)$ is an $L$-flat approximation and if $r$ is sufficiently large, then there is a canonical isomorphism of chain complexes
\[
ECC_*(Y,\lambda,\Gamma;J) \overset{\simeq}{\to} \widehat{CM}^{-*}_L(Y,\frak{s}_{\xi,\Gamma}; \lambda_1,J_1,r).
\]
(The right hand side denotes the filtered perturbed Seiberg-Witten cochain complex, which we are about to dive into.)
\end{remark}

\subsection{Direct systems via filtered perturbed Seiberg-Witten}\label{ss:direct}
Our approach to defining and computing knot filtered embedded contact homology using a a {knot admissible pair} requires the use of Seiberg-Witten Floer cohomology to define the direct system, similarly to \cite[\S 7.1]{preech}.  These  cobordism maps belong to the realm of energy filtered contact form perturbed Seiberg-Witten Floer cohomology $\widehat{HM}^*_L(Y,\frak{s};\lambda,J,r)$, which we now review from \cite{cc2}.  We do not take the time to define Seiberg-Witten Floer cohomology, which is fully explained in the book by Kronheimer and Mrowka \cite{KMbook}, the contact and symplectic form perturbations of the Seiberg-Witten equations, or the energy filtration that is analogous to the action filtration in ECH; a summary can be found in \cite[\S 7.1.1-7.1.3]{preech}, and many more details in \cite[\S 2, 4]{cc2}.  

\begin{remark}
We define $\widehat{CM}^*_L(Y,\frak{s};\lambda,J,r)$ to be the submodule of  $\widehat{CM}^*_\text{irr}$ generated by irreducible solutions $(A, \psi)$ to the contact form perturbation of the Seiberg-Witten equations \cite[(28)]{cc2} 
with energy $E(A) < 2\pi L$ \cite[Lem.~2.3]{cc2} and abstract perturbation (if necessary to obtain suitable transversality).  The energy of a reducible solution $(A,0)$ to the contact form perturbation of the Seiberg-Witten equations is a linear increasing function in $r$, so if $r$ is sufficiently large then the condition that elements of $\widehat{CM}^*_L$ be elements of $\widehat{CM}^*_\text{irr}$ is redundant: if the energy $E(A)<2\pi L$ then if $r$ is large enough, the pair $(A,0)$ cannot be a solution to the perturbed Seiberg-Witten equations.
\end{remark}

First we recall the necessary conditions for defining the homology of the submodule $\widehat{CM}^*_L$:
\begin{lemma}[{\cite[Lem.~2.3]{cc2}}]\label{lem:CMLsubcx}
Fix $Y,\lambda, J$ as above and $L\in\R$. Suppose that $\lambda$ has no Reeb current of action exactly $L$. Fix $r$ sufficiently large, and a 2-form $\mu$ so that all irreducible solutions to the perturbed Seiberg-Witten equations are cut out transversely. Then for every $\mathfrak{s}$ and for every sufficiently small generic abstract perturbation, $\widehat{CM}^*_L(Y,\mathfrak{s};\lambda,J,r)$ is a subcomplex of $\widehat{CM}^*(Y,\mathfrak{s};\lambda,J,r)$.
\end{lemma}

When the hypotheses of Lemma \ref{lem:CMLsubcx} apply, we denote the homology of $\widehat{CM}^*_L(Y,\mathfrak{s};\lambda,J,r)$ by $\widehat{HM}^*_L(Y,\lambda,\mathfrak{s})$. (If $r$ is sufficiently large then this homology is independent of $\mu$ and $r$, and it is also independent of $J$, as shown in \cite[Cor.~3.5]{cc2}.) We use the notation $\widehat{HM}^*_L(Y,\lambda,\mathfrak{s};\lambda,J,r)$ when we wish to emphasize the roles of $J$ and $r$.

We have that filtered Seiberg-Witten Floer cohomology is isomorphic to ECH:
\begin{lemma}[{\cite[Lem.~3.7]{cc2}}]\label{lem:fswech} Suppose that $\lambda$ is $L$-nondegenerate and $J$ is $ECH^L$-generic (see Lemma \ref{lem:nicecobmap1}). Then for all $\Gamma\in H_1(Y)$, there is a canonical isomorphism of relatively graded $\Z/2$-modules
\begin{equation}\label{eqn:filterediso}
\Psi^L:ECH^L_*(Y,\lambda,\Gamma;J)\overset{\simeq}{\longrightarrow}\widehat{HM}^{-*}_L(Y,\lambda,\mathfrak{s}_{\xi,\Gamma}),
\end{equation}
where $\mathfrak{s}_{\xi,\Gamma}$ is the spin-c structure $\mathfrak{s}_{\xi} + \op{PD}(\Gamma)$.
\end{lemma}

We now illustrate the ideas behind Lemma \ref{lem:nicecobmap1}, as encapsulated in \cite[\S 5.3]{cc2}, which will be used  to construct a direct system.  

\begin{proposition}\label{prop:admissible}
Fix $L$ such that $\lambda_{\varepsilon(L)}$ has no Reeb currents of action exactly $L$.  Then for
\[
\varepsilon' < \varepsilon \leq \varepsilon(L)
\]
and pairs $(\lambda_\varepsilon,J_1)$ and $(\lambda_{\varepsilon'},J_1)$  
satisfying the conditions in Lemma \ref{lem:nicecobmap1}, there is a cobordism map
\[
\varphi_{\varepsilon,\varepsilon'}^L : ECH_*^{L}(Y,\lambda_{\varepsilon},J_1)  \to ECH_*^{L}(Y,\lambda_{\varepsilon'}, J_0) 
\]
which is an isomorphism.  
\end{proposition}
\begin{proof}
Since $\ker \lambda_\varepsilon = \ker \lambda_{\varepsilon'}$, we know $\lambda_\varepsilon = e^{g_\varepsilon} \lambda_{\varepsilon'}$ for some $g_\varepsilon \in \C^\infty(Y, \R)$.  We may assume $g_\varepsilon > 0$ everywhere by the scaling isomorphism \eqref{e:scale}, which preserves the knot filtration $\fb$.  
Let $g\in \C^\infty(\R \times Y,\R)$ such that
\begin{itemize}
\itemsep-.35em
\item $g(s,y) = s$ for $s \in (-\infty, \varsigma)$ for some $\varsigma > 0$;
\item $g(s,y) = g_\varepsilon(y) + s -1$ for $s \in (1- \varsigma, \infty)$ for some $\varsigma > 0$;
\item $\partial_s g >0$.
\end{itemize}
Since $d(e^{g(s,\cdot)}\lambda_{\varepsilon'})$ is symplectic, $([0,1] \times Y, e^{g(s,\cdot)}\lambda_{\varepsilon'})$ is an exact symplectic cobordism from $(Y,\lambda_{\varepsilon})$ to $(Y,\lambda_{\varepsilon'})$.  Consider the {admissible deformation} (in the sense of \cite[Def.~3.3]{cc2})
\begin{equation}\label{eq:eta}
\rho:=\{ e^{g(s,\cdot)} \lambda_{\varepsilon'}, L, J_s, r_s) \ | \ s\in[0,1] \},
\end{equation}
where $r_s$ is sufficiently large, and $J_s$ is $ECH^L$-generic. Define $\eta_s:=e^{g(s,\cdot)}\lambda_{\ve'}.$ 

By Lemma \ref{lem:fswech} and \cite[\S 3.5]{cc2}, since $\eta_s$ has no orbit sets of action $L$, 
\begin{equation}\label{eq:echhml}
ECH^L_*(Y,\eta_s,\Gamma;J_s) \ {\simeq} \ \widehat{HM}^{-*}_L(Y,\eta_s,\mathfrak{s}_{\xi,\Gamma}).
\end{equation}
By \cite[Lem.~3.4]{cc2}, the admissible deformation $\rho$ gives an isomorphism
\begin{equation}\label{eq:hmhml}
\widehat{HM}^{-*}_L(Y,\mathfrak{s}_{\xi,\Gamma};\lambda_{\varepsilon},J_1,r_1)\overset{\simeq}{\longrightarrow}\widehat{HM}^{-*}_L(Y,\mathfrak{s}_{\xi,\Gamma};\lambda_{\varepsilon'},J_0,r_0).
\end{equation}
Composing the isomorphisms \eqref{eq:echhml} and \eqref{eq:hmhml} gives the desired map $\varphi_{\varepsilon,\varepsilon'}^L.$

\end{proof}

For a fixed $L$, when $\varepsilon > \varepsilon(L)$, we cannot typically directly compute $ECH^L_*(Y,\lambda_{\varepsilon},J)$.  For example, the chain complex $ECC^L_*(S^3,\lambda_{2,q,\varepsilon},J)$ will contain orbits which do not project to critical points of $H_{2,q}$.  Thus it is desirable to instead compute the direct limit over $L$ with respect to a sequence of contact forms $\{ \lambda_{\varepsilon(L)} \}$. This requires some additional cobordism maps from Seiberg-Witten Floer cohomology, as in \cite[\S 7.1.3]{preech}.  However, we will need to take slightly more care, so that we can define the direct system associated to a {knot admissible} pair and take its direct limit; this will allow us to define and establish invariance of knot filtered ECH with respect to a rational rotation number in \S \ref{ss:directlimit}.  

Analogous to the cobordism maps on $ECH^L_*$, there are cobordism maps on $\widehat{HM}^*_L$. The following is a modified version of \cite[Cor.~5.3(a)]{cc2}, which keeps track of the spin-c structures in our setting. Note that therefore our notation for the cobordism maps on $\widehat{HM}^*_L$ differs slightly from that of \cite{cc2}.

\begin{lemma}\label{lem:HMLcobmapdef} Let $(X,\lambda)$ be an exact symplectic cobordism from $(Y_+,\lambda_+)$ to $(Y_-,\lambda_-)$ where $\lambda_\pm$ is $L$-nondegenerate. Let $\mathfrak{s}$ be a spin-c structure on $X$ and let $\mathfrak{s}_\pm$ denote its restrictions to $Y_\pm$, respectively. Let $J_\pm$ be $\lambda_\pm$-compatible almost complex structures. Suppose $r$ is sufficiently large. Fix 2-forms $\mu_\pm$ and small abstract perturbations sufficient to define the chain complexes $\widehat{CM}^*(Y_\pm,\mathfrak{s}_\pm;\lambda_\pm,J_\pm,r)$. Then there is a well-defined map
\begin{equation}\label{eqn:HMLcobmap}
\widehat{HM}^*_L(X,\lambda,\mathfrak{s}):\widehat{HM}^*_L(Y_+,\mathfrak{s}_+;\lambda_+,J_+,r)\to\widehat{HM}^*_L(Y_-,\mathfrak{s}_-;\lambda_-,J_-,r),
\end{equation}
depending only on $X,\mathfrak{s},\lambda,L,r,J_\pm,\mu_\pm$, and the perturbations, such that if $L<L'$ and if $\lambda_\pm$ are also $L'$-nondegenerate, then the diagram
\begin{equation}\label{eqn:HMLcd}
\xymatrixcolsep{5pc}\xymatrix{
\widehat{HM}^*_{L}(Y_+,\mathfrak{s}_+;\lambda_+,J_+,r) \ar[r]^{\widehat{HM}^*_{L}(X,\lambda,\mathfrak{s})} \ar[d] & \widehat{HM}^*_{L}(Y_-,\mathfrak{s}_-;\lambda_-,J_-,r) \ar[d]
\\\widehat{HM}^*_{L'}(Y_+,\mathfrak{s}_+;\lambda_+,J_+,r) \ar[r]_{\widehat{HM}^*_{L'}(X,\lambda,\mathfrak{s})} & \widehat{HM}^*_{L'}(Y_-,\mathfrak{s}_-;\lambda_-,J_-,r)
}
\end{equation}
commutes, where the vertical arrows are induced by inclusions $i^{L,L'}$ of chain complexes.
\end{lemma}

After passing through the isomorphism with filtered perturbed Seiberg-Witten Floer cohomology, Lemma \ref{lem:fswech}, we obtain the following cobordism map from Lemma \ref{lem:HMLcobmapdef} associated to the product exact symplectic cobordism $(X:= [0,1] \times Y, \lambda:=e^{g(s,\cdot)},\lambda_{\varepsilon'})$
\begin{equation}\label{eq:maps}
\Phi^L(X,\lambda,J): ECH_*^{L}(Y,\lambda_{\varepsilon},J_1)  \to ECH_*^{L}(Y,\lambda_{\varepsilon'},J_0). 
\end{equation}

Assuming that $L<L'$ and $\ve(L) > \ve(L')$, after combining the cobordism map induced by admissible deformation in Proposition \ref{prop:admissible} with the cobordism map induced by inclusion, we obtain the following commutative diagram, as in the proof of \cite[Lem.~3.7]{cc2}:
\begin{equation}\label{eq:square}
\xymatrixcolsep{3pc}\xymatrix{
ECH^{L}(Y,\lambda_{\ve(L)}) \ar[r]^{\varphi^L_{\ve(L),\ve(L')}} \ar[d]_{\iota^{L,L'}} & ECH^L(Y,\lambda_{\ve(L')}) \ar[d]^{\iota^{L,L'}}
\\ECH^{L'}(Y,\lambda_{\ve(L)}) \ar[r]_{\varphi^{L'}_{\ve(L),\ve(L')}} & ECH^{L'}(Y,\lambda_{\ve(L')})
}
\end{equation}

Here the $\iota^{L,L'}$ are the inclusion induced cobordism maps as in Theorem \ref{thm:cobmaps} (Inclusion).  We have also (abusively) suppressed the almost complex structures, though this is acceptable by \cite[\S 5.3]{cc2}.  Moreover, if $\ve=\ve(L)$ and $\ve'=\ve'(L)$ then by Lemma \cite[Lem.~5.6]{cc2}, $ \varphi^L_{\ve(L),\ve(L')}$ and $ \varphi^{L'}_{\ve(L),\ve(L')}$ can be identified with the respective exact cobordism maps $\Phi^L(X,\lambda,J)$ and $\Phi^{L'}(X,\lambda,J)$ from \eqref{eq:maps}.

Next we verify that \eqref{eq:square} produces a direct system.  We need to check that we have a well-defined composition for the cobordism maps, defined via either path in \eqref{eq:square},
\begin{equation}\label{eq:comp}
\Phi^{L,L'}(\ve(L),\ve(L')): ECH^{L}(Y,\lambda_{\ve(L)}) \to ECH^{L'}(Y,\lambda_{\ve(L')}) .
\end{equation}
To do so, for $L<L'<L''$, we define $\ve:=\ve(L) > \ve':=\ve(L') > \ve'':=\ve(L'')$.  Then the composition \eqref{eq:comp} is given by the commutative diagram:

\begin{equation}\label{eq:giant}
\xymatrixcolsep{3pc}\xymatrix{
ECH^{L}(Y,\lambda_{\ve}) \ar[r]_{\varphi^L_{\ve,\ve'}} \ar[d]_{\textcolor{magenta}{\iota^{L,L'}}} & ECH^L(Y,\lambda_{\ve'}) \ar[r]_{\varphi^L_{\ve',\ve''}} \ar[d]_{\iota^{L,L'}} & ECH^L(Y,\lambda_{\ve''}) \ar[d]_{\iota^{L,L'}} 
\\ECH^{L'}(Y,\lambda_{\ve}) \ar[r]_{\varphi^{L'}_{\ve,\ve'}} \ar[d]_{\textcolor{magenta}{{\iota^{L',L''}}}}  & ECH^{L'}(Y,\lambda_{\ve'}) \ar[r]_{\varphi^{L'}_{\ve',\ve''}} \ar[d]_{\textcolor{blue}{\iota^{L',L''}}}  & ECH^{L'}(Y,\lambda_{\ve''})  \ar[d]_{\iota^{L',L''}}
\\ECH^{L''}(Y,\lambda_{\ve}) \ar[r]_{\textcolor{magenta}{\varphi^{L''}_{\ve,\ve'}}} & ECH^{L''}(Y,\lambda_{\ve'}) \ar[r]_{\textcolor{violet}{\varphi^{L''}_{\ve',\ve''}}} & ECH^{L''}(Y,\lambda_{\ve''}) 
}
\end{equation}
To complete the direct limit of the filtered ECH complexes with respect to the above maps, some additional algebraic manipulations are required, akin to those found at the very end of \cite[\S 7]{luya}, which we complete in the next subsection.   We also review there why the chain maps do not increase the knot filtration.  (The colors provide navigation of the diagram as needed in the proof of Lemma \ref{lem:trick}.)

\begin{remark}
We do not want to directly pass to Seiberg-Witten theory to complete the direct limit, as we did in \cite[\S 7.1.4]{preech}.  This is because we need to verify that the action filtered cobordism maps count broken $J$-holomorphic currents, and hence preserve the knot filtration, which is defined within the realm of embedded contact homology.  
\end{remark}

Before completing this argument, for the purposes of the Morse-Bott computations that we made use of in the proof of Theorem \ref{thm:kECH}, we review why the maps induced by the composition of exact symplectic cobordisms $([\ve',\ve] \times S^3, (1+s\fp^*H_{2,q})\lambda_{2,q})$ compose properly.  This relies on a version of \cite[Prop. 5.4]{cc2} explaining the composition law for $\widehat{HM}^*_L$, which we previously explained in \cite[\S 7.1]{preech}

 Assume $\varepsilon''<\varepsilon'<\varepsilon$ and recall Notation \ref{n:cobord}. We consider the exact symplectic cobordism $(X_{[\varepsilon'',\varepsilon]},\lambda_{2,q,[\ve',\ve]})$, which is the composition of 
 \[
(X_{[\varepsilon'',\varepsilon']\circ[\varepsilon',\varepsilon]}, \lambda_{2,q, [\varepsilon'',\varepsilon']\circ[\varepsilon',\varepsilon]}):=(X_{[\varepsilon'',\varepsilon']}, \lambda_{2,q, [\varepsilon'',\varepsilon']}) \circ  (X_{[\varepsilon',\varepsilon]}, \lambda_{2,q, [\varepsilon',\varepsilon]}),
\]
 in the sense of \cite[\S1.5]{cc2}, where $\lambda_{2,p,\varepsilon}, \lambda_{2,p,\varepsilon'}$, and $\lambda_{2,p,\varepsilon''}$ are $L(\varepsilon)$-nondegenerate. We also assume $J, J'$, and $J''$ are $\lambda_\varepsilon$-, $\lambda_{\varepsilon'}$-, and $\lambda_{\varepsilon''}$-compatible almost complex structures, respectively. Further, we choose a spin-c structure $\mathfrak{s}''$ on $[\varepsilon'',\varepsilon]\times Y$ which restricts to spin-c structures $\mathfrak{s}'$ and $\mathfrak{s}$ on $[\varepsilon'',\varepsilon']\times Y$ and $[\varepsilon',\varepsilon]\times Y$, respectively, where $\mathfrak{s}'$ restricts to $\mathfrak{s}_2$ on $\{\varepsilon''\}\times Y$, $\mathfrak{s}$ restricts to $\mathfrak{s}_0$ on $\{\varepsilon\}\times Y$, and both $\mathfrak{s}'$ and $\mathfrak{s}$ restrict to $\mathfrak{s}_1$ on $\{\varepsilon'\}\times Y$. Finally we choose abstract perturbations and $r$ large enough to define the chain complexes $\widehat{CM}^*_L$.
\begin{lemma}\label{lem:HMLcomp} The maps of Lemma \ref{lem:HMLcobmapdef} for the above data satisfy
\[
\widehat{HM}^*_L(X_{[\varepsilon'',\varepsilon]}, \lambda_{2,q, [\varepsilon'',\varepsilon]},\mathfrak{s}'')=\widehat{HM}^*_L(X_{[\varepsilon'',\varepsilon']}, \lambda_{2,q, [\varepsilon'',\varepsilon']},\mathfrak{s}')\circ\widehat{HM}^*_L(X_{[\varepsilon',\varepsilon]}, \lambda_{2,q, [\varepsilon',\varepsilon]},\mathfrak{s}).
\]
\end{lemma}
While \cite[Prop. 5.4]{cc2} does not discuss the spin-c structures, it is proved with a neck-stretching argument for holomorphic curves whose ends must be homologous, thus it will preserve spin-c structures in the case considered in Lemma \ref{lem:HMLcomp}, see \cite[Rmk. 1.10]{cc2}.

In case it is instructive to better understand the proof of Theorem \ref{thm:kECH}, we summarize our ``cruder" Morse-Bott limit argument from \cite[\S 7]{preech}, wherein we passed to the isomorphism with Seiberg-Witten at an earlier stage.  We have
\[
ECH^{L(\varepsilon)}_*(Y,\lambda_\varepsilon,\Gamma)=ECH^L_*(Y,\lambda_{\varepsilon(L)},\Gamma)
\]
and therefore
\[
\lim_{L\to\infty}ECH^L_*(Y,\lambda_{\varepsilon(L)},\Gamma)=\lim_{\varepsilon\to0}ECH^{L(\varepsilon)}_*(Y,\lambda_\varepsilon,\Gamma).
\]
We then invoked the following sequence of isomorphisms by way of the contact form perturbation of the Seiberg-Witten equations:
\begin{align}
\lim_{\varepsilon\to0}ECH_*^{L(\varepsilon)}(Y,\lambda_\varepsilon,\Gamma)&\simeq\lim_{\varepsilon\to0}\widehat{HM}^{-*}_{L(\varepsilon)}(Y,\lambda_\varepsilon,\mathfrak{s}_{\xi,\Gamma})\label{eqn:ECHeHMe}
\\&\simeq\lim_{\varepsilon\to0}\lim_{L\to\infty}\widehat{HM}^{-*}_L(Y,\lambda_\varepsilon,\mathfrak{s}_{\xi,\Gamma})\label{eqn:factoringdleL}
\\&\simeq\lim_{\varepsilon\to0}\widehat{HM}^{-*}(Y,\mathfrak{s}_{\xi,\Gamma})\label{eqn:sendLtoinfty}
\\&\simeq ECH_*(Y,\xi,\Gamma).\label{eqn:HMtoECH}
\end{align}
A few remarks are in order regarding the above chain of isomorphisms. 

 The direct limit on the right hand side (\ref{eqn:ECHeHMe}) is defined using composition of the ``trivial" symplectic cobordisms and the commutative diagram of inclusion maps modified from \cite[Cor.~5.3(a)]{cc2} to keep track of spin-c structures; cf.~the recap in Lemmas \ref{lem:HMLcomp} and \ref{lem:HMLcobmapdef} .  Thus the isomorphism between action filtered Seiberg-Witten and action filtered ECH \cite[Lem.~3.7]{cc2} establishes (\ref{eqn:ECHeHMe}).

The groups $\widehat{HM}^{-*}_L(Y,\lambda_\varepsilon,\mathfrak{s}_{\xi,\Gamma})$ in (\ref{eqn:factoringdleL}) are only defined for $L$ and $\varepsilon$ such that $\lambda_\varepsilon$ has no Reeb currents of action exactly $L$.  (For any given $\varepsilon$, we still obtain a full measure set of $L$.)  A calculation carried out in \cite[\S 7.1.4]{preech} yields  (\ref{eqn:factoringdleL}).

That \eqref{eqn:sendLtoinfty} holds follows from the last equation of \cite[\S 3.5]{cc2}, which follows from \cite[Thm.~4.5]{taubesechswf}.  It bears mention that although the equation in former citation is only required to hold for nondegenerate $\lambda$, it is in fact true for all $\lambda$.

Finally, we obtain \eqref{eqn:HMtoECH} because the groups $\widehat{HM}^{-*}(Y,\mathfrak{s}_{\xi,\Gamma})$ are all equal and independent of $\varepsilon$, together with Taubes' isomorphism \cite{taubesechswf}-\cite{taubesechswf5}: $ECH_*(Y,\lambda,\Gamma,J) \simeq \widehat{HM}^{-*}(Y, \mathfrak{s}_{\xi,\Gamma})$.

\subsection{The knot filtration and doubly filtered direct limits}\label{ss:directlimit}
We now have all the ingredients to prove Theorem \ref{thm:introok}.
Before completing the algebraic arguments which allow us to take direct limits, we review why the knot filtration is respected by the chain maps, as previously explained in \cite[\S 7]{HuMAC}.  Let us collect the usual suspects.

Fix a closed contact 3-manifold $(Y,\xi)$ with $H_1(Y;\Z)=0$\footnote{This condition can be relaxed to allow $H_1(Y;\Z)$ to be torsion, cf. \cite[Thms.~5.2 \& 5.3]{weiler} to see how to obtain a well-defined rotation number in this case. } and let $\lambda_0$ and $\lambda_1$ be two contact forms for $\xi$, which are nondegenerate and both admit the same transverse knot $b$ as an elliptic embedded Reeb orbit so that $\rot_1(b) \geq \rot_0(b)$. 

\begin{remark}
 It is to be understood from context that the knot filtration with respect to $b$ is computed with respect to $\lambda_i$, $\mathcal{F}_b(b^m\alpha) := m\rot_i(b) + \ell(\alpha,b),$ where $\alpha$ is a Reeb current not containing $b$ associated to $\lambda_i$.  
\end{remark}

 Let $([0,1]_s\times Y, e^{g(s,\cdot)}\lambda_0)$ be an exact symplectic cobordism from $(Y,\lambda_1)$ to $(Y,\lambda_0)$, as defined in proof of Proposition \ref{prop:admissible}, and identify its completion with $\R \times Y$ in the obvious way.  Let $J$ be a cobordism compatible almost complex structure on $\R \times Y$, which agrees with $J_1$ on $[1,\infty)\times Y$ and with $J_0$ on $(-\infty,0]\times Y$.  We can choose $J$ so that $\R \times b$ is a $J$-holomorphic curve because $[1,\infty) \times b$ is a $J_1$-holomorphic submanifold, $(-\infty,0] \times b$ is a $J_0$-holomorphic submanifold, and $[0,1]\times b$ is a symplectic submanifold of $\R\times Y$.  (The space of $J$ satisfying these conditions is contractible.)

\begin{proposition}\label{prop:cobknot}
Given $L>0$, let
\[
{\phi}:ECC_*^L(Y,\lambda_1,J_1) \to ECC_*^L(Y,\lambda_0,J_0) 
\]
be a chain map in the set $\Theta^L([0,1] \times Y, e^{g(s,\cdot)}\lambda)$, as provided by Theorem \ref{thm:cobmaps}.  If $\lambda_1$ and $\lambda_2$ both admit $b$ as an embedded elliptic Reeb orbit and $\rot_1(b) \geq \rot_0(b)$ then $\phi$ preserves the knot filtration $\fb$, meaning that if $\langle \phi \alpha_+, \beta_-\rangle \neq 0$, then $\fb( \alpha_+) \geq \fb(\beta_-)$.
\end{proposition}
\begin{proof}
By the Theorem \ref{thm:cobmaps} (Holomorphic Curves), if $\langle \phi \alpha_+, \beta_-\rangle \neq 0$, there exists a broken  holomorphic current
\[
\cur=(\cur_{N_-},...,\cur_0,...,\cur_{N_+}) \in \overline{\M^J(\alpha_+,\beta_-)}.
\]
The proof of \cite[Lem.~5.1]{HuMAC} shows that if there exists a $J_i$-holomorphic current $\cur^i \in \M^{J_i}(b^{m_+} \gamma_+, b^{m_-} \gamma_-)$ then $\fb(b^{m_+} \gamma_+) \geq \fb( b^{m_-} \gamma_-)$; we now summarize how it goes.  The argument relies on intersection positivity of the $J_i$-holomorphic trivial cylinder $\R \times b$ with  each irreducible somewhere injective component $C$ of $\cur$ which does not agree with $\R \times b$.  By \cite[Cor.~2.5, 2.6]{s1}, if $s_0>0$ is sufficiently large, then $C$ is transverse to $\{ \pm s_0\} \times Y$ and $C \cap ((-\infty,s_0] \times Y)$ and  $C \cap ([s_0,\infty) \times Y)$ do not intersect $\R \times b$.  Let $s_0$ be sufficiently large and denote $\eta_\pm = C \cap( \{ \pm s_0\} \times Y)$ to be these intersections.  Then 
\[
\ell(\eta_+,b)  - \ell(\eta_-,b) = \#( C \cap (\R \times b)) \geq 0.
\]
Moreover, the link $\eta_\pm$ consists of a link approximating the Reeb currents $\gamma_\pm$ and a link $\zeta_\pm$ in a neighborhood of $b$.  We have
\[
\ell(\eta_\pm,b) = \ell(\gamma_\pm, b) + \ell(\zeta_\pm,b).
\]
By the bounds on the winding number of the associated asymptotic eigenfunction of $L_{b^{m^k_\pm}}$ in terms of the Conley-Zehnder index of $b^{m^k_\pm}$, going back to \cite[\S 3]{hwz2} and as reviewed in \cite[Lem.~3.2, 3.4]{dc}, it follows that $\ell(\zeta_+,b) \leq m_+ \rot(b)$ and $\ell(\zeta_-,b) \geq m_- \rot(b)$, with respective equality holding if and only if $m_\pm=0$.  Thus it follows that the current $\cur^i$ respects the knot filtration.

Thus if $k \neq 0$, it immediately follows from this argument that $\cur_k$ preserves the filtration.  If $k=0,$ then since $\R \times b$ is $J$-holomorphic, the same intersection positivity argument and the fact that $\rot_1(b) \geq \rot_0(b)$ implies that $\cur_0$ also preserves the filtration $\fb$, and therefore $\cur$ does as well.
\end{proof}

In light of Proposition \ref{prop:cobknot}, we can restrict the action filtered subcomplex to the subcomplex where $\fb \leq K$, and after passing to homology, we obtain a map
\begin{equation}\label{eq:kcob}
\phi_*:ECH_*^{\af}(Y_1,\lambda_1,\rot_1(b),J_1) \to ECH_*^{\af}(Y_0,\lambda_0,\rot_0(b),J_0). 
\end{equation}
By Theorem \ref{thm:cobmaps} (Homotopy Invariance) and intersection positivity as in Proposition \ref{prop:cobknot}, we know that the map \eqref{eq:kcob} does not depend on the choice of $g$, $J$, or $\phi \in \Theta([0,1] \times Y, e^{g(s, \cdot)}\lambda_{\ve'},J)$.

Thus the noncanonical chain maps $\hat{\varphi}$ which induced the cobordism maps in \eqref{eq:maps}, \eqref{eq:square}, \eqref{eq:comp},  and \eqref{eq:giant} all preserve the knot filtration {assuming the knot admissibility condition, e.g. $\rot_1(b) \geq \rot_0(b)$.  Thus we could in fact have restricted to the subcomplexes where $\fb \leq K$ prior to passing to homology.  Assuming that $\rot_\ve(b) \geq \rot_{\ve'}(b),$ this produces the the commutative diagram 
\begin{equation}\label{eq:fsquare}
\xymatrixcolsep{3pc}\xymatrix{
ECH^{\af}(Y,\lambda_{\ve},\rot_\ve(b)) \ar[r]^{\varphi^L_{\ve,\ve'}} \ar[d]_{\iota^{L,L'}} & ECH^{\af}(Y,\lambda_{\ve'}, \rot_{\ve'}(b)) \ar[d]^{\iota^{L,L'}}
\\ECH^{\afp}(Y,\lambda_{\ve},\rot_\ve(b)) \ar[r]_{\varphi^{L'}_{\ve,\ve'}} & ECH^{\afp}(Y,\lambda_{\ve'}, \rot_{\ve'}(b))
}
\end{equation}
 which analogously produces a direct system. 

Next, we need to understand the direct limit as $L\to \infty$, so that we can establish the definition of knot filtered ECH with respect to a knot admissible pair, as it appeared in Theorem \ref{thm:introok}.  We need the following lemma to allow us to take a doubly filtered direct limit.  

\begin{lemma}\label{lem:trick} Given a knot admissible pair $\{(\lambda_\ve,J_\ve)\}$, as defined in Definition \ref{def:knotadmissible}, the map
\begin{equation}\label{eq:trick}
\Psi:  
\lim_{\varepsilon \to 0}ECH^{\afe}(Y, \lambda_{\varepsilon}, b, \rot_{\varepsilon}(b)) \to 
\lim_{\varepsilon \to 0} \lim_{L \to \infty}ECH^{\af}(Y, \lambda_{\varepsilon}, b, \rot_{\varepsilon}(b)),
\end{equation}
which sends the the equivalence class of an element $\sigma_\ve \in ECH^{\afe}(Y, \lambda_{\varepsilon}, b, \rot_{\varepsilon}(b))$ under $\lim_{\ve \to 0}$ to the equivalence class of $\sigma_\ve$ under   $\lim_{\varepsilon \to 0} \lim_{L \to \infty}$, is well-defined and a bijection. 
\end{lemma}

\begin{remark}
In light of Lemma \ref{lem:trick} and the associated considerations that came beforehand, we have that
\begin{equation}\label{eq:kech}
\begin{split}
ECH_*^{\fb \leq K} (Y, \lambda, b, \rot(b)) & = \lim_{\varepsilon \to 0} \lim_{L \to \infty}ECH^{\af}(Y, \lambda_{\varepsilon}, b, \rot_{\varepsilon}(b)) ,
\\
&= \lim_{\varepsilon \to 0}ECH^{\afe}(Y, \lambda_{\varepsilon}, b, \rot_{\varepsilon}(b)) ,\\
& = \lim_{L \to \infty}ECH^{\af}(Y, \lambda_{\varepsilon(L)}, b, \rot_{\varepsilon(L)}(b)).
\end{split}
\end{equation}
A highly desirable consequence of Lemma \ref{lem:trick} is that it allows us to move from a region where $L$ is too large compared to $\varepsilon$ into a `good region' where $\varepsilon < \varepsilon(L)$.  We will establish invariance after the proof of Lemma \ref{lem:trick} so as to complete the proof of Theorem \ref{thm:introok}.  A similar argument appears in the context of the connected sum formula for ECH in \cite[\S 7]{luya}.

\end{remark}
 
\begin{proof}
The proof relies on choosing appropriate paths through the knot filtered analogue of the commutative diagram \eqref{eq:giant}.  
A schematic illustration of the proof is given in Figure \ref{fig:trick}, where we have plotted $(\ve, L)$ coordinates on a plane and the shaded region corresponds to the `good region' where  $\varepsilon < \varepsilon(L)$ .  We can always move into this `good region' by moving vertically or horizontally via chain maps induced by cobordisms in the knot filtered analogue of \eqref{eq:giant}.  We now provide the details.

\begin{figure}[h]
\centering
\begin{minipage}{0.49\textwidth}
 \begin{overpic}[width=.95\textwidth, unit=1.75mm]{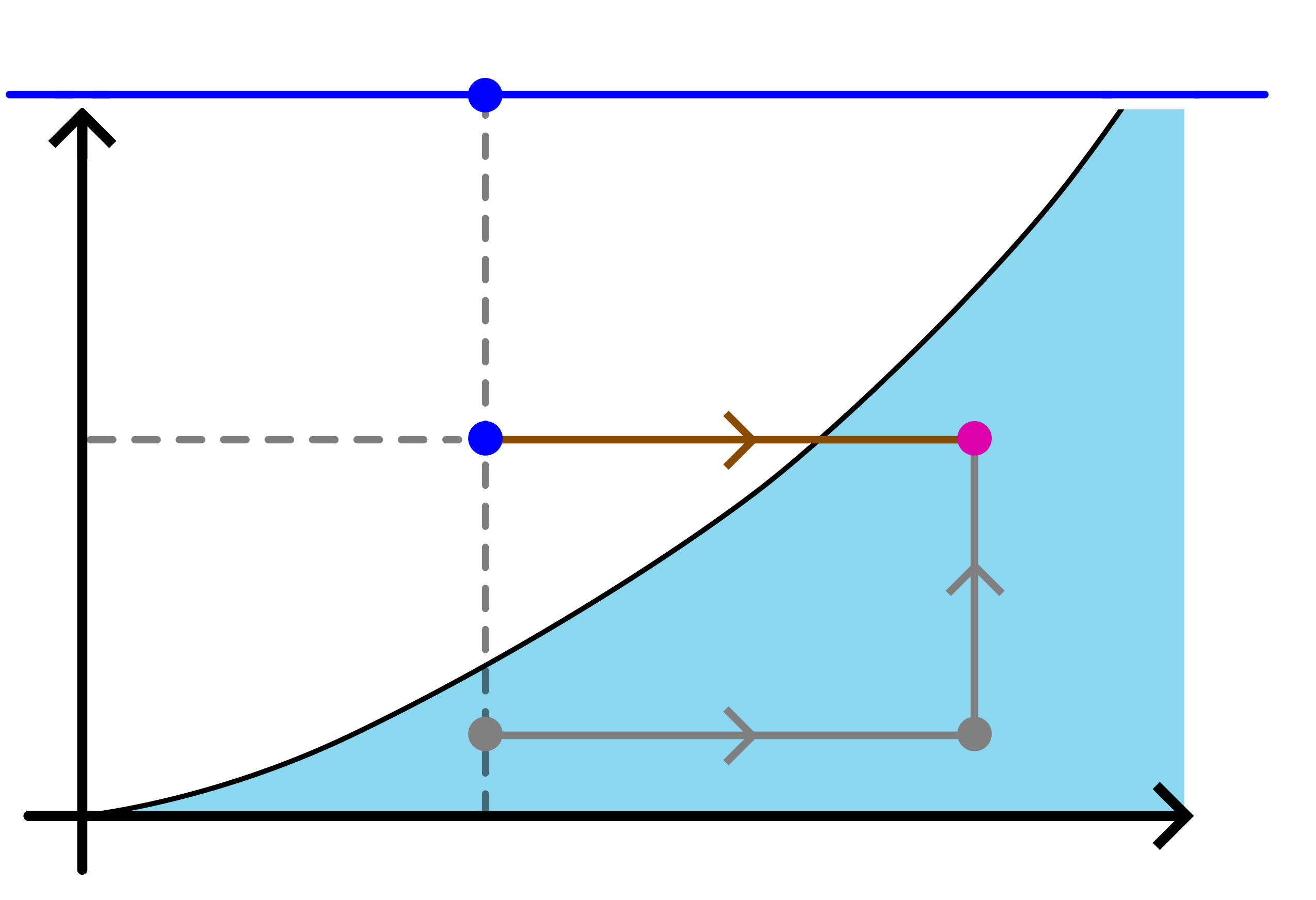}
\put(41,28){\tb{``$L=\infty$''}} 

\put(-.5,14.5){\textcolor{gray}{$L_0$}} 
\put(15,.5){\textcolor{gray}{$\ve_0$}}

\put(5,6){$\varepsilon(L)$}

\put(11,28.5){\tb{$\Psi(\sigma)=0$}}
\put(11,17){\tb{$\widetilde{\Psi(\sigma)} =0$}}
\put(30.75,17){\textcolor{magenta}{$\Phi^{L_0}\widetilde{\Psi(\sigma)}=0$}}

\end{overpic}
\caption*{(Injectivity) Let $\Psi(\sigma)=0$. Then there exists $L_0$ such that $\widetilde{\Psi(\sigma)} =0$ for some $\varepsilon_0$.  We can map $\widetilde{\Psi(\sigma)}$ to zero under $\Phi^{L_0}$ from \eqref{eq:maps}, thus $\sigma \sim 0$. }
 \end{minipage}\hfill
 \begin{minipage}{0.49\textwidth}
 \begin{overpic}[width=.95\textwidth, unit=1.75mm]{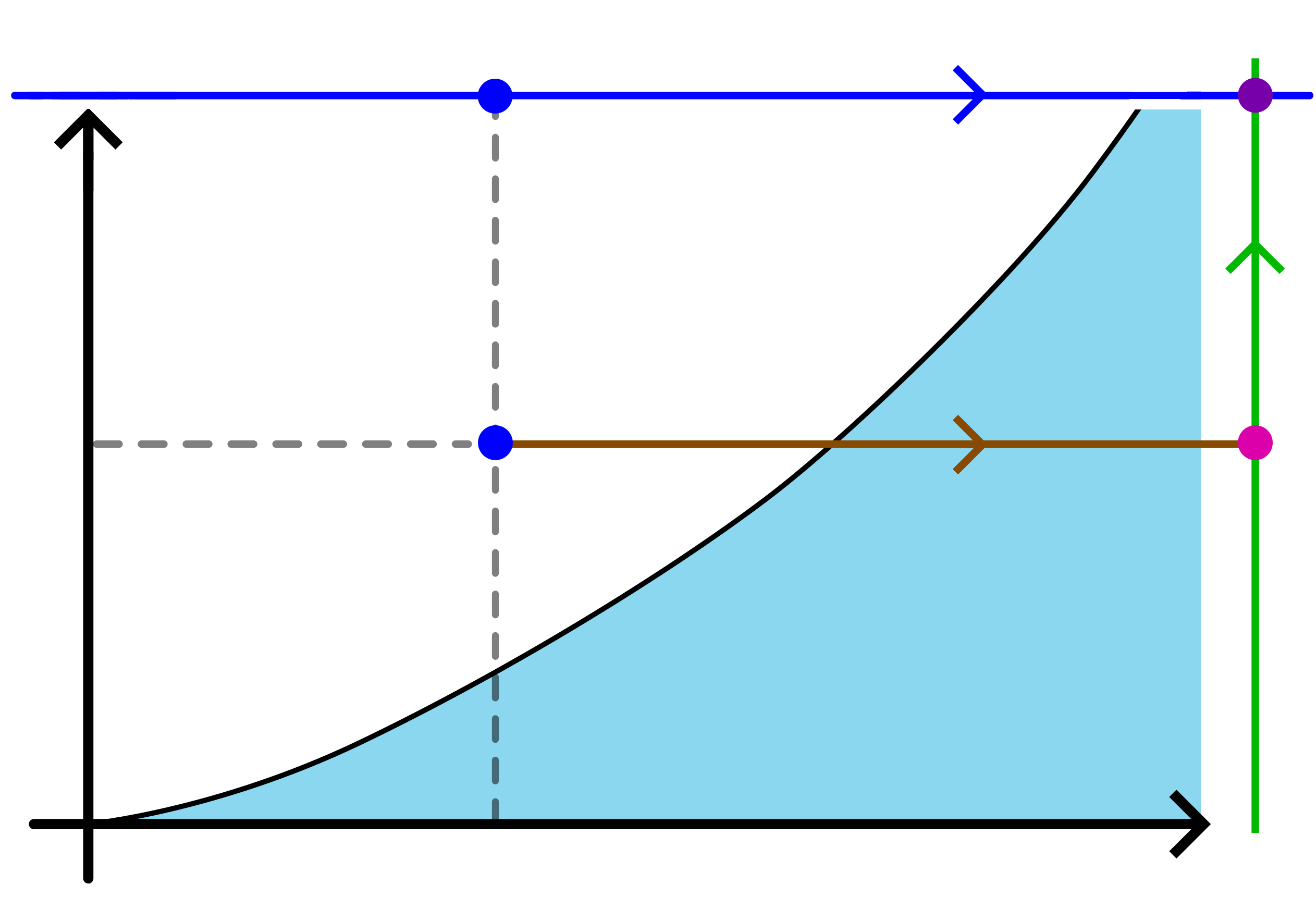}
\put(-.5,14.5){\textcolor{gray}{$L_0$}} 
\put(15,.5){\textcolor{gray}{$\ve_0$}}

\put(5,6){$\varepsilon(L)$}

\put(37,0){\textcolor{Green}{``$\ve = 0$''}} 

\put(14,28.5){\tb{$\tau=[\sigma_0]$}}
\put(17,17){\tb{$\sigma_0$}}

\put(37,17){\textcolor{magenta}{$\sigma'=\op{Im}(\sigma_0)$}}

\put(33,28.5){\textcolor{violet}{$[\tau]=\Psi([\sigma'])$}}

\end{overpic}
\caption*{(Surjectivity) Let $\sigma_0$ be some representative of an element $\tau$ for some $\ve_0$.  By taking the limit as $\ve \to 0$, the image of $\sigma_0$ is $\sigma'$.  Then $[\tau] = \Psi([\sigma'])$.   }
\end{minipage}
\caption{The `good region' corresponds to $\ve \leq \ve(L)$, which is shaded. When $\ve_0 > \ve(L_0)$ we illustrate how appropriate choices of vertical and horizontal paths in the knot filtered analogue of \eqref{eq:giant} allow us to move into this `good region.'  We reach the ``$L=\infty$'' and ``$\ve =0$" lines by taking direct limits.}
\label{fig:trick}
\end{figure}

First, we establish that $\Psi$ is well-defined.  Consider 
\[
\sigma_\ve,\ \sigma_{\ve'} \in \lim_{\ve \to 0}ECH^{\afe}(Y, \lambda_{\varepsilon}, b, \rot_{\varepsilon}(b)),
\]
 where $\ve>\ve'$ and $\sigma_\ve \sim \sigma_{\ve'}$.  This means there is a common element 
 \[
 \sigma_{\ve''} \in  ECH^{\afepp}(Y, \lambda_{\varepsilon}, b, \rot_{\varepsilon}(b)),
\]
that both $\sigma_\ve$ and $\sigma_{\ve'}$ are mapped to under the direct limit.  By composing the following maps we can show that $\Psi(\sigma_\ve) \sim \Psi(\sigma_{\ve'}),$ by showing that they both map to $\Psi(\sigma_{\varepsilon''})$.  To obtain $\Psi(\sigma_\ve)$, we compose the maps arising from the four edges (\textcolor{magenta}{three in magenta} and  \textcolor{violet}{one in purple}) obtained by first following the left column, then going along the bottom row of the knot filtered analogue of the commutative diagram \eqref{eq:giant},
\[
ECH^{\afe}(Y, \lambda_{\varepsilon}, b, \rot_{\varepsilon}(b)) {\to} ECH^{\afepp}(Y, \lambda_{\varepsilon}, b, \rot(b)_{\varepsilon}) \to ECH^{\afepp}(Y, \lambda_{\varepsilon''}, b, \rot_{\varepsilon''}(b)).
\]
To obtain $\Psi(\sigma_{\ve'})$, we compose the maps arising from the 
two edges following the bottom half of the middle column (\textcolor{blue}{in blue}) and bottom row  (\textcolor{violet}{in purple}),
\[
ECH^{\afep}(Y, \lambda_{\varepsilon'}, b, \rot_{\varepsilon'}(b)) \to ECH^{\afepp}(Y, \lambda_{\varepsilon'}, b, \rot_{\varepsilon'}(b)) \to ECH^{\afepp}(Y, \lambda_{\varepsilon''}, b, \rot_{\varepsilon''}(b)).
\]

Next, we prove injectivity of $\Psi$. The schematic illustration in Figure \ref{fig:trick} (Injectivity) may be helpful to understand $\ve$ and $L$ ranges as we consider the doubly filtered complexes. Suppose that $\Psi(\sigma)=0$.  Then there exists $L_0$ such that a representative 
\[
\widetilde{\Psi(\sigma)} \in ECH^{\afo}(Y,\lambda_{\varepsilon_0}, b, \rot_{\varepsilon_0}(b))
\]
 is zero for some $\varepsilon_0$, where $[\widetilde{\Psi(\sigma)}] = {\Psi(\sigma)}$.  If $\ve_0 \leq \ve(L_0)$ then we are done, because by Proposition \ref{prop:admissible} for the `good region' $\ve \leq \ve(L),$
\[
ECH^{\af}(Y, \lambda_{\varepsilon}, b, \rot_{\varepsilon}(b)) = ECH^{\af}(Y, \lambda_{\varepsilon(L)}, b, \rot_{\varepsilon(L)}(b)),
\]
thus $\Psi$ is a bijection. Suppose that $\ve_0 > \ve(L_0)$.  Then $\widetilde{\Psi(\sigma)} $ is mapped to zero in \\ $ECH^{\afo}(Y,\lambda_{\varepsilon_0}, b, \rot_{\varepsilon_0}(b))$ under the map in \eqref{eq:maps}.  Therefore $\sigma \sim 0$.  

Finally, we show that $\Psi$ is surjective. The schematic illustration in Figure \ref{fig:trick} (Surjectivity) may be helpful to understand $\ve$ and $L$ ranges as we consider the doubly filtered complexes.  Let
\[
\tau \in \lim_{L\to \infty} ECH^{\af}(Y,\lambda_{\ve_0},b,\rot_{\ve_0}(b)) 
\]
for some $\ve_0$.  Then there exists $L_0$ so that the element
\[
\sigma_0 \in ECH^{\afo}(Y,\lambda_{\varepsilon_0}, b, \rot_{\varepsilon_0}(b))
\]
 is a representative such that $[\sigma_0] = \tau$.  Similarly to the proof of injectivity, if $\ve_0 \leq \ve(L_0)$, then we are done.  Suppose that $\ve_0 > \ve(L_0)$.  Let
\[
\sigma' \in ECH^{\afo}(Y,\lambda_{\varepsilon(L)}, b, \rot_{\varepsilon(L)}(b))
\]
be the image of $\sigma_0$ when taking the limit as $\ve \to 0$ defined by the exact cobordism map as in  \eqref{eq:maps}.  Then $[\tau] = \Psi([\sigma'])$.
\end{proof}
 
It remains to establish invariance to complete the proof of Theorem \ref{thm:introok}.  

\begin{proposition}\label{prop:che}
Let $\{(\lambda^+_{\ve^+},J^+_{\ve^+}) \}$ be a knot admissible pair for $(Y,\lambda^+, b, \rot(b))$ and let $\{(\lambda^-_{\ve^-},J^-_{\ve^-}) \}$ be a knot admissible pair for $(Y,\lambda^-, b, \rot(b))$ where $\ker \lambda^+ = \ker \lambda^-. $ Then there exists a chain map
\begin{equation}\label{eq:propcm}
\Phi^+_-:ECC^{\fb \leq K}_*(Y,\lambda^+,b, \rot(b), J^+) \to ECC^{\fb \leq K}_*(Y,\lambda^-,b, \rot(b),J^-),
\end{equation}
which is a chain homotopy equivalence.
\end{proposition}
\begin{proof}
Since the rotation angles of the knot admissible families $\{\lambda^+_{\ve^+} \}$ and $\{\lambda^-_{\ve^-} \}$ both monotonically converge from above to $\rot(b)$, we can pass to subsequences of $\{(\lambda^+_{\ve^+},J^+_{\ve^+}) \}$ and $\{(\lambda^-_{\ve^-},J^-_{\ve^-}) \}$ and re-index whenever necessary to guarantee that our intersection positivity argument as in Proposition \ref{prop:cobknot} applies.  Thus there is a chain map
\begin{equation}\label{eq:cm}
\phi^+_-:ECC^{\afplus}_*(Y,\lambda^+,b, \rot(b), J^+) \to ECC^{\afminus}_*(Y,\lambda^-,b, \rot(b),J^-),
\end{equation}
 which we now want to show is a chain homotopy equivalence.  This means that we need to show that there is a chain map
 \begin{equation}\label{eq:cm-+}
\phi^-_+:ECC^{\afplus}_*(Y,\lambda^-,b, \rot(b),J^-) \to ECC^{\afminus}_*(Y,\lambda^+,b, \rot(b), J^+),
\end{equation}
such that 
\begin{align}
\phi^+_- \circ \phi^-_+ & \ \ \ \mbox{ is chain homotopic to } \id_-, \label{eq:id-}\\
 \phi^-_+ \circ \phi^+_- & \ \ \ \mbox{ is chain homotopic to } \id_+. \label{eq:id+} 
 \end{align}
To show this, we need two collections $\{ \kappa^+ \}$ and $\{ \kappa^- \}$ of chain contractors
\[
\kappa^\pm:  ECC^{\afpm}_*(Y,\lambda^\pm,b, \rot(b),J^\pm) \to ECC^{\afpm}_{*+1}(Y,\lambda^\pm,b, \rot(b), J^\pm)
\] 
such that 
\begin{align}
\phi^+_- \circ \phi^-_+ - \id_-&= \partial \kappa^- + \kappa^- \partial,  \label{eq:id--}\\
 \phi^-_+ \circ \phi^+_- - \id_+&= \partial \kappa^+ + \kappa^+\partial. \label{eq:id++} 
\end{align}

 The desired properties \eqref{eq:id--}
and   \eqref{eq:id++} follow from the Theorem \ref{thm:cobmaps} (Trivial Cobordisms) and (Composition) together with the same intersection positivity argument as before.  We obtain the map \eqref{eq:propcm} by taking direct limits as $L \to \infty$ and $\varepsilon \to 0$ as in Lemma \ref{lem:trick}.

\end{proof}

\noindent \textsc{Jo Nelson \\   Rice University \\}
{\em email: }\texttt{jo.nelson@rice.edu}\\

\noindent \textsc{Morgan Weiler \\  Cornell University}\\
{\em email: }\texttt{morgan.weiler@rice.edu}\\

\end{document}